\documentclass{amsart}

\usepackage[square, numbers, sort&compress]{natbib} 
\usepackage[foot]{amsaddr}
\usepackage{amsmath, amssymb, amsthm, amsfonts, graphicx, mathrsfs, cancel, color, genyoungtabtikz, enumitem, tikz, mathtools, nccmath, needspace, comment, stmaryrd, array, microtype}
\usepackage[a4paper,margin=0.96in]{geometry}
\usepackage[bookmarks=true,colorlinks=true,linktoc=page,citecolor=darkgreen,linkcolor=blue]{hyperref}
\usepackage{cleveref}
\usepackage{eqparbox}

\usetikzlibrary{backgrounds}
\usetikzlibrary{calc,intersections,through,backgrounds}

\newtheorem{lem}{Lemma}[section]
\newtheorem{thm}[lem]{Theorem}
\newtheorem{cor}[lem]{Corollary}
\newtheorem{prop}[lem]{Proposition}
\newtheorem{conj}[lem]{Conjecture}

\newtheoremstyle{lscite}{}{}{\itshape}{}{\bfseries}{}{ }{}
\theoremstyle{lscite}

\newtheorem{thmciting}[lem]{Theorem}

\theoremstyle{definition}
\newtheorem{defn}[lem]{Definition}
\newtheorem{eg}[lem]{Example}

\theoremstyle{remark}

\newtheorem*{rem}{Remark}
\newtheorem*{rmk}{Remarks}

\crefname{defn}{Definition}{Definitions}
\crefname{thm}{Theorem}{Theorems}
\crefname{thmx}{Theorem}{Theorems}
\crefname{prop}{Proposition}{Propositions}
\crefname{lem}{Lemma}{Lemmas}
\crefname{cor}{Corollary}{Corollaries}
\crefname{conj}{Conjecture}{Conjectures}
\crefname{section}{Section}{Sections}
\crefname{subsection}{Subsection}{Subsections}
\crefname{chapter}{Chapter}{Chapters}
\crefname{eg}{Example}{Examples}
\crefname{Example}{Example}{Examples}

\Crefname{defn}{Definition}{Definitions}
\Crefname{thm}{Theorem}{Theorems}
\Crefname{thmx}{Theorem}{Theorems}
\Crefname{prop}{Proposition}{Propositions}
\Crefname{lem}{Lemma}{Lemmas}
\Crefname{cor}{Corollary}{Corollaries}
\Crefname{conj}{Conjecture}{Conjectures}
\Crefname{section}{Section}{Sections}
\Crefname{subsection}{Subsection}{Subsections}
\Crefname{chapter}{Chapter}{Chapters}
\Crefname{eg}{Example}{Examples}
\Crefname{Example}{Example}{Examples}

\crefname{thmciting}{Theorem}{Theorems}
\crefname{propciting}{Proposition}{Propositions}
\crefname{lemciting}{Lemma}{Lemmas}
\crefname{corciting}{Corollary}{Corollaries}
\crefname{conjciting}{Conjecture}{Conjectures}

\Crefname{thmciting}{Theorem}{Theorems}
\Crefname{propciting}{Proposition}{Propositions}
\Crefname{lemciting}{Lemma}{Lemmas}
\Crefname{corciting}{Corollary}{Corollaries}
\Crefname{conjciting}{Conjecture}{Conjectures}

\AddToHook{env/thm/begin}{\crefalias{lem}{thm}}
\AddToHook{env/cor/begin}{\crefalias{lem}{cor}}
\AddToHook{env/prop/begin}{\crefalias{lem}{prop}}
\AddToHook{env/conj/begin}{\crefalias{lem}{conj}}
\AddToHook{env/thmx/begin}{\crefalias{lem}{thmx}}

\AddToHook{env/lemciting/begin}{\crefalias{lem}{lemciting}}
\AddToHook{env/propciting/begin}{\crefalias{lem}{propciting}}
\AddToHook{env/thmciting/begin}{\crefalias{lem}{thmciting}}
\AddToHook{env/corciting/begin}{\crefalias{lem}{corciting}}
\AddToHook{env/conjciting/begin}{\crefalias{lem}{conjciting}}

\AddToHook{env/defn/begin}{\crefalias{lem}{defn}}
\AddToHook{env/Example/begin}{\crefalias{lem}{Example}}
\AddToHook{env/Notation/begin}{\crefalias{lem}{Notation}}
\AddToHook{env/Assumption/begin}{\crefalias{lem}{Assumption}}
\AddToHook{env/Remark/begin}{\crefalias{lem}{Remark}}

\AddToHook{env/Example/begin}{\crefalias{lem}{Example}}
\AddToHook{env/Notation/begin}{\crefalias{lem}{Notation}}
\AddToHook{env/Assumption/begin}{\crefalias{lem}{Assumption}}

\newcount\tableauRow\newcount\tableauCol
\newcommand\TikTableau[3]{
  \begin{tikzpicture}[scale=#1,draw/.append style={thick,black},baseline={([yshift=-.8ex]current bounding box.center)}]
    \tableauRow=0
    \foreach \Row in {#3} {
       \tableauCol=1
       \foreach\k in \Row {
          \draw(\the\tableauCol,\the\tableauRow)+(-.5,-.5)rectangle++(.5,.5);
          \draw(\the\tableauCol,\the\tableauRow)node{\k};
          \global\advance\tableauCol by 1
       }
       \global\advance\tableauRow by -1
    }
  \end{tikzpicture}
}

\newdimen\shadedBaseline\shadedBaseline=-4mm
\newcount\tableauRow\newcount\tableauCol
\newcommand\ShadedTableau[3][\relax]{%
  \begin{tikzpicture}[scale=0.4,draw/.append style={thick,black},baseline=-{#3}mm]
    \ifx\relax#1\relax%
    \else 
      \foreach\bx in {#1} { \filldraw[blue!20]\bx+(-.5,-.5)rectangle++(.5,.5); }
    \fi
    \tableauRow=0
    \foreach \Row in {#2} {
       \tableauCol=1
       \foreach\k in \Row {
          \draw(\the\tableauCol,\the\tableauRow)+(-.5,-.5)rectangle++(.5,.5);
          \draw(\the\tableauCol,\the\tableauRow)node{\k};
          \global\advance\tableauCol by 1
       }
       \global\advance\tableauRow by -1
    }
  \end{tikzpicture}%
}

\def\h{5}
\def\sep{2}
\def\a{5+\sep}
\def\dashpart{0.9}

\newenvironment{braid}[1][10]{
  \begin{tikzpicture}[baseline={([yshift=-.8ex]current bounding box.center)},blue,line width=1pt, xscale=0.4,yscale=0.4,
                      draw/.append style={rounded corners},
                      every node/.append style={font=\fontsize{5}{5}\selectfont}]%
  }{\end{tikzpicture}
}

\colorlet{darkgreen}{green!60!black}
\tikzset{dots/.style={very thick,loosely dotted,color=blue!50},
         greendot/.style={fill,circle,color=darkgreen,inner sep=1.5pt,outer sep=0}
}
\tikzset{dots/.style={very thick,loosely dotted,color=blue!50},
         reddot/.style={fill,circle,color=red,inner sep=1.5pt,outer sep=0}
}

\def\greendot(#1);{\node[greendot] at(#1){};}
\def\reddot(#1);{\node[reddot] at(#1){};}
\newcommand\blam{{\boldsymbol\lambda}}
\newcommand\bmu{{\boldsymbol\mu}}
\newcommand\bnu{{\boldsymbol\nu}}
\newcommand\bsig{{\boldsymbol\sigma}}

\newcommand\sgn{{\mathsf{sgn}}}

\let\<=\langle
\let\>=\rangle
\def\({\big(}
\def\){\big)}

\newcommand\iarrow[1][i]{\xrightarrow{\ #1\ }}
\newcommand\nuarrow[1][\nu]{\iarrow[\nu]}

\def\s{\mathfrak{s}}
\def\t{\mathfrak{t}}
\def\tmu{\t^{\bmu}}
\def\a{\mathfrak{a}}

\def\tlam{\t^\blam}

\def\Z{\mathbb{Z}}

\def\p{\mathfrak{p}}
\def\Multipart{\mathcal{P}}
\def\resi{\pmb{i}}
\def\resj{\pmb{j}}

\def\floor#1/#2{\lfloor\tfrac{#1}{#2}\rfloor}
\def\ceil#1/#2{\lceil\tfrac{#1}{#2}\rceil}

\let\gedom=\trianglerighteq
\let\gdom=\vartriangleright

\let\ledom=\trianglelefteq

\DeclareMathOperator\mmod{mod}

\DeclareMathOperator\Add{Add}
\DeclareMathOperator\Rem{Rem}

\DeclareMathOperator\row{row}
\DeclareMathOperator\col{col}
\DeclareMathOperator\diag{diag}

\DeclareMathOperator\res{res}

\DeclareMathOperator\Hom{Hom}

\DeclareMathOperator\Std{Std}

\DeclareMathOperator\Shape{Shape}

\newcommand{\IniNu}[1][n]{\Std^e_{#1}(\bnu)}

\newcommand{\StdCheck}[1][\bsig]{\Std^\vee_{#1}(\bnu)}
\newcommand{\StdHat}[1][\bsig]{\Std^\wedge_{#1}(\bnu)}

\DeclareMathOperator\Rank{Rank}

\begin{document}

\title{A Generalization of Carter--Payne Homomorphisms}


\author{Martín Forsberg Conde}
\address{Okinawa Institute of Science and Technology}
\email{martin.forsberg@oist.jp}
\begin{abstract}

We construct graded homomorphisms between Specht modules of quiver Hecke algebras of type A that differ by an ``$e$-small'' partition-shaped removable set of nodes by expanding on methods by Lyle and Mathas. Our main result constitutes a full generalization of the classical result by Carter and Payne for Specht modules of the symmetric group.
\end{abstract}

\maketitle

\section{Introduction} \label{S:Intro}

This paper is concerned with the problem of finding homomorphisms between Specht modules of cyclotomic quiver Hecke algebras of type A. We build upon methods developed by Lyle and Mathas~\cite{lm14}.

Carter and Lusztig~\cite{cl74} gave an important family of homomorphisms between Weyl modules of the general linear groups whose respective partitions differ by a single row-strip of nodes under certain conditions in positive characteristic $p$. Through the Schur functor, these maps can be seen as homomorphisms between Specht modules of the symmetric group. This result was later extended by Carter and Payne to the following theorem, remarkable for its generality and simplicity.

\begin{thmciting}[\cite{CP}]\label{T:Carter-Payne}  Suppose $\lambda, \mu$ are partitions of $n$ such that the Young diagram for $\mu$ is obtained from that of $\lambda$ by raising $d$ nodes from row $j$ to row $i$. Suppose that $\lambda_i - \lambda_j +j - i + d$ is divisible by $p(d)$, the smallest power of $p$ such that $p(d) > d$. Then $\Hom_{k\mathfrak{S}_n}(S^\lambda, S^\mu) \neq 0$.
\end{thmciting}

Cyclotomic Hecke algebras of type A generalize the group algebra of the symmetric group. They are known to be Morita equivalent to a direct sum of tensor products of \emph{integral} cyclotomic Hecke algebras of type A~\cite{dm02} . Moreover, all integral cyclotomic Hecke algebras of type A over a field are isomorphic to quiver Hecke algebras according to the celebrated isomorphism theorem of Brundan and Kleshchev~\cite{bkisom}. 

The representation theory of integral cyclotomic Hecke algebras and cyclotomic quiver Hecke algebras depends crucially on a parameter $e$, called the \emph{quantum characteristic} of the algebra. For the group algebra of the symmetric group, this parameter coincides with $p$, the characteristic of the ground field.

Lyle and Mathas~\cite{lm14} gave a partial generalization of \cref{T:Carter-Payne} to cyclotomic quiver Hecke algebras. Furthermore, their methods allowed them to prove that the homomorphisms they find are homogeneous of positive degree. The generalization is only partial because Lyle and Mathas's result requires that the removable strip of nodes has length $d < e$. In contrast, the classical Carter--Payne theorem allows for longer strips to be moved, under stronger divisibility conditions. It should be noted that further generalizations are described in~\cite{lm14}, but we do not discuss them in this text.

A shape $[\xi] \subset [\blam]$ is an orthogonally connected subset of $[\blam]$, the Young diagram of the multipartition $\blam$. The shape $[\xi]$ is said to be removable if $[\blam] \setminus [\xi]$ is the Young diagram of a multipartition. The shape $[\xi]$ is said to be \emph{straight} if its nodes are in the shape of the Young diagram of a partition. It is said to be \emph{skew} if its nodes are in the shape of the Young diagram of a skew partition. See \cref{SS:PartTab} for the background and notation on partitions and tableaux that we use.

In his doctoral thesis, Witty~\cite[Theorem 3.14]{Witty} gave an important generalization of Lyle and Mathas's result in the case of bipartitions with zero quantum characteristic: according to his result, there exist homomorphisms between bipartitions differing by a removable skew shape (instead of a removable row-strip). The techniques in this paper can be used to give a new proof of Witty's result. We comment on this at the end of \cref{S:Straight}.

The main theorems in \cite{CP}, \cite{lm14} and \cite{Witty} hold under very different conditions. This paper aims to unify and substantially generalize \cite{lm14} and \cite{CP}, and in doing so we make significant progress into further unifying these results with \cite[Theorem 3.14]{Witty}. Speaking concretely, a simple result in \cref{S:Degeneration} extends \cite[Theorem 3.12]{lm14} into a full generalization of \cref{T:Carter-Payne}; and \cref{T:MainIntro}, our main result, further unifies \cite[Theorem 3.13]{lm14} and \cref{T:Carter-Payne} with the special case in \cite[Theorem 3.14]{Witty} where the two components differ by a \emph{straight} shape. The author expects that a full unification of the three results holds, see \cref{C:Skew}. We also refine the methods in \cite{lm14} to give an explicit description of the homomorphisms in the main theorem, which send the original Specht module generator to a single standard basis element of the new Specht module.

It is significant that \cref{T:MainIntro} gives many new homomorphisms even in the special case where the quiver Hecke algebra is isomorphic to the symmetric group. Our theorem constitutes the second generalization of \cref{T:Carter-Payne} that produces new homomorphisms between Specht modules of the symmetric group, after \cite{fm04}. The new homomorphisms given in this paper never coincide with those given in \cite{fm04}. Before conjecturing a unification of \cref{T:MainIntro} with the main result of \cite{fm04}, a description of the latter in the quiver Hecke algebra setting would be desirable.

Next we quickly skim over the definitions we need to state the main theorem. We explain the necessary background more leisurely in \cref{S:Background}.  Let $e \in \mathbb{Z}_{\geq 2} \cup \{\infty\}$. An \emph{$\ell$-multicharge} is an $\ell$-tuple of elements of $\mathbb{Z}/e\mathbb{Z}$, denoted by $\kappa = (\kappa_1, \dots, \kappa_\ell)$. Given $e$, $\kappa$ and an $\ell$-multipartition $\blam \vdash_\ell n$, the \emph{residue} of a node $N = (r, c, m) \in [\blam]$ is $\res(N) = c - r + \kappa_m \mod e$. A shape $[\xi]$ is said to be \emph{$e$-small} if its nodes do not contain the full set $\mathbb{Z}/e\mathbb{Z}$ as residues. We order the nodes of $[\blam]$ in lexicographic order according to their component, row and column, in that order of precedence. 

\begin{thm} \label{T:MainIntro}
Let $\blam$ and $\bmu$ be multipartitions of $n$ such that the Young diagram of $\bmu$ can be constructed from $\blam$ by removing a removable $e$-small straight shape and placing it at an earlier position in the multipartition so that the residues match. Then there is a homogeneous homomorphism from $S^\blam$ to $S^\bmu$ of positive degree which maps the Specht module generator of $S^\blam$ to a single standard basis element of $S^\bmu$.
\end{thm}

The standard basis of a Specht module depends on a certain choice of reduced words for permutations. The theorem above holds for a specific choice of reduced words, which we describe in \cref{S:Background}, along with the necessary background on partitions, cyclotomic Hecke algebras and cyclotomic quiver Hecke algebras.

In \cref{S:Stubborn}, we introduce accessible nodes and generalize stubborn strings as introduced in~\cite{lm14}. The tools introduced in this section are essential to the proof of the main theorem.

In \cref{S:Degeneration}, we use the technique of change of rings to reduce homomorphisms between Specht modules for integral cyclotomic Hecke algebras with $e = r^a, p = 0$, where $r \in \mathbb{N}$ is prime,  to homomorphisms between Specht modules for integral cyclotomic Hecke algebras with $e = p = r$.

In \cref{S:Rows} we go over the result by Lyle and Mathas where the partitions differ by a single ($e$-small) row-strip. The proof given here differs slightly from that in~\cite{lm14} so that it more closely resembles the proof of the more general \cref{T:MainStraight}. We hope that the inclusion of this section will improve the readability of the text.

In \cref{S:Straight} we state the Main \cref{T:MainStraight}, providing information about the degree of the maps as well as an explicit description of the homomorphisms. Most of the section is dedicated to the proof of this theorem. We finish with a conjecture about Carter--Payne homomorphisms between Specht modules whose multipartitions differ by a skew shape.
\subsection*{Acknowledgements} I am grateful to Liron Speyer for suggesting work on Carter--Payne homomorphisms, for his helpful advice and discussions, and for correcting the draft.

\section{Background} \label{S:Background}

\subsection{Partitions and tableaux} \label{SS:PartTab}

A \emph{partition} $\lambda$ of $n$ is a weakly decreasing sequence of natural numbers $(\lambda_1, \lambda_2, \dots, \lambda_k)$ which sum to $n$. We write $|\lambda| = \sum_{j=1}^k \lambda_j = n$. When useful, we take partitions to be an infinite sequence where $\lambda_k$ is followed by infinitely many trailing zeros. An \emph{$\ell$-multipartition} of $n$ is a sequence of partitions $\blam = (\lambda^1, \dots, \lambda^\ell)$ such that $\sum_{j = 1}^\ell |\lambda^j| = n$. Each $\lambda^j$ inside a multipartition is called a \emph{component}. We denote multipartitions with boldface greek letters, and $\blam \vdash_\ell n$ denotes that $\blam$ is an $\ell$-multipartition of $n$. The set of $\ell$-multipartitions of $n$ is denoted by $\Multipart^\ell_n$. We sometimes write simply $\blam \vdash n$. Given a multipartition $\blam$, its \emph{Young diagram} is the set $[\blam] = \{(r, c, m) | c \leq \lambda^m_r\}$. Young diagrams are to be thought of graphically: for instance, the Young diagram for $\blam = \left((5, 3), (2), (3, 3)\right)$ is given by
$$
\left( \ShadedTableau{{,,,,},{,,}}{4}, \ShadedTableau{{,}}{4}, \ShadedTableau{{,,},{,,}}{4}\right).
$$
The \emph{dominance order} on multipartitions is denoted by $\trianglelefteq$. That is, we write $\blam \trianglelefteq \bmu$ if, for all $m^* \in \{1, \dots, \ell\}$ and $r^* \in \mathbb{N}$, we have $\sum_{i = 1}^{m^*-1}|\lambda^i| + \sum_{j=1}^{r^*} |\lambda^{m^*}_j| \leq \sum_{i = 1}^{m^*-1}|\mu^i| + \sum_{j=1}^{r^*} |\mu^{m^*}_j|$.

If a node in a Young diagram can be removed and the remaining nodes still form the Young diagram for some multipartition, the node is said to be \emph{removable}. Similarly, if a node can be added to a Young diagram to form the Young diagram for a new multipartition, the node is said to be \emph{addable}. 

A $\blam$-tableau $\t$ is a bijection $\t {:} [\blam] \rightarrow \{1, \dots, n\}$. Tableaux are usually depicted by writing the associated numbers inside each node of the Young diagram. A $\blam$-tableau is \emph{standard} if its entries increase along the rows and down the columns within each component. We denote the set of standard $\blam$-tableaux by $\Std(\blam)$. Given $\blam \vdash n$ and a standard $\blam$-tableau
 $\t$, the restriction $\t_{\downarrow m}$ of $\t$, for $m < n$, is the tableau given by removing all nodes with numbers $m+1, \dots, n$ from the diagram. We define a partial order on $\blam$-tableaux that we call \emph{dominance order}, obtained by declaring $\t \succeq \s$ if $\Shape(\t_{\downarrow m}) \gedom \Shape(\s_{\downarrow m})$ for all $m \in \{1, \dots, n\}$, where $\Shape(\t)$ is the multipartition such that $\t$ is a $\Shape(\t)$-tableau. We define the initial tableau $\t^\blam$ to be the unique maximal tableau under dominance order. The initial tableau $\t^\blam$ gives a natural order to the nodes of $\blam$: if $N_1, N_2 \in \blam$, then we write $N_1 > N_2$ if $\t^\blam(N_1) > \t^\blam(N_2)$. 

Next, let $e \in \mathbb{Z}_{\geq 2} \cup {\infty}$, and let $\kappa = (\kappa_1, \dots, \kappa_\ell)$ be an $\ell$-tuple of elements of $\mathbb{Z}/e\mathbb{Z}$. Given $e$, $\kappa$ and $\blam \vdash_\ell n$, the \emph{residue} of a node $N = (r, c, m) \in [\blam]$ is defined to be $\res(N) = c - r + \kappa_m \pmod e$. Given a residue $i \in \mathbb{Z}/e\mathbb{Z}$, the node $N$ is said to be an $i$-node if $\res(N) = i$. For instance, for $\blam = ((5, 3, 2), \varnothing, (3, 1))$, $e = 3$ and $\kappa = (0, 0, 1)$, the residues associated to the Young diagram are given by 

\begin{equation*}
    \left(\,\ShadedTableau{{0,1,2,0,1},{2,0,1},{1,2}}{4}, \quad \varnothing, \quad\ShadedTableau{{1,2,0},{0}}{4}\,\right).
\end{equation*}

Let $\blam \vdash n$, and let $N \in \blam$ be a removable node. We define
\begin{multline*}
d_N(\blam) = \#\{\text{addable $i$-nodes $A \in [\blam]$ such that $A > N$}\} \\
- \#\{\text{removable $i$-nodes $A \in [\blam]$ such that $A > N$}\}.
\end{multline*}

Next we define the \emph{degree} of a tableau $\t \in \Std(\blam)$ recursively: let $N = \t^{-1}(n)$ be the node occupied by $n$ in $\t$. Then we have
\begin{equation} \label{E:Degree}
\deg(\t) = \deg(\t_{\downarrow (n-1)})  + d_N(\blam).
\end{equation}
This degree function is connected to the grading on $S^\blam$.

Let $\blam$ be a multipartition of $n$. A \emph{shape} $[\xi] \subset [\blam]$ is an orthogonally connected subset of $[\blam]$. A shape $[\xi]$ is said to be \emph{$e$-small} if its nodes do not contain the full set $\mathbb{Z}/e\mathbb{Z}$ as residues. Moreover, a shape $[\xi]$ is said to be \emph{straight} if its nodes are in the shape of the Young diagram of a partition, and it is said to be \emph{skew} if its nodes are in the shape of the Young diagram of a skew partition. A shape $[\xi]$ is said to be \emph{removable} if one can reach $[\blam] \setminus [\xi]$ from $[\blam]$ by successively substracting removable nodes. Finally, let $[\xi]$ and $[\rho]$ be straight shapes, and let $\xi$ and $\rho$ be the partitions associated with the shapes $[\xi]$ and $[\rho]$. If $\xi_j \leq \rho_j$ for all $j \in \mathbb{N}$ and the residue of the top-left node of $\xi$ matches the residue of the top-left node of $\rho$, then we say that $[\xi]$ is a \emph{subshape} of $[\rho]$.

\begin{eg}
Let $\blam = ((5, 3, 2), \varnothing, (3, 1))$, $e = 3$ and $\kappa = (0, 0, 1)$. Let $[\xi] \subset [\blam]$ be the set of shaded nodes inside the first component and let $[\rho]$ be the set of shaded nodes inside the third component. Then both are straight shapes, and $[\rho]$ is a subshape of $[\xi]$. The shape $[\rho]$ is $e$-small, while $[\xi]$ is not.
\begin{equation*}
    \left(\,\ShadedTableau[(3, 0), (3, -1), (4, 0), (5, 0)]{{0,1,2,0,1},{2,0,1},{1,2}}{4}, \quad \varnothing, \quad\ShadedTableau[(2, 0), (3, 0)]{{1,2,0},{0}}{4}\,\right)
\end{equation*}
\end{eg}

If a node $N \in \blam$ satisfies $N < N'$ for all $N' \in [\xi] \subset \blam$ we write $N \prec [\xi]$. 
If $N$ satisfies $N \geq N'$ for some $N' \in [\xi]$, then we write $N \nprec [\xi]$. Further, given two shapes $[\xi]$ and $[\rho]$, we may write $[\xi] \prec [\rho]$ if $N \prec [\rho]$ for each $N \in [\xi]$. 

Given a straight shape or a Young diagram $[\xi]$, a \emph{hook} of $[\xi]$ is a $\Gamma$-shaped strip of nodes whose north-eastern and south-western tips are orthogonally connected to the right and bottom boundaries of $[\xi]$ respectively, while a \emph{rim hook} of $[\xi]$ is a removable shape with at most one node in each diagonal. We picture a hook and a rim hook for the partition $(5,4,4,3) \vdash 16$ below.

$$
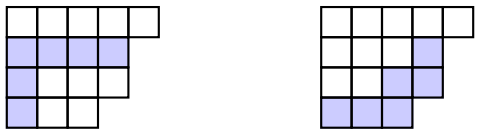
\ShadedTableau[(1,-1),(1,-2),(1,-3),(2,-1),(3,-1),(4,-1)]{{,,,,},{,,,},{,,,}, {,,}}{4} \qquad \qquad \qquad \ShadedTableau[((1,-3),(2,-3),(3,-3),(3,-2),(4,-2),(4,-1)]{{,,,,},{,,,},{,,,}, {,,}}{4}
$$

\subsection{Cyclotomic Hecke algebras} \label{SS:CHA}

From now on we fix a field $k$ of characteristic $p$, possibly zero. Unless stated otherwise, all algebras throughout this text are $k$-algebras. Cyclotomic Hecke algebras are Hecke algebras for the wreath product $C_\ell \wr \mathfrak{S}_n$, where $C_\ell$ is the cyclic group with $\ell$ elements. This is the complex reflection group of type $G(\ell, 1, n)$. In particular, if $\ell = 1$, the cyclotomic Hecke algebra is the Hecke algebra of $\mathfrak{S}_n$, and for $\ell = 2$, it is the Hecke algebra of $C_2 \wr \mathfrak{S}_n$, which is the Coxeter group of type B or C.

Next we give the definition of cyclotomic Hecke algebras with generators and relations, following~\cite{m14surv}. Let $v \in k \setminus \{0\}$ and $Q_1, \dots, Q_\ell \in k$. The \emph{cyclotomic Hecke algebra} $\mathcal{H}_n^\ell(v; Q_1, \dots, Q_\ell)$ has generators $L_1, \dots, L_n, T_1, \dots, T_{n-1}$ and relations
\begin{align*} &\prod_{i=1}^\ell (L_1 - Q_i) = 0,  &(T_r + v^{-1})(T_r - v) = 0, \\
			&L_{r+1} = T_rL_rT_r + L_r,  &T_rT_s = T_sT_r \text{ if } |r-s| \geq 2, \\
			&T_sT_{s+1}T_s = T_{s+1}T_sT_{s+1}, & T_rL_t = L_tT_r \text{ if } t \neq r, r+1.
\end{align*}

Given $v \in k \setminus \{0\}$ and $m \in \mathbb{Z}_{\geq 0}$, the \emph{quantum integer} $[m]_v$ is defined as $[m]_v = v+v^3 + \dots + v^{2m - 1}$. Let $e \in \mathbb{Z}_{\geq 2} \cup \{\infty\}$. If $e$ is finite, let $\zeta \in k$ be a primitive $e$-th root of unity, and otherwise let $\zeta \in k$ be a non-root of unity.  Let $\kappa \in \left(\mathbb{Z}/e\mathbb{Z}\right)^\ell$ and define $\Lambda(\kappa) = \Lambda_{\kappa_1} + \dots + \Lambda_{\kappa_\ell}$, where $\Lambda_i$ are the fundamental weights as defined in~\cite{kac}. We also make use of the normalized invariant bilinear form $(\cdot, \cdot)$ satisfying
$$(\alpha_i, \alpha_j) = a_{ij}, \qquad (\Lambda_i, \alpha_j) = \delta_{ij}. \qquad (i, j \in I)$$
The \emph{integral} cyclotomic Hecke algebra $\mathcal{H}_{n,e}^\Lambda$ is the cyclotomic Hecke algebra with parameters $v = \zeta$ and $Q_j = [\kappa_j]_\zeta$ for each $j \in \{1, \dots, \ell \}$. In \cref{S:Degeneration}, we will use the notation $\mathcal{H}_{n,\zeta}^\Lambda(R)$, where $\zeta$ is a root of unity, to mean the integral cyclotomic Hecke algebra with $v = \zeta$ and $Q_i =[\kappa_j]_\zeta$ over a commutative ring $R$. Dipper and Mathas~\cite{dm02} showed that all cyclotomic Hecke algebras are Morita equivalent to a direct sum of tensor products of integral cyclotomic Hecke algebras.

\subsection{Cyclotomic quiver Hecke algebras}

Let $e \in \mathbb{Z}_{\geq 2} \cup {\infty}$. Given a quiver of affine type $A_{e-1}^{(1)}$, we assign the vertices of our quiver labels $0, \dots, e-1$ so that there is an edge $i \rightarrow j$ if and only if $j = i-1 \pmod e$. If the quiver is of type $A_\infty^{(0)}$, we assign labels in $\mathbb{Z}$ to each vertex so that there is an edge $i \rightarrow j$ if and only if $j = i-1$. We refer to the set of vertices in this quiver as $I$. Through the vertex labels, $I$ is identified with $\mathbb{Z}/e\mathbb{Z}$ if $e \neq \infty$, or with $\mathbb{Z}$ if $e = \infty$. The \emph{quiver Hecke algebra} $R_{n,e}$, or simply $R_n$, associated to the quiver above is defined with generators and relations. The description in algebraic notation can be found in many sources, see for instance \cite{lm14}. We will give the generators and relations in diagrammatic notation, as we will be working mainly with diagrams. The following are the generators of $R_{n}$.

\begin{equation}\label{E:generators}
\begin{gathered} 
\def\h{3}
\psi_r = \begin{braid}
	\draw (0,0) node[anchor=north]{$1$}-- (0,\h);
	\draw (1,0) node[anchor=north]{$2$}-- (1,\h);
	\draw (1.5+\sep/2, \h/2) node{$\cdots$};
	\draw (2+\sep, 0) node[anchor=north]{$r$}-- (3 + \sep, \h);
	\draw (3 + \sep, 0) -- (2 + \sep, \h);
	\draw (3+\sep,-0.42) node{$r{+}1$};
	\draw (3.5+\sep + \sep/2, \h/2) node{$\cdots$};
	\draw (4 + 2*\sep, 0)  -- (4 + 2*\sep,\h);
	\draw (4 + 2*\sep, -0.46) node{$n{-}1$};
	\draw (5+2*\sep, 0) node[anchor=north]{$n$}-- (5 + 2*\sep, \h);
\end{braid}, \qquad
y_s = \begin{braid}
	\draw (0,0) node[anchor=north]{$1$}-- (0,\h);
	\draw (1,0) node[anchor=north]{$2$}-- (1,\h);
	\draw (1.5+\sep/2, \h/2) node{$\cdots$};
	\draw (2+\sep, 0) node[anchor=north]{$s$}-- (2 + \sep, \h);
	\greendot(2+\sep, \h/2);
	\draw (2.5+\sep + \sep/2, \h/2) node{$\cdots$};
	\draw (3 + 2*\sep, 0)  -- (3 + 2*\sep,\h);
	\draw (3+2*\sep,-0.46) node{$n{-}1$};
	\draw (4+2*\sep, 0) node[anchor=north]{$n$}-- (4 + 2*\sep, \h);
\end{braid}, \\
\def\h{3}
e(\resi) = \begin{braid}
	\draw (0,0) node[anchor=north]{$1$}-- (0,\h) node[anchor= south]{$i_{1}$};
	\draw (1,0) node[anchor=north]{$2$}-- (1,\h)node[anchor= south]{$i_{2}$};
	\draw (1.5+\sep/2, \h/2) node{$\cdots$};
	\draw (2 +\sep, 0)  -- (2 + \sep,\h);
	\draw (2+\sep, \h+0.44) node{$i_{n{-}1}$};
	\draw (2+\sep, -0.46)node{$n{-}1$};
	\draw (3+\sep, 0) node[anchor=north]{$n$}-- (3 + \sep, \h)node[anchor= south]{$i_{n}$};
\end{braid},
\end{gathered}
\end{equation}
where $1 \leq r < n$, $1 \leq s \leq n$ and $\pmb{i} = (i_1, \dots, i_n) \in I^n$. The relations on this algebra include some natural isotopy relations (for instance, $\psi_i \psi_j = \psi_j \psi_i$ if $|i - j| \geq 2$, see \cite{lm14}) which we do not describe in detail, along with the relations below. The only  new results for $e=2$ in this text appear in \cref{S:Degeneration}, where we do not use the quiver Hecke algebra presentation, so we assume that $e > 2$ in the following relations for the sake of simplicity.

\begin{equation} \label{E:PsiSquared}
        \begin{braid}
            \draw (0,0) -- (1, \h/2) -- (0, \h) node[anchor=south]{$i$};
            \draw (1,0) -- (0, \h/2) -- (1, \h) node[anchor=south]{$j$};
        \end{braid}
      =  
      \def\h{5/3}
      \begin{cases}
          \eqmakebox[lhs][c]{$0$} &\text{if } i = j,\\[10pt]
          \eqmakebox[lhs][c]{$\begin{braid}
            	\draw (0,0) -- (0, \h) node[anchor=south]{$i$};
            	\draw (1,0) -- (1, \h) node[anchor=south]{$j$};
            	\greendot(1,\h/2);
                \end{braid} -
          \begin{braid}
            	\draw (0,0) -- (0, \h) node[anchor=south]{$i$};
            	\draw (1,0) -- (1, \h) node[anchor=south]{$j$};
            	\greendot(0,\h/2);
                \end{braid}$} &\text{if  }i = j + 1,\\[20pt]
          \eqmakebox[lhs][c]{$\begin{braid}
            	\draw (0,0) -- (0, \h) node[anchor=south]{$i$};
            	\draw (1,0) -- (1, \h) node[anchor=south]{$j$};
            	\greendot(0,\h/2);
                \end{braid} - 
                \begin{braid}
            	\draw (0,0) -- (0, \h) node[anchor=south]{$i$};
            	\draw (1,0) -- (1, \h) node[anchor=south]{$j$};
            	\greendot(1,\h/2);
                \end{braid}$} &\text{if }i = j - 1,\\[20pt]
          \eqmakebox[lhs][c]{$\begin{braid}
            	\draw (0,0) -- (0, \h) node[anchor=south]{$i$};
            	\draw (1,0) -- (1, \h) node[anchor=south]{$j$};
                \end{braid}$} &\text{otherwise.}\\
        \end{cases}
\end{equation}
\begin{equation} \label{E:Braid}
    \begin{braid}
    	\draw (0,0) -- (2, \h) node[anchor=south]{$k$};
    	\draw (2,0) -- (0,\h) node[anchor=south]{$i$};
    	\draw (1, 0) -- (0, \h/2) -- (1, \h)node[anchor=south]{$j$};
\end{braid}
  = 
 \def\h{3}
\begin{cases}
        \eqmakebox[lhs][c]{$\begin{braid}
        	\draw (0,0) -- (2, \h) node[anchor=south]{$k$};
        	\draw (2,0) -- (0,\h) node[anchor=south]{$i$};
        	\draw (1, 0) -- (2, \h/2) -- (1, \h)node[anchor=south]{$j$};
        \end{braid} + 
        \begin{braid}
        	\draw (0,0) -- (0, \h) node[anchor=south]{$i$};
        	\draw (2,0) -- (2,\h) node[anchor=south]{$k$};
        	\draw (1, 0) -- (1, \h)node[anchor=south]{$j$};
        \end{braid}$} 
        & \text{if }i = k = j + 1,\\[20pt]
        \eqmakebox[lhs][c]{$\begin{braid}
        	\draw (0,0) -- (2, \h) node[anchor=south]{$k$};
        	\draw (2,0) -- (0,\h) node[anchor=south]{$i$};
        	\draw (1, 0) -- (2, \h/2) -- (1, \h)node[anchor=south]{$j$};
        \end{braid} - 
        \begin{braid}
        	\draw (0,0) -- (0, \h) node[anchor=south]{$i$};
        	\draw (2,0) -- (2,\h) node[anchor=south]{$k$};
        	\draw (1, 0) -- (1, \h)node[anchor=south]{$j$};
        \end{braid}$} &\text{if } i = k = j - 1, \\[20pt]
        \eqmakebox[lhs][c]{$\begin{braid}
        	\draw (0,0) -- (2, \h) node[anchor=south]{$k$};
        	\draw (2,0) -- (0,\h) node[anchor=south]{$i$};
        	\draw (1, 0) -- (2, \h/2) -- (1, \h)node[anchor=south]{$j$};
        \end{braid}$} & \text{otherwise.}
\end{cases}
\end{equation}
\begin{equation} \label{E:DotSlide}
    \def\h{3}
    \begin{braid}
    	\draw (0,0) -- (2, \h) coordinate[pos = 0.25] (B) node[anchor=south]{$j$};
    	\draw (2,0) -- (0,\h) node[anchor=south]{$i$};
    	\greendot(B);
    \end{braid}
     = 
     \def\h{3}
     \begin{braid}
    	\draw (0,0) -- (2, \h) coordinate[pos = 0.75] (B) node[anchor=south]{$j$};
    	\draw (2,0) -- (0,\h) node[anchor=south]{$i$};
    	\greendot(B);
    \end{braid}
    -\delta_{i,j}
    \begin{braid}
    	\draw (0,0) -- (0, \h) node[anchor=south]{$i$};
    	\draw (1,0) -- (1,\h) node[anchor=south]{$j$};
    \end{braid}
    \qquad \text{and} \qquad
    \def\h{3}
    \begin{braid}
    	\draw (0,0) -- (2, \h)  node[anchor=south]{$j$};
    	\draw (2,0) -- (0,\h) coordinate[pos = 0.25] (B) node[anchor=south]{$i$};
    	\greendot(B);
    \end{braid}
    = 
    \def\h{3}
    \begin{braid}
    	\draw (0,0) -- (2, \h)  node[anchor=south]{$j$};
    	\draw (2,0) -- (0,\h) coordinate[pos = 0.75] (B) node[anchor=south]{$i$};
    	\greendot(B);
    \end{braid}+\delta_{i,j}
    \begin{braid}
    	\draw (0,0) -- (0, \h) node[anchor=south]{$i$};
    	\draw (1,0) -- (1,\h) node[anchor=south]{$j$};
    \end{braid}.
\end{equation} 
Finally, the $e(\resi)$ form a set of mutually orthogonal idempotents.
\begin{equation} \label{E:idempotentrelations}
    e(\resi)e(\resj)= \delta_{\resi, \resj} e(\resi), \qquad \sum_{\resi \in I^n} e(\resi) = 1.
\end{equation}
While the sum above is not finite if $e = \infty$, it does become finite in the cyclotomic quotient described below, as $e(\resi) = 0$ for all but finitely many values under this quotient.

Quiver Hecke algebras can be given a grading as follows.
\begin{equation} \label{E:grading}
    \deg e(\resi) = 0, \qquad \deg y_r = 2, \qquad \deg \psi_re(\resi)=-a_{i_r, i_{r+1}}.
\end{equation}

Considering $\mathfrak{S}_n$ as a Coxeter group, let $w = s_{i_1}s_{i_2}\dots s_{i_m}$ be a word in $\mathfrak{S}_n$. That is, an element of the free group generated by the same Coxeter generators as the symmetric group. We use the notation $\psi_w = \psi_{i_1} \psi_{i_2} \dots \psi_{i_m}$.

To take a \emph{cyclotomic quotient} we first fix a positive integer $\ell$ (the level) and $\kappa = (\kappa_1, \dots, \kappa_{\ell}) \in I^\ell$. We then have $(\Lambda(\kappa), \alpha_i) = \#\{\kappa_j \mid \kappa_j = i\}$. We denote by $R_{n,e}^\Lambda$ the cyclotomic quotient of $R_{n,e}$ with dominant weight $\Lambda = \Lambda(\kappa)$. It is the quotient of $R_{n,e}$ by the relations $y_1^{(\Lambda, \alpha_{i_1})}e(\resi) = 0$ for all $\pmb{i} \in I^n$.

It was shown by Brundan and Kleshchev~\cite{bkisom} that, over a field, all integral cyclotomic Hecke algebras of type A are isomorphic to a cyclotomic quiver Hecke algebra. Since quiver Hecke algebras are graded, this isomorphism endows cyclotomic Hecke algebras with a non-trivial graded structure. The parameter $e$ is defined differently for the two algebras, but it is preserved under the isomorphism. It is called the \emph{quantum characteristic}. Brundan and Kleshchev's result justifies the notation $\mathcal{H}_{n,e}^\Lambda$ for cyclotomic Hecke algebras, as it proves that all cyclotomic Hecke algebras over a field with parameters $n, e$ and $\Lambda$ are all isomorphic to the same cyclotomic quiver Hecke algebra $R_{n, e}^\Lambda$, and hence they are all isomorphic to each other.

Given a single diagram $u \in R_{n,e}^\Lambda$, we assign positive integers to each string in increasing order from left to right at the bottom of the diagram as in~\cref{E:generators}, so that every string is labeled by a number in $\{1, \dots, n\}$. When we refer to string $s$ we mean the string labeled by the positive integer $s$. This labelling provides the set of strings with a total order, as strings can be compared through their integer labels.
\subsection{Specht modules} \label{SS:Specht}
Graded lifts of Specht modules for $\mathcal{H}_{n,e}^\Lambda$ were constructed by Brundan, Kleshchev and Wang~\cite{bkw11}. Kleshchev, Mathas and Ram~\cite{kmr} gave a description of Specht modules in terms of a single generator and homogeneous relations.

The Specht module $S^\blam$ is generated by $v_{\t^\blam} \in S^\blam$. The module $S^\blam$ is drawn by placing $v_{\t^\blam}$ at the top of an $R_{n,e}^\Lambda$-diagram. We associate the top of the diagram with the nodes of the Young diagram of $\blam$ in order, according to the initial tableau $\t^\blam$. Each string is given a residue according to the residue of the node it reaches. An \emph{$i$-string} is a string reaching a node of residue $i$.

For each tableau $\t$, there is a unique permutation $\sigma_\t \in \mathfrak{S}_n$ such that $\t = \t^\blam  \sigma_\t$. For each $\t \in \Std(\blam)$, we fix a reduced word $w_\t$ for the permutation $\sigma_\t$. This choice can be made in any number of ways. Given a tableau $\t \in \Std(\blam)$, a \emph{$321$-pattern} is any triple of natural numbers $i < j < k$ such that $\t^{-1}(i) > \t^{-1}(j) > \t^{-1}(k)$. With an eye towards \cref{T:MainStraight}, we choose $w_\t$ so that for each $321$-pattern in $\t$ where $\t^{-1}(i), \t^{-1}(j), \t^{-1}(k)$ have residues $i, i \pm 1, i$ respectively, the diagram for $w_\t$ looks as follows when ignoring all other strings.

$$
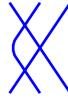
\begin{braid}
	\def\h{3}
    	\draw (0,0) -- (2, \h);
    	\draw (2,0) -- (0,\h);
    	\draw (1, 0) -- (0, \h/2) -- (1, \h);
\end{braid}
$$

For $\t \in \Std(\blam)$, we let $v_\t =  v_{\t^\blam} \psi_{w_\t}$. The elements $v_\t$ are well-defined up to a choice of reduced words $w_\t$. In fact, it is easy to see that the choice we have made is enough to determine $v_\t$. In the diagrams for Specht modules, we will draw brown boxes at the top of the diagram around nodes which are in the same row of $\blam$. According to \cite{kmr}, we need to quotient by the simple Specht relations
\begin{equation}\label{E:Specht}
    \def\h{3}
    \begin{braid}
    	\draw (0,0) -- (1, \h);
    	\draw (1,0) -- (0,\h);
        \draw[brown] (-1.7, \h) rectangle (2.7, \h + 1);
    \end{braid}
    = \quad 0 \qquad \text{and} \qquad
    \def\h{3}
    \begin{braid}
    	\draw (0,0) -- (0, \h) coordinate[pos=0.5] (C);
        \draw[brown] (-1.7, \h) rectangle (1.7, \h+1);
        \greendot(C);
    \end{braid}
    = \quad 0.
\end{equation}
Define $\resi^\blam$ to be the sequence of residues $\left(\res \left((\t^\blam)^{-1}(1)\right), \dots, \res \left((\t^\blam)^{-1}(n)\right) \right) \in I^n$. We also need to quotient by the relation
\begin{equation} \label{E:ResidueRelation}
v_{\t^\blam} e(\resi) = \delta_{\resi, \resi^\blam} v_{\t^\blam}
\end{equation}
and by the homogeneous Garnir relations, which are in general more complicated and significantly less intuitive in this graphical notation. However, in this text we will only need the simplest case, where 
\begin{equation} \label{E:Garnir}
\begin{braid}
	\def\a{7.5};
	\def\dashpart{0.9}
	
	\draw(0, 0) -- (0, \h) ;
	\draw (2, 0) node{$\cdots$};
	\draw(4, 0) -- (4, \h) ;
	\draw (2.5, 0);
	
	\draw (5, 0) -- (0.5+\a, \h) ;
	\draw (6.5, 0) node{$\cdots$};
	\draw (\a + 0.5, 0) -- (3.5 + \a, \h) ;

	\draw (1.5 + \a, 0) -- (5, \h);
	\draw (2.5+\a,0) node{$\cdots$};
	\draw (3.5 + \a, 0) -- (7, \h);
	
	\draw(\a + 4.5, 0) -- (\a+4.5, \h) ;
	\draw (\a + 5.5, 0) node{$\cdots$};
	\draw(\a + 6.5, 0) -- (\a + 6.5, \h) ;
	
	\draw[brown] (-0.5 , \h) rectangle (\a - 0.1, \h + 1);
    \draw[brown] (\a+0.1 , \h) rectangle (2*\a - 0.1, \h + 1);
	
	\draw (\a/2, \h+1) node[anchor=south]{$\lambda^{(m)}_r$};
	\draw (\a + 3.5, \h+1) node[anchor=south]{$\lambda^{(m)}_{r+1}$};
	\draw(\a+2.5,-0.4)node{$\underbrace{\hspace*{10mm}}_{<e}$};
    	\draw(8,-1.4)node{$\underbrace{\hspace*{25mm}}_{\lambda^{(m)}_r+1}$};

\end{braid} \quad = \quad 0.
\end{equation}

The Specht module $S^\blam$ has a homogeneous basis given by $\{v_\t \mid \t \in \Std(\blam)\}$ \cite[Section 3.4]{kmr}, where $\deg(v_\t) = \deg(\t)$. We refer to elements $v \in S^\blam$ which are given by acting on the Specht generator $v_{\t^\blam}$ by a monomial in elements $\psi_1, \dots, \psi_{n-1}, y_1, \dots, y_n$ simply as \emph{monomials}. Given a monomial $v \in S^\blam$, we write $\sigma_v$ for the permutation that the strings in $v$ form, and we let $\t_v = {\t^\blam} \sigma_v$. Suppose $v$ is given as $v_{\t^\blam}\psi_{w}$, where $w$ is a reduced word for a permutation $\sigma_w$ such that $\t^\blam \sigma_w$ is standard. Then we say that $v \in S^\blam$ is a \emph{standard monomial}. In particular, the standard basis given by $\{v_\t \mid \t \in \Std(\blam)\}$ is composed of standard monomials.

Sometimes it is convenient to apply Specht relations only for some of the nodes of a diagram. For instance, consider the following diagram.
$$\begin{braid}
	\def\h{12};
	\def\sep{2};
	\def\a{5+\sep};
	\def\dashpart{0.9};
	\draw (0,0) -- (0, 2);
	\draw (2, 0) node{$\dots$};
	\draw (4,0) -- (4,2);
	\draw[black, rounded corners = false] (-0.3,2) rectangle (4.3, 3.3);
	\draw[black] (2, 2.6) node[font = \fontsize{8}{10}\selectfont]{$B$};
	\draw (4, 3.3) -- (4, 9);
	\draw (2.5, 3.3) -- (2.5, 9);
	
	\draw( 0, 3.3) -- (0, 5.3);
	\draw(1.5, 3.3) -- (1.5, 5.3);
	\draw( 0, 6.6) -- (0, 9);
	\draw(1.5, 6.6) -- (1.5, 9);
	
	\draw (-2.5, 0) -- (-2.5, 5.3);
	\draw (-1, 0) -- (-1, 5.3);
	\draw(-2.5, 6.6) -- (-2.5, 9);
	\draw(-1, 6.6) -- (-1, 9);

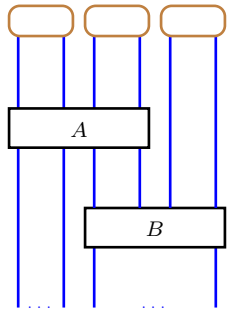
	\draw (-1.75, 0) node{$\dots$};
	\draw[black, rounded corners = false] (-2.8,5.3) rectangle (1.8, 6.6);
	\draw[black] (-0.5, 5.9) node[font = \fontsize{8}{10}\selectfont]{$A$};

	\draw[brown] (-2.8, 9) rectangle (-0.7, 10);
	\draw (-1.75, 9.5);
	\draw[brown] (-0.3, 9) rectangle (1.8, 10);
	\draw (0.75, 9.5);
	\draw[brown] (2.2, 9) rectangle (4.3, 10);
	\draw (3.25, 9.5);
\end{braid}$$
Here $A$ and $B$ represent some crossings between and some dots on the involved strings. Sometimes we want to focus on $B$, perform some operations on it by applying the appropriate relations, and keep $A$ arbitrary. Continuously drawing $A$ can be cumbersome, so we want to stop drawing this part of the diagram. Specht relations may be applied only to the strings coming out of $B$ that do not enter $A$. We draw this as a \emph{mixed} Specht module diagram. 
$$\begin{braid}
	\def\h{12};
	\def\sep{2};
	\def\a{5+\sep};
	\def\dashpart{0.9};
	\draw (0,0) -- (0, 2);
	\draw (2, 0) node{$\dots$};
	\draw (4,0) -- (4,2);
	\draw[black, rounded corners = false] (-0.3,2) rectangle (4.3, 3.3);
	\draw[black] (2, 2.6) node[font = \fontsize{8}{10}\selectfont]{$B$};
	\draw (4, 3.3) -- (4, 6);
	\draw (2.5, 3.3) -- (2.5, 6);
	
	\draw( 0, 3.3) -- (0, 6);
	\draw(1.5, 3.3) -- (1.5, 6);

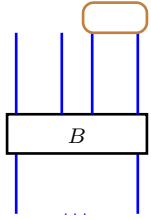
	\draw[brown] (2.2, 6) rectangle (4.3, 7);
\end{braid}$$
We will use mixed diagrams like the above extensively throughout the text.

\section{Stubborn Strings} \label{S:Stubborn}

\subsection{Accessible Nodes}
Stubborn strings were introduced in \cite{lm14}. Lyle and Mathas observed the following: given a monomial $v \in S^\blam$, it is possible to find its decomposition in terms of basis elements by applying quiver Hecke and Specht relations \emph{only in the direction where crossings are broken}. In this sense, no new crossings are created. Crossings are either preserved or undone after the application of a relation, but never created anew. Moreover, $i$ and $j$-strings never switch places, so that an $i$-string can only reach $i$-nodes. We fix a multipartition $\blam$ for the rest of the section. 

When a relation gives a summand where some $(i, i)$-crossing is no longer a crossing, we say that the relation \emph{cuts} that crossing. For instance, \cref{E:DotSlide,E:Braid} cut crossings. We also say that a dot cuts a crossing in \cref{E:DotSlide} and a string cuts a crossing in \cref{E:Braid}. Since $(i,i)$-crossings may only be cut as in
$\begin{braid}
        \draw (0,0) -- (1, \h/3);
        \draw (1,0) -- (0, \h/3);
    \end{braid}
    \rightarrow
    \begin{braid}
        \draw (0,0) -- (0.5, \h/6) -- (0, \h/3);
        \draw (1,0) -- (0.5, \h/6) -- (1, \h/3);
    \end{braid}$ and never as in $    \begin{braid}
        \draw (0,0) -- (1, \h/3);
        \draw (1,0) -- (0, \h/3);
\end{braid}
\rightarrow
\begin{braid}
    \draw (0,\h/3) -- (0.5, \h/6)-- (1, \h/3);
    \draw (0,0) -- (0.5, \h/6)-- (1, 0);
\end{braid}$, we can infer a restriction on the nodes at the top of the diagram which a string may reach after cutting some $(i,i)$-crossings, as follows: start at the bottom of string $s$ in the diagram for $v$ and continuously move upward. At every $(i, i)$-crossing, make a choice to continue on the current string or move on to the other string in the crossing, and continue moving towards the top of the diagram. If we color all possible paths from the bottom of string $s$ to the top of the diagram, we obtain a ``tree diagram''. We call the set of nodes we can reach by this procedure the \emph{accessible nodes} of string $s$ in monomial $v$, and denote this set by $A_v(s)$. Let $Y_{n,e}^\Lambda \subset R_{n,e}^\Lambda$ be the $k$-subalgebra generated by $y_1, y_2, \dots, y_n$. Let $\mathcal{V}_v$ be the set of monomials which appear as summands inside $v Y_{n,e}^\Lambda$ by applying relations without creating any new crossings. In particular, for a monomial $v' \in S^\blam$ for which string $s$ reaches a node $N \notin A_v(s)$, we have $v' \notin \mathcal{V}_v$. Finally, $\mathcal{W}_v \subset \mathcal{V}_v$ is the subset of standard monomials inside $\mathcal{V}_v$. We often omit the subscript in $A_v(s), \mathcal{V}_v$ and $\mathcal{W}_v$ when it is clear from the context. 

\begin{defn}
    We say a string $s$ in a monomial $v \in S^\blam$ is \emph{stubborn} if it reaches the smallest node in $A(s)$. We say a string $s$ is \emph{co-stubborn} if it reaches the greatest node in $A(s)$. 
\end{defn}

Note that the definition above describes a weaker condition than the definition of stubborn strings in~\cite{lm14}. 

\begin{eg} Let $\blam = ((1), (1), (1), (1), (1))$ and $\kappa = (0, 0, 0, 0, 0)$. In the following diagram for an element $v \in S^{\blam}$, we have labeled the nodes of $\blam$ at the top of the diagram. Usually we label the residues of the nodes instead, but in this case it is not necessary, as all strings share the same residue.
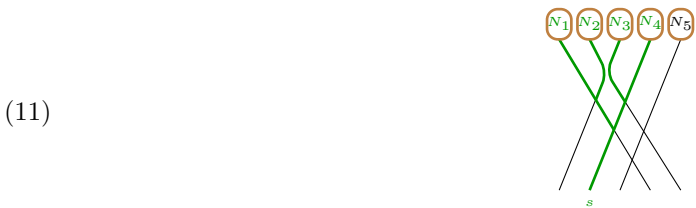
\begin{equation}
\begin{braid}
	\def\h{5};
	\def\sep{2}
	\def\a{3};
	\def\dashpart{0.9}
	\def\hrect{1.15}
	
	\path (0,0) -- (2, \h) coordinate[pos = 0.78] (A);
	\draw[color=black, thin] (0,0) -- (A) -- (1, \h)  node[anchor=south, color = green!60!black!]{$N_2$};
	\draw [color=black, thin](2,0) -- (4, \h)  node[anchor=south, color = black]{$N_5$};
	\draw[color=black, thin](3, 0) -- (0, \h) node[anchor=south, color = green!60!black!]{$N_1$};
	\draw [color=black, thin](4,0) --(A) --  (2, \h)  node[anchor=south, color = green!60!black!]{$N_3$};
	\draw [color=green!60!black!](1,0) node[anchor = north]{$s$} -- (3, \h)  coordinate[pos = 0.4] (B) coordinate[pos = 0.585] (C) node[anchor=south, color = green!60!black!]{$N_4$} ;
	\draw [color=green!60!black!] (B) -- (0, \h) coordinate[pos =0.33] (D);
	\draw[color=green!60!black!] (C) -- (A) -- (2, \h);
	\draw[color=green!60!black!] (D) -- (A) -- (1, \h);
	\draw[brown] (-0.4, \h) rectangle (0.4, \h + 1);
	\draw[brown] (0.6, \h) rectangle (1.4, \h + 1);
	\draw[brown] (1.6, \h) rectangle (2.4, \h + 1);
	\draw[brown] (2.6, \h) rectangle (3.4, \h + 1);
	\draw[brown] (3.6, \h) rectangle (4.4, \h + 1);
\end{braid}
\end{equation}
\sloppy We have drawn the Specht module diagram with thin black lines, and marked the top with labels $N_1, N_2, N_3, N_4, N_5$ for five nodes. We have drawn the tree diagram for string $s$ with thicker green lines superposed over the Specht module diagram. We see that the tree reaches nodes $N_1, N_2, N_3, N_4$ in green, and therefore $A(s) = \{N_1, N_2, N_3, N_4\}$. Since $s$ reaches the largest node in $A(s)$, we see that $s$ is co-stubborn.
\end{eg}

Next we prove a few simple lemmas concerning stubborn strings. Recall that strings in a diagram are totally ordered through their integer labels.

\begin{lem} \label{L:Stubborn0}
Let $v \in S^\blam$ be a standard monomial, and let $s$ be a string of residue $i$ in $v$. Then $s$ is stubborn if and only if all $i$-strings crossing $s$ are smaller than $s$. 
\end{lem}
\begin{proof}
Suppose $s$ reaches node $N$ in $v$. For $s$ to be stubborn, all strings which cross $s$ must reach a larger node than $N$. Since $v$ is a standard monomial, strings cannot cross $s$ twice in $v$. Therefore, only strings which are smaller than $s$ may cross it. The converse is clear.
\end{proof}

\begin{lem} \label{L:Stubborn1}
Let $v \in S^\blam$ be a standard monomial. Suppose $s$ is a stubborn string in $v$ reaching node $N$. Then for all $v' \in \mathcal{V}_v$, the string reaching $N$ is smaller than or equal to $s$. 
\end{lem}

\begin{proof}
Consider $v \in S^{\blam}$ as an element of $R_{n,e}^\Lambda$. There is an involution $\omega: R_{n,e}^{\Lambda} \rightarrow R_{n,e}^{\Lambda}$ that performs a reflection across a horizontal line through a diagram. It is clear from the previous lemma that string $s$ in $v \in R_{n,e}^\blam$ is stubborn if and only if its associated string $\omega(s)$ in $\omega(v)$ is co-stubborn. Therefore,  for any $\omega(v')$ with $v' \in \mathcal{V}$, $\omega(s)$ may only reach an element at the top of the diagram to the left of the element that it reaches in $\omega(v)$. Consequently, for any $v' \in \mathcal{V}$, only a string that is smaller than or equal to $s$ may reach node $N$.
\end{proof}

The following lemma concerns standard monomials whose underlying permutation is fully commutative (that is, $321$-avoiding). We refer to these as \emph{fully commutative standard monomials}.

\begin{lem} \label{L:Stubborn2}
Let $v \in S^\blam$ be a fully commutative standard monomial. Then all strings in $v$ are either stubborn or co-stubborn.
\end{lem}

\begin{proof}
Suppose that string $s$ in the diagram for $v$ is neither stubborn nor co-stubborn. Then there must be two strings crossing it from opposite directions: a string $s_1$ smaller than $s$ crosses it from bottom-left to top-right, and a string $s_2$ greater than $s$ crosses it from bottom-right to top-left. Since $v$ is a standard monomial, strings may not intersect twice. Therefore, $s_2, s, s_1$ make a $321$-pattern, which is in contradiction with the assumption that the underlying permutation for $v$ is fully commutative.
\end{proof}

\begin{lem} \label{L:Stubborn3}
Let $v \in S^\blam$ be a standard monomial. Suppose $s$ is a stubborn string with residue $i$ in $v$ and reaches node $N \in [\blam]$. Let $v' \in \mathcal{V}$ be any diagram where at least one of the $(i, i)$-crossings of $s$ has been cut. Then $N \notin A_{v'}(s)$.
\end{lem}

\begin{proof}
By \cref{L:Stubborn0}, only smaller strings than $s$ may cross $s$ in $v$. If they do, then they must cross $s$ only once from bottom-left to top-right, as $v$ is a standard monomial. To reach $N$, string $s$ must travel a path through its tree diagram to reach $N$. 

Let $v' \in \mathcal{V}$. At each crossing of its own tree diagram, $s$ may continue to the left or right, depending on whether the crossing has been cut to obtain diagram $v'$ from diagram $v$. At the lowest $(i, i)$-crossing of $s$ in $v$ which has been cut, $s$ veers to the right in the diagram for $v'$. It is impossible for string $s$ to return to its original course in $v$, as it is situated to the right-hand side of the path that $s$ takes in $v$, and no $i$-string intersects $s$ from bottom-right to upper-left in $v$.
\end{proof}

The following lemma is clear from the definition of stubborn strings.
\begin{lem} 
Let $v \in S^\blam$ be a monomial, and suppose a string $s$ in the diagram for $v$ is stubborn. Then $v = \sum c_\t v_\t$, where the sum is over all standard $\blam$-tableaux and $c_\t \neq 0$ only if $\t^{-1}(s)$ is a node larger than or equal to $\t_v^{-1}(s)$.
\end{lem}

The next two definitions describe stronger conditions than stubbornness.

\begin{defn}
    We say a string $s$ in a monomial $v \in S^\blam$ is \emph{strongly (co-)stubborn} if $s$ is (co-)stubborn in all $v' \in \mathcal{V}_v$.
\end{defn}

\begin{defn}
We say that a string $s$ of residue $i$ in monomial $v \in S^\blam$ is \emph{immobile} if, for any monomial $v' \in \mathcal{V}$ where an $(i, i)$-crossing of string $s$ has been cut, we have $v' = 0$. We say that a node is immobile if the string reaching it is immobile. 
\end{defn}

In particular, strings which do not cross other strings with the same residue are immobile, as the condition above is vacuously true.

Given a subset $\mathfrak{A}$ of the strings of a particular residue in a monomial $v$, we will need a definition for $A(\mathfrak{A})$, the accessible sets of nodes for strings in $\mathfrak{A}$. We start with an example.

\begin{eg}
	Let $v \in S^\blam$ be the diagram below, where $\blam = ((1),(1),(1),(1),(1))$ and all strings share the same residue.
	
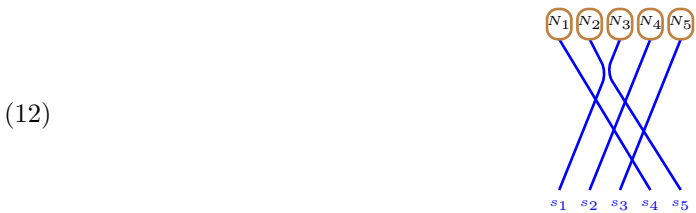
\begin{equation}
	\begin{braid}
	\def\h{5};
	\def\sep{2}
	\def\a{3};
	\def\dashpart{0.9}
	\def\hrect{1.15}
	\path (0,0) -- (2, \h) coordinate[pos = 0.78] (A);
	\draw (0,0) node[anchor = north]{$s_1$} -- (A) -- (1, \h)  node[anchor=south, color = black]{$N_2$};
	\draw (1,0) node[anchor = north]{$s_2$} -- (3, \h)  node[anchor=south, color = black]{$N_4$} ;
	\draw (2,0) node[anchor=north]{$s_3$} -- (4, \h)  node[anchor=south, color = black]{$N_5$};
	\draw(3, 0) node[anchor = north]{$s_4$} -- (0, \h) node[anchor=south, color = black]{$N_1$};
	\draw (4,0) node[anchor = north]{$s_5$} --(A) --  (2, \h)  node[anchor=south, color = black]{$N_3$};
	\draw[brown] (-0.4, \h) rectangle (0.4, \h + 1);
	\draw[brown] (0.6, \h) rectangle (1.4, \h + 1);
	\draw[brown] (1.6, \h) rectangle (2.4, \h + 1);
	\draw[brown] (2.6, \h) rectangle (3.4, \h + 1);
	\draw[brown] (3.6, \h) rectangle (4.4, \h + 1);
	\end{braid}
	\end{equation}
	Then strings $s_1, s_2$ and $s_3$ are all co-stubborn, while $s_4$ and $s_5$ are stubborn. Moreover, $s_1$ is strongly co-stubborn, while $s_5$ is strongly stubborn. 

	Next we look at the accessible nodes of $s_2$ and $s_3$. We have $A(s_2) = \{N_1, N_2, N_3, N_4\}$ and $A(s_3) = \{N_1, N_2, N_3, N_4, N_5\}$, but note that $\{s_1,s_2\}$ can never reach the set $\{N_1, N_2\}$ in a diagram in $\mathcal{V}_v$. To explain the reason, we draw a thin red line through the middle of the diagram.
	
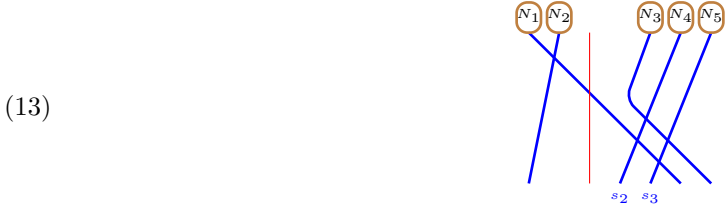
\begin{equation}
	\begin{braid}
	\def\h{5};
	\def\sep{2}
	\def\a{3};
	\def\dashpart{0.9}
	\def\hrect{1.15}
	\path (6,0) -- (1, \h) coordinate[pos = 0.56] (A);
	\draw (0,0) -- (1, \h) node[anchor=south, color = black]{$N_2$};
	\draw (3,0) node[anchor = north]{$s_2$}-- (5, \h) node[anchor=south, color = black]{$N_4$};
	\draw (4,0) node[anchor = north]{$s_3$} -- (6, \h) node[anchor=south, color = black]{$N_5$};
	\draw(5, 0) -- (0, \h) node[anchor=south, color = black]{$N_1$};
	\draw (6,0) --(A) --  (4, \h) node[anchor=south, color = black]{$N_3$};
	\draw[thin, color=red] (2, 0) -- (2, \h);
	\draw[brown] (-0.4, \h) rectangle (0.4, \h + 1);
	\draw[brown] (0.6, \h) rectangle (1.4, \h + 1);
	\draw[brown] (3.6, \h) rectangle (4.4, \h + 1);
	\draw[brown] (4.6, \h) rectangle (5.4, \h + 1);
	\draw[brown] (5.6, \h) rectangle (6.4, \h + 1);
\end{braid}
	\end{equation}
	Consider the graph with vertices given by all $(i, i)$-crossings, all origins of $i$-strings at the bottom of the diagram, and all $i$-nodes at the top of the diagram; and edges given by the segments of $i$-strings connecting the vertices in the diagram above. Designate the starting points of strings $s_2$ and $s_3$ as sources and nodes $N_1, N_2$ as sinks. Then it is possible to cut off the sources from the sinks by cutting a single edge. By the max-flow/min-cut theorem, the maximal flow from $s_2, s_3$ to $N_1, N_2$ is at most $1$, so that only one of $s_2$ or $s_3$ may reach nodes $N_1, N_2$ in a single diagram. Therefore, it makes sense to define $A(\{s_2,s_3\})$ in such a way that  $\{N_1, N_2\} \notin A(\{s_2, s_3\})$.
\end{eg}

The argument above motivates the following definition and can easily be modified to prove the lemma below it.

\begin{defn} \label{D:AccessibleSets}
Let $v \in S^\blam$ be a monomial, $\mathfrak{A}$ be a set of $m$ strings in $v$ with the same residue. Define the graph of residue $i$ associated with the diagram of $v$ as in the example above. Then $A(\mathfrak{A})$ is the set of all $\mathcal{N} \subset [\blam]$ of size $m$ such that the maximal flow from the start of the strings of $\mathfrak{A}$ to the nodes in $\mathcal{N}$ in diagram $v$ is $m$.
\end{defn}

\begin{lem} \label{L:AccessibleSets}
Let $v \in S^\blam$ be a monomial and let $\mathfrak{A}$ be a set of strings in $v$ with the same residue. Let $v' \in \mathcal{V}_v$, and let $\mathcal{N}$ be the set of nodes reached by strings $\mathfrak{A}$ in $v'$. Then $\mathcal{N} \in A(\mathfrak{A})$.
\end{lem}

We introduce a partial order between sets of nodes of the same size. Let $\mathcal{N}, \mathcal{M} \subset [\blam]$ be two sets of nodes such that $|\mathcal{N}|=|\mathcal{M}|$. Then we write $\mathcal{N} \succ \mathcal{M}$ if the greatest node of $\mathcal{N}$ is greater than the greatest node of $\mathcal{M}$.

We define two notions of stubbornness for sets of strings.

\begin{defn}
	Let $v\in S^\blam$ be a monomial. Let $\mathfrak{A}$ be a set of strings with residue $i$ reaching the set of nodes $\mathcal{N}$ in $v$. Then we say that $\mathfrak{A}$ is \emph{stubborn} if $\mathcal{N}$ is minimal in $A(\mathfrak{A})$. We say that $\mathfrak{A}$ is \emph{strongly stubborn} if $\mathfrak{A}$ is stubborn for all $v' \in \mathcal{V}_v$.
\end{defn}

\begin{eg}
	Consider the following diagram, where all strings are assumed to be of the same residue.
	\begin{equation*}
	v =
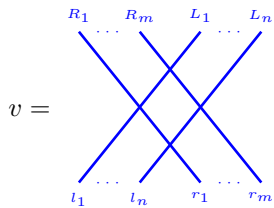
	\begin{braid}
	\def\h{5};
	\def\sep{2}
	\def\a{2+\sep};
	\def\dashpart{0.9}
	\def\hrect{1.15}
	
	\draw (0,0) node[anchor=north]{$l_1$} -- (\a, \h) node[anchor=south]{$L_1$};
	\draw (\sep/2, 0) node{$\cdots$};
	\draw (\a + \sep/2, \h, 0) node{$\cdots$};
	\draw (\sep, 0) node[anchor=north]{$l_n$}-- (\a + \sep, \h) node[anchor=south]{$L_n$};
	\draw (\a, 0) node[anchor=north]{$r_1$} -- (0, \h) node[anchor=south]{$R_1$};
	\draw (\a + \sep/2, 0) node{$\cdots$};
	\draw (\sep/2, \h, 0) node{$\cdots$};
	\draw (\a + \sep, 0) node[anchor=north]{$r_m$} -- (\sep, \h) node[anchor=south]{$R_m$};
	\end{braid}	
	\end{equation*}
	Write $\mathcal{L} = \{l_1, \dots, l_n\}$ and $\mathcal{R} = \{r_1, \dots, r_m\}$. Note that all strings in $\mathcal{L}$ are co-stubborn and all strings in $\mathcal{R}$ are stubborn. Then we claim that $\mathcal{R}$ is strongly stubborn. 
	
	Let $v' \in \mathcal{V}_v$. Let $l_x$ be the rightmost string in $\mathcal{L}$ such that at least one of its crossings in $v$ has been cut to reach diagram $v'$. Then we argue that the rightmost node reached by a string in $\mathcal{R}$ is $L_x$. Indeed, after we cut some crossing for string $l_x$, $L_x$ is no longer in $A(l_x)$. Therefore, some other string $s$ must reach node $L_x$ in $v'$. All strings $l_y > l_x$ reach $L_y$, so we must show that $s \in \mathcal{R}$. If we assume the contrary, then $s \in \mathcal{L}$, and moreover $s > l_x$ because all the strings in $\mathcal{L}$ are co-stubborn. But this leads to a contradiction: by assumption, we have not cut any crossings for strings larger than $l_x$, so any string larger than $l_x$ cannot reach $L_x$ in $v'$.
	
	In a further diagram $v'' \in \mathcal{V}_{v'}$, the rightmost node reached by a string in $\mathcal{R}$ is either $L_x$ or a node $L_y > L_x$, where  $l_y$ is the rightmost string in $\mathcal{L}$ such that at least one of its crossings in $v'$ has been cut to reach diagram $v''$.  We conclude that $\mathcal{R}$ is strongly stubborn.
\end{eg}

The argument above can be used to prove the following more general lemma.

\begin{lem} \label{L:StronglyStubbornSet}
	Let $v \in S^\blam$ be a standard monomial. Let $\mathcal{R}$ be a set of strings in $v$, and denote the rest of the strings in the diagram by $\mathcal{L}$. Suppose any two strings $r \in \mathcal{R}$ and $l \in \mathcal{L}$ may only intersect if $r > l$. Suppose further that the strings in $\mathcal{L}$ do not intersect each other. Then $\mathcal{R}$ is strongly stubborn.
\end{lem}

Stubborn sets of strings satisfy the following property, which generalizes \cref{L:Stubborn3} and can be proved in a similar manner.

\begin{lem} \label{L:Stubborn3Sets}
Let $v \in S^\blam$ be a standard monomial. Suppose $\mathfrak{A}$ is a stubborn set of $i$-strings in $v$ reaching the set of nodes $\mathcal{N} \subset [\blam]$. Let $v' \in \mathcal{V}$ be any diagram where at least one of the $(i, i)$-crossings of a string $s \in \mathfrak{A}$ with a string $s' \notin \mathfrak{A}$ has been cut. Then $\mathcal{N} \notin A_{v'}(\mathfrak{A})$.
\end{lem}

\subsection{String activity}

Let $v \in S^\blam$ be a monomial, and suppose $s$ is a string in $v$ with residue $i$ such that all other strings have residue $j$ with $|i-j| \geq 2$. We refer to such strings as \emph{inert} strings. Then \cref{E:PsiSquared,E:Braid,E:DotSlide} simplify to

\begin{equation*} 
    \def\h{3}
 \begin{braid}
            \draw (0,0) node[anchor=north]{$s$} -- (1, \h/2) -- (0, \h) node[anchor=south]{$i$};
            \draw (1,0) -- (0, \h/2) -- (1, \h) node[anchor=south]{$j$};
        \end{braid}
        =
        \begin{braid}
            	\draw (0,0) node[anchor=north]{$s$} -- (0, \h) node[anchor=south]{$i$};
            	\draw (1,0) -- (1, \h) node[anchor=south]{$j$};
                \end{braid}
    \text{,} \quad 
        \begin{braid}
    	\draw (0,0)  -- (2, \h) node[anchor=south]{$k$};
    	\draw (2,0) node[anchor=north]{$s$} -- (0,\h) node[anchor=south]{$i$};
    	\draw (1, 0) -- (0, \h/2) -- (1, \h)node[anchor=south]{$j$};
\end{braid}
=
\begin{braid}
        	\draw (0,0)  -- (2, \h) node[anchor=south]{$k$};
        	\draw (2,0) node[anchor=north]{$s$} -- (0,\h) node[anchor=south]{$i$};
        	\draw (1, 0) -- (2, \h/2) -- (1, \h)node[anchor=south]{$j$};
        \end{braid}
        \text{,} \quad 
        \begin{braid}
    	\draw (0,0)  -- (2, \h) node[anchor=south]{$k$};
    	\draw (2,0)  -- (0,\h) node[anchor=south]{$j$};
    	\draw (1, 0) node[anchor=north]{$s$}-- (0, \h/2) -- (1, \h)node[anchor=south]{$i$};
\end{braid}
=
\begin{braid}
        	\draw (0,0)  -- (2, \h) node[anchor=south]{$k$};
        	\draw (2,0) -- (0,\h) node[anchor=south]{$j$};
        	\draw (1, 0) node[anchor=north]{$s$}  -- (2, \h/2) -- (1, \h)node[anchor=south]{$i$};
        \end{braid}
                \quad \text{and} \quad
                \begin{braid}
    	\draw (0,0)  -- (2, \h) coordinate[pos = 0.25] (B) node[anchor=south]{$j$};
    	\draw (2,0) node[anchor=north]{$s$} -- (0,\h) node[anchor=south]{$i$};
    	\greendot(B);
    \end{braid}
     = 
     \begin{braid}
    	\draw (0,0) -- (2, \h) coordinate[pos = 0.75] (B) node[anchor=south]{$j$};
    	\draw (2,0) node[anchor=north]{$s$} -- (0,\h) node[anchor=south]{$i$};
    	\greendot(B);
    \end{braid}.
\end{equation*} 

The above relations become simple isotopy relations when $s$ is removed from the diagrams:

\begin{equation*} 
    \def\h{3}
 \begin{braid}
            \draw (1,0) -- (0, \h/2) -- (1, \h) node[anchor=south]{$j$};
        \end{braid}
        =
        \begin{braid}
            	\draw (1,0) -- (1, \h) node[anchor=south]{$j$};
                \end{braid}
    \text{,} \quad 
        \begin{braid}
    	\draw (0,0)  -- (2, \h) node[anchor=south]{$k$};
    	\draw (1, 0) -- (0, \h/2) -- (1, \h)node[anchor=south]{$j$};
\end{braid}
=
\begin{braid}
        	\draw (0,0)  -- (2, \h) node[anchor=south]{$k$};
        	\draw (1, 0) -- (2, \h/2) -- (1, \h)node[anchor=south]{$j$};
        \end{braid}
        \text{,} \quad 
        \begin{braid}
    	\draw (0,0)  -- (2, \h) node[anchor=south]{$k$};
    	\draw (2,0)  -- (0,\h) node[anchor=south]{$j$};
\end{braid}
=
\begin{braid}
        	\draw (0,0)  -- (2, \h) node[anchor=south]{$k$};
        	\draw (2,0) -- (0,\h) node[anchor=south]{$j$};
        \end{braid}
                \quad \text{and} \quad
                \begin{braid}
    	\draw (0,0)  -- (2, \h) coordinate[pos = 0.25] (B) node[anchor=south]{$j$};
    	\greendot(B);
    \end{braid}
     = 
     \begin{braid}
    	\draw (0,0) -- (2, \h) coordinate[pos = 0.75] (B) node[anchor=south]{$j$};
    	\greendot(B);
    \end{braid}
\end{equation*} 

Since all relations involving $s$ are invariant when the string disappears, we may simply stop drawing $s$. The extra terms in \cref{E:Braid} and \cref{E:DotSlide} which do not appear for inert strings are called \emph{error terms}.

The following lemma is clear from the definitions.

\begin{lem} \label{L:ImmobileInactive}
Let $v \in S^\blam$ be a monomial, and suppose $s$ is an immobile string in $v$ with residue $i$. Then the relations
\begin{equation*} 
    \def\h{3}
        \begin{braid}
    	\draw (0,0)  -- (2, \h) node[anchor=south]{$k$};
    	\draw (2,0) node[anchor=north]{$s$} -- (0,\h) node[anchor=south]{$i$};
    	\draw (1, 0) -- (0, \h/2) -- (1, \h)node[anchor=south]{$j$};
\end{braid}
=
\begin{braid}
        	\draw (0,0)  -- (2, \h) node[anchor=south]{$k$};
        	\draw (2,0) node[anchor=north]{$s$} -- (0,\h) node[anchor=south]{$i$};
        	\draw (1, 0) -- (2, \h/2) -- (1, \h)node[anchor=south]{$j$};
        \end{braid}
                \quad \text{and} \quad
                \begin{braid}
    	\draw (0,0)  -- (2, \h) coordinate[pos = 0.25] (B) node[anchor=south]{$j$};
    	\draw (2,0) node[anchor=north]{$s$} -- (0,\h) node[anchor=south]{$i$};
    	\greendot(B);
    \end{braid}
     = 
     \begin{braid}
    	\draw (0,0) -- (2, \h) coordinate[pos = 0.75] (B) node[anchor=south]{$j$};
    	\draw (2,0) node[anchor=north]{$s$} -- (0,\h) node[anchor=south]{$i$};
    	\greendot(B);
    \end{braid}
\end{equation*} 
hold for string $s$.
\end{lem}

In this sense, immobile strings are partially inactive: dots and strings can be pulled through their crossings without leaving error terms. On the other hand, they differ from inert strings, as \cref{E:PsiSquared} is not simplified, and sliding immobile strings past crossings may leave error terms. 

Sometimes we want to show that a certain monomial is non-zero, but carrying out a full decomposition into basis elements is computationally challenging. If the monomial features stubborn or co-stubborn strings, then the next lemma can be helpful.

\begin{lem} \label{L:ImmobileTrick}
Let $v \in S^\blam$ be a monomial, and suppose $s$ is a (co-)stubborn string in $v$. Substitute the usual relations involving $s$ for the relations in \cref{L:ImmobileInactive}, and compute a ``simplified'' decomposition of $v$ into basis elements by successively applying the usual quiver Hecke and Specht relations, except for this substitution. If the result is non-zero, then the actual decomposition of $v$ into basis elements is also non-zero. Moreover, if a basis element appears in the simplified decomposition, then it also appears in the actual decomposition with the same coefficient.
\end{lem}

\begin{proof}
Suppose string $s$ reaches node $N$ in $v$. Let us compute the terms in the decomposition of $v$ that feature string $s$ reaching node $N$. Let $v' \in \mathcal{V}$ be any monomial in $\mathcal{V}$ where string $s$ has been cut at least once. Then $N \notin A_{v'}(s)$. Therefore, no monomial where $s$ reaches $N$ can be obtained from $v'$, and we may discard $v'$ in our current computation.

On the other hand, by substituting the usual relations by those in \cref{L:ImmobileInactive} we are discarding exactly all terms $v'$ where string $s$ has been cut. It follows that we will obtain all terms in the decomposition of $v$ that feature string $s$ reaching node $N$. The consequences in the statement follow easily.
\end{proof}

If we apply the above lemma to a string, we may also say that we \emph{immobilize} that string. In computations, we draw both immobile and immobilized strings with green dashed lines. As long as we do not slide strings over green dashed lines, creating double crossings, and we do not slide crossings past green dashed lines, these lines can be ignored. If several green dashed lines are next to each other, we may bundle them together and draw them as a single line.

\begin{eg}
Let $v \in S^\blam$ be the diagram below, where $\blam = ((1),(1),(1),(1),(1),(1))$ and all strings have the same residue.
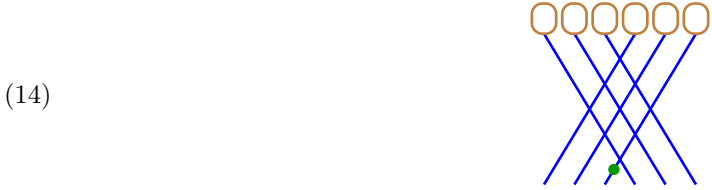
\begin{equation}
\begin{braid}
	\draw (0, 0) -- (3, \h) ;
	\draw (1, 0) -- (4, \h);
	\draw (2, 0) -- (5, \h) coordinate[pos=0.1] (A);
	\greendot(A);
	\draw (3, 0) -- (0, \h);
	\draw (4, 0) -- (1, \h);
	\draw (5, 0) -- (2, \h);
	\draw[brown] (-0.4, \h) rectangle (0.4, \h + 1);
	\draw[brown] (0.6, \h) rectangle (1.4, \h + 1);
	\draw[brown] (1.6, \h) rectangle (2.4, \h + 1);
	\draw[brown] (2.6, \h) rectangle (3.4, \h + 1);
	\draw[brown] (3.6, \h) rectangle (4.4, \h + 1);
	\draw[brown] (4.6, \h) rectangle (5.4, \h + 1);
\end{braid}
\end{equation}
Since the fourth and fifth strings are stubborn, we may immobilize them, and since they are right next to each other we draw them as a single string.
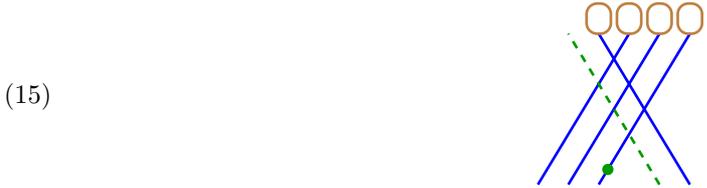
\begin{equation}
\begin{braid}
	\draw (0, 0) -- (3, \h) ;
	\draw (1, 0) -- (4, \h);
	\draw (2, 0) -- (5, \h) coordinate[pos=0.1] (A);
	\greendot(A);
	\draw[green!60!black!,dashed] (4, 0) -- (1, \h);
	\draw (5, 0) -- (2, \h);
	\draw[brown] (1.6, \h) rectangle (2.4, \h + 1);
	\draw[brown] (2.6, \h) rectangle (3.4, \h + 1);
	\draw[brown] (3.6, \h) rectangle (4.4, \h + 1);
	\draw[brown] (4.6, \h) rectangle (5.4, \h + 1);
\end{braid}
\end{equation}
Next, we may slide the green dot up through the third string. The resulting diagram when the dot reaches the top of the string equals zero due to \cref{E:Specht}, so the above simplifies to
\begin{equation}
-\begin{braid}
	\draw (0, 0) -- (3, \h) ;
	\draw (1, 0) -- (4, \h);
	\path (2, 0) -- (5, \h) coordinate[pos=0.1] (A) coordinate[pos=0.5] (B);
	\draw (2,0) -- (B) -- (2,\h);
	\draw[green!60!black!,dashed] (4, 0) -- (1, \h);
	\draw (5, 0) -- (B) -- (5, \h);
	\draw[brown] (1.6, \h) rectangle (2.4, \h + 1);
	\draw[brown] (2.6, \h) rectangle (3.4, \h + 1);
	\draw[brown] (3.6, \h) rectangle (4.4, \h + 1);
	\draw[brown] (4.6, \h) rectangle (5.4, \h + 1);
\end{braid}.
\end{equation}
By \cref{L:ImmobileTrick}, this term is guaranteed to appear in the full decomposition of the original diagram with the same multiplicity. 
\end{eg}

We are especially interested in the case where $v \in S^\blam$ is a fully commutative standard monomial. Then all strings of $v$ are stubborn or co-stubborn by \cref{L:Stubborn2}. Let $s$ be one such string. By the fully commutative condition, no string can cross $s$ twice in any monomial $v' \in \mathcal{V}_v$ where the crossings of $s$ have not been cut. In particular, the following holds.
\begin{lem} \label{L:ImmobileIgnore}
Suppose $v \in S^\blam$ is a fully commutative standard monomial. Let $s$ be an immobile or immobilized string in $v$. Then, $s$ does not intersect any string more than once in any non-zero monomial $v' \in \mathcal{V}_v$. 
\end{lem}
In this case,  there is no danger of creating double crossings when applying the relations in \cref{L:ImmobileInactive} to immobile or immobilized strings. Consequently, we may stop drawing those strings and perform our computations without considering them.

\section{Change of rings} \label{S:Degeneration}

Let $r$ be a prime number, $a \in \mathbb{Z}_{\geq 1}$, $e = r^a$ and let $\zeta$ be a primitive $e$-th root of unity. Let $M, N$ be $H_{n,\zeta}^\Lambda(\mathbb{Z}[\zeta])$-modules. Then $M_\mathbb{C} \coloneq \mathbb{C} \otimes_{\mathbb{Z}[\zeta]} M$ is an $H_{n,\zeta}^\Lambda(\mathbb{C})$-module, and according to~\cite[Theorem 2.38]{CR1}, we have
\begin{equation}
\Hom_{H_{n,\zeta}^\Lambda(\mathbb{C})}(M_\mathbb{C}, N_\mathbb{C}) \cong \mathbb{C} \otimes_{\mathbb{Z}[\zeta]}\Hom_{H_{n,\zeta}^\Lambda(\mathbb{Z}[\zeta])}(M, N).
\end{equation}
In particular, it is clear that $\Hom_{H_{n,\zeta}^\Lambda(\mathbb{C})}(M_\mathbb{C}, N_\mathbb{C}) \neq 0$ if and only if $\Hom_{H_{n,\zeta}^\Lambda(\mathbb{Z}[\zeta])}(M, N) \neq 0$.

Recall that the cyclotomic polynomial $\Phi_e(x)$ satisfies $\Phi_e(1) = r$, so that the map $\phi: \mathbb{Z}[\zeta] \twoheadrightarrow \mathbb{F}_{r}$ defined by $\phi(\zeta) = 1$ is a homomorphism of rings with kernel $\p = (1 - \zeta)$.

\begin{prop} \label{T:Degen1}
There exists a functor $\mathcal{F}: H_{n,\zeta}^\Lambda(\mathbb{Z}[\zeta])\text{-}\mmod \rightarrow H_{n,1}^\Lambda(\mathbb{F}_r)\text{-}\mmod$ defined as $M \mapsto M/\p M$, and $\mathcal{F}(\psi) \neq 0$ if and only if $\psi \in \Hom_{H_{n,\zeta}^\Lambda(\mathbb{Z}[\zeta])}(M, N) \setminus \p\Hom_{H_{n,\zeta}^\Lambda(\mathbb{Z}[\zeta])}(M, N)$. In particular, if $\Hom_{H_{n,\zeta}^\Lambda(\mathbb{Z}[\zeta])}(M, N) \neq \{0\}$ then it follows that $\Hom_{H_{n,1}^\Lambda(\mathbb{F}_r)}(\mathcal{F}(M), \mathcal{F}(N)) \neq \{0\}$.
\end{prop}

\begin{proof}
Let $\Psi$ be the matrix associated with $\psi$, considered as a linear transformation. It is clear that $\phi({\Psi})$, the matrix obtained by applying $\phi$ to each entry of $\Psi$, represents a homomorphism from $M/\p M$ to $N / \p N$, which confirms that $\mathcal{F}$ is a functor. 

It is clear that, if $\psi \in \p\Hom_{H_{n,\zeta}^\Lambda(\mathbb{Z}[\zeta])}(M, N)$, then $\mathcal{F}(\psi) = 0$. Now, suppose that $\mathcal{F}(\psi) = 0$, so that $\phi(\Psi) = 0$. We show that then $1 - \zeta$ divides all entries of $\Psi$. Let $x \in \mathbb{Z}[\zeta]$ be one such entry, so that $\phi(x) = 0$. Writing $x = \sum_{i = 0}^k a_i \zeta^i$, it follows that $\sum_{i=0}^k a_i$ is a multiple of $r$. But we also have $x = \left(\sum_{i=0}^k a_i \right) + \left( \sum_{i = 0}^k a_i(\zeta^i - 1) \right)$. The number $1 - \zeta$ clearly divides each summand in the second sum, and it divides the first sum by~\cite[4.38]{CR1}. Therefore $1 - \zeta$ divides $x$, and we see that $\psi \in \p\Hom_{H_{n,\zeta}^\Lambda(\mathbb{Z}[\zeta])}(M, N)$.
\end{proof}

We write $M_{\mathbb{F}_r} \coloneq \mathcal{F}(M) = M/\mathfrak{p}M$.

\begin{cor} \label{C:DegenHecke}
Let $M, N$ be $H_{n,\zeta}^\Lambda(\mathbb{Z}[\zeta])$-modules and suppose that $\Hom_{H_{n,\zeta}^\Lambda(\mathbb{C})}(M_\mathbb{C}, N_\mathbb{C}) \neq \{0\}$. Then $\Hom_{H_{n,1}^\Lambda(\mathbb{F}_r)}(M_{\mathbb{F}_r}, N_{\mathbb{F}_r}) \neq \{0\}$.
\end{cor}

The Brundan--Kleshchev isomorphism yields an analogous result for Specht modules of cyclotomic quiver Hecke algebras.
\begin{cor} \label{C:Degen}
Let $r$ be a prime number, $a \in \mathbb{Z}_{\geq 1}$ and $e = r^a$. Let $M, N$ be $R_{n,e}^\Lambda(\mathbb{C})$-modules and suppose that $\Hom_{R_{n,e}^\Lambda(\mathbb{C})}(M, N) \neq  \{0\}$. Then $\Hom_{R_{n,r}^\Lambda(\mathbb{F}_r)}(M_{\mathbb{F}_r}, N_{\mathbb{F}_r}) \neq \{0\}$. \end{cor}

Informally, let $p$ be the characteristic of the field and $e$ be the length of the quiver $A_{e-1}^{(1)}$ corresponding to a cyclotomic quiver Hecke algebra. Then, if we have an $R_n^\Lambda$-module homomorphism under $(p, e) = (0, r^a)$ for some integer $a$, it follows that we have an $R_n^\Lambda$-module homomorphism under $(p, e) = (r, r)$.

\begin{eg} \label{Eg:Degen} 
Let $R_{5, 4}^{\Lambda_0}(\mathbb{C})$ be the cyclotomic quiver Hecke algebra with parameters $n = 5, e = 4, p = 0$. The main result in~\cite{lm14} gives a non-zero homomorphism from $S^\lambda$ to $S^\mu$, where $\lambda = (3, 2), \mu = (5)$. This is because the two shapes differ by a single horizontal strip of length less than $e = 4$ with matching residues:
\begin{equation}
\lambda = \ShadedTableau[(1, -1), (2, -1)]{{0,1,2},{3,0}}{3}, \quad \mu = \ShadedTableau[(4, 0), (5, 0)]{{0,1,2,3,0}}{1}
\end{equation}
Then an application of \cref{C:Degen} shows that there is a non-zero homomorphism in $\Hom_{R_{5,2}^{\Lambda_0}(\mathbb{F}_2)}(S^\lambda_{\mathbb{F}_2}, S^\mu_{\mathbb{F}_2})$, which cannot be directly deduced from the main result in \cite{lm14}. The existence of a homomorphism can also be seen by applying \cref{T:Carter-Payne}. 
\end{eg}

In fact, the application of \cref{C:Degen} to the main theorem in~\cite{lm14} yields a full generalization of the classical Carter--Payne theorem for symmetric groups (\cref{T:Carter-Payne}). We do not write the details because our main theorem is a further generalization.

A result analogous to \cref{C:DegenHecke} for decomposition matrices can be found in~\cite[Section 6.2]{Mathas}. In the above discussion we did not make use of the generators and relations of the cyclotomic Hecke algebra, so the result can be made more general. For instance, the cyclotomic $q$-Schur algebra may be considered instead of the cyclotomic Hecke algebra.

\begin{rem} 
Our proof of \cref{C:Degen} used the Hecke algebra setting. A proof in the \emph{quiver} Hecke algebra setting is highly desirable for several reasons. Firstly, the degree of a homomorphism is in general not preserved under \cref{C:Degen}, unless $a = 1$. For instance, in \cref{Eg:Degen}, the homomorphism where $e = 4, p = 0$ has degree $+1$, as the description of the degree in \cite{lm14} shows. However, the homomorphism where $e = 2, p = 2$ has degree $0$. Secondly, it is likely that a finer result holds for quiver Hecke algebras, where a homomorphism between $R_n^\Lambda$-modules under $(p, e) = (r, r^a)$ implies the existence of a homomorphism of $R_n^\Lambda$-modules under $(p, e) = (r, r^{a-1})$. However, Hecke algebras do not exist for $(p,e)=(r,r^a)$ unless $a = 1$, so the finer result cannot be obtained via the Brundan--Kleshchev isomorphism. A proof of \cref{C:Degen} directly in the quiver Hecke algebra setting is likely to shine light on both of these phenomena.
\end{rem}

\section{One-row Carter--Payne homomorphisms} \label{S:Rows}

If $\blam \vdash n$ and $\bnu \vdash n + \gamma$ are multipartitions such that $[\blam] \subset [\bnu]$, define $\t_\blam^\bnu$ to be the {unique maximal tableau in the dominance order such that $(\t_\blam^\bnu)_{\downarrow n} = \t^\blam$. A \emph{row-strip} $[\xi] \subset [\blam]$ is a contiguous set of horizontally-adjacent nodes inside a component of $[\blam]$ . It is said to be \emph{removable} if $[\blam] \setminus [\xi]$ is the shape of a partition.
 
Let $\blam$ and $\bmu$ be multipartitions of~$n$ and suppose that
  $\bmu\gdom\blam$, that $\bnu$ is a multipartition of
  $n+\gamma$ such that $[\bnu] = [\blam] \cup [\bmu]$, and $[\bnu]\backslash[\blam]$ and $[\bnu]\backslash[\bmu]$ are both
  removable row-strips of the same length~$\gamma < e$ and with the same sequence of residues $J=\res(\bnu\backslash\blam)=\res(\bnu\backslash\bmu)$. Write $[\lambda^*] = [\bnu] \backslash [\bmu]$ and $[\mu^*] = [\bnu] \backslash [\blam]$, and label the maximal removable row-strips which are subshapes of $[\mu^*]$ between $[\mu^*]$ and $[\lambda^*]$ as $[\mu^*] = [\xi^0] \prec [\xi^1] \prec \dots \prec [\xi^{c-1}] \prec [\xi^c] = [\lambda^*]$. That is, within the set of removable row-strips which are subshapes of $[\mu^*]$, we consider only those which are maximal under set containment, so that we do not take $[\xi]$ if there exists $[\rho]$ with $[\xi] \subset [\rho]$ (see \cref{Eg:Rows}). We assume throughout the section that $R^\Lambda_n$ is defined over a field of characteristic $p$, possibly zero, and with respect to a quiver $A_{e-1}^{(1)}$ with $e \in \mathbb{Z}_{\geq 2} \cup \{\infty\}$.

\begin{thm}\emph{\textbf{(\cite[Theorem 3.12]{lm14})}} \label{T:MainOneRow}
Under the conditions above, there exists a homomorphism $\theta \in \Hom_{R_n^\Lambda}(S^\blam\<a-b+2c\>,S^\bmu)$, where

\begin{equation*}
  \begin{aligned}
    a&=\#\{N \in \Add(\bnu)\mid\res(N) \in J \text{ and } \row(\bnu\backslash\blam) < \row(N) \leq \row(\bnu\backslash\bmu)\},\\
   b&=\#\{N \in \Rem(\bnu)\mid\res(N) \in J \text{ and } \row(\bnu\backslash\blam) <\row(N) \leq \row(\bnu\backslash\bmu)\}.\\
 \end{aligned}
\end{equation*}
Moreover, $a-b+2c>0$ and $\theta(v_{\t^\blam})$ is a basis element of $S^\bmu$ (for the basis described in \cref{SS:Specht}).
\end{thm}

We define $\t^*_\bmu \in \Std(\bmu)$ so that $v_{\t^*_\bmu}$ is the basis element described in the theorem. This tableau is equal to the tableau $\t_{z+1}$ in~\cite[Section 3]{lm14}, but we construct it slightly differently.

Let $j \in \{1, \dots, c\}$, and suppose that $\xi^j$ is a row-strip of length $l \leq \gamma$. Let $\sigma_{\mu^*, \xi^j}$ be the permutation that exchanges the numbers appearing in $[\xi^j]$ in the tableau $\t^\bnu_\blam$ with the numbers $\{n+1, \dots, n+l\}$, in order. Then we have $\t_\bmu^* = \left(\t_\blam^\bnu \sigma_{\mu^*, \xi^1} \dots \sigma_{\mu^*, \xi^{c-1}} \sigma_{\mu^*, \lambda^*}\right)_{\downarrow n}$. Recall that permutations act on tableaux by permuting the entries of the tableaux. 

\begin{eg} \label{Eg:Rows}
Let $p = 0$, $e = 4$ and $\Lambda = (0, 2, 1)$. Suppose $\blam = (\varnothing, (3,2,2), (4,3)) \vdash 14$ and $\bmu = ((3), (3,2,2), (4)) \vdash 14$, so that $\bnu =  ((3), (3,2,2), (4,3)) \vdash 17$. We draw $\bnu$ with the residues corresponding to each node.
\begin{equation*}
    \left(\,\ShadedTableau[(1,0),(2,0),(3,0)]{{0,1,2}}{4}, \quad \ShadedTableau{{2,3,0},{1,2},{0,1}}{4}, \quad\ShadedTableau[(1,-1),(2,-1),(3,-1)]{{1,2,3,0},{0,1,2}}{4}\,\right)
\end{equation*}
Here $[\mu^*]$ and $[\lambda^*]$ have been shaded in. It is easy to see that $a = 2, b = 4$ and $c = 4$, so that $\deg \theta = a - b + 2c = 6$. Let $\theta \in \Hom_{R_n^\Lambda}(S\<6\>, S^\bmu)$ be the homomorphism given by the theorem. To compute $\theta(v_{\t^\blam})$, we first write
\begin{equation*}
\t_\blam^\bnu = \left(\,\ShadedTableau[(1,0),(2,0),(3,0)]{{15,16,17}}{4}, \quad \ShadedTableau[(3,0),(1,-2),(2,-2)]{{1,2,3},{4,5},{6,7}}{4}, \quad\ShadedTableau[(1,-1),(2,-1),(3,-1),(4,0)]{{8,9,10,11},{12,13,14}}{4}\,\right) \in \Std(\bmu),
\end{equation*}
where we have shaded in the shapes $[\mu^*] = [\xi^0] \prec \dots \prec [\xi^4] = [\lambda^*]$.  Entries of $[\t_{\blam}^\bnu]$ outside every $[\xi^j]$ are fixed under the actions of the permutations $\sigma_{\mu^*, \xi^j}$. Next we apply the permutations $\sigma_{\mu^*, \xi^j}$ in order. As an example, we have
\begin{equation*}
\t_\blam^\bnu \sigma_{\mu^*, \xi^1} = \left(\,\ShadedTableau[(1,0)]{{3,16,17}}{4}, \quad \ShadedTableau[(3,0)]{{1,2,15},{4,5},{6,7}}{4}, \quad\ShadedTableau{{8,9,10,11},{12,13,14}}{4}\,\right) \in \Std(\bmu).
\end{equation*}
If we apply all permutations, we obtain
\begin{equation*}
\t_\blam^\bnu \sigma_{\mu^*, \xi^1} \dots \sigma_{\mu^*, \lambda^*} = \left(\,\ShadedTableau[(1,0),(2,0),(3,0)]{{3,7,14}}{4}, \quad \ShadedTableau[(3,0),(1,-2),(2,-2)]{{1,2,6},{4,5},{11,13}}{4}, \quad\ShadedTableau[(1,-1),(2,-1),(3,-1),(4,0)]{{8,9,10,12}, {15,16,17}}{4}\,\right) \in \Std(\bmu).
\end{equation*}
The tableau $\t_\bmu^*$ is obtained simply by restricting to the first $n=14$ positive integers. 
\begin{equation*}
\t_\bmu^* = \left(\,\ShadedTableau[(1,0),(2,0),(3,0)]{{3,7,14}}{4}, \quad \ShadedTableau[(3,0),(1,-2),(2,-2)]{{1,2,6},{4,5},{11,13}}{4}, \quad\ShadedTableau[(4,0)]{{8,9,10,12}}{4}\,\right) \in \Std(\bmu).
\end{equation*}
The associated basis element $v_{\t_\bmu^*}$ is given by the diagram
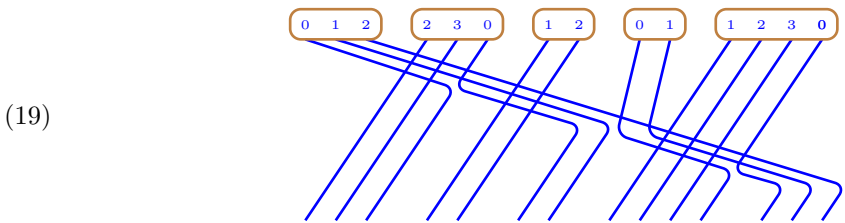
\begin{equation} \label{E:egDiagramRow}
 \begin{braid}
	\def\h{6};
	\def\sep{1}
	\def\a{15+\sep};
	\def\dashpart{0.9}
	
	\draw (0, 0) -- (3 + \sep, \h) node[anchor = south]{$2$};
	\draw (1, 0) -- (4 + \sep, \h) node[anchor = south]{$3$};
	\path (2, 0) -- (5 + \sep, \h) coordinate[pos = 0.73] (A);
	\draw (2, 0) -- (A) -- (0, \h) node[anchor = south]{$0$} ;
	\draw (3+\sep, 0) -- (6 + 2*\sep, \h) node[anchor = south]{$1$};
	\draw (4+\sep, 0) -- (7 + 2*\sep, \h) node[anchor = south]{$2$};
	\path (5+2*\sep, 0) -- (8 + 3*\sep, \h) coordinate[pos = 0.52] (B);
	\draw (5+2*\sep, 0) -- (B) -- (A) -- (5+\sep,\h) node[anchor = south]{$0$};
	\path (6+2*\sep, 0) -- (9 + 3*\sep, \h) coordinate[pos=0.525] (C);
	\draw (6+2*\sep, 0) -- (C) -- (1, \h) node[anchor = south]{$1$};
	
	\draw (7 + 3*\sep, 0) -- (10 + 4*\sep, \h) node[anchor = south]{$1$};
	\draw (8 + 3*\sep, 0) -- (11 + 4*\sep, \h) node[anchor = south]{$2$};
	\draw (9 + 3*\sep, 0) -- (12 + 4*\sep, \h) node[anchor = south]{$3$};
	
	\path (11 + 4*\sep, 0) -- (14 + 5*\sep, \h) coordinate[pos=0.19] (F);
	\path (11 + 4*\sep, 0) -- (F) -- (0,\h) coordinate[pos = 0.35] (G);
	
	\path (12 + 4*\sep, 0) -- (15 + 5*\sep, \h) coordinate[pos=0.19] (E);
	\path (12 + 4*\sep, 0) -- (E) -- (1,\h) coordinate[pos = 0.35] (H);
	\draw (12 + 4*\sep, 0) -- (E) -- (H) -- (9+3*\sep,\h) node[anchor = south]{$1$};
	\path (13 + 4*\sep, 0) -- (16 + 5*\sep, \h) coordinate[pos=0.19] (D);
	\draw (13 + 4*\sep, 0) -- (D) -- (2, \h) node[anchor = south]{$2$};
	\path (10 + 3*\sep, 0) -- (13 + 4*\sep, \h) coordinate[pos=0.27] (I) node[anchor = south]{$0$};
	\draw (10 + 3*\sep, 0) -- (I) -- (G) -- (8+3*\sep,\h) node[anchor = south]{$0$};
	\draw (11 + 4*\sep, 0) -- (F) -- (I) --  (13 + 4*\sep, \h) node[anchor = south]{$0$};
	
	\draw[brown] (-0.5, \h) rectangle (2.5 , \h + 1);
	\draw[brown] (2.5+\sep, \h) rectangle (5.5 + \sep, \h + 1);
	\draw[brown] (5.5 + 2*\sep, \h) rectangle (7.5 + 2*\sep , \h + 1);
	\draw[brown] (7.5 + 3*\sep, \h) rectangle (9.5 + 3*\sep , \h + 1);
	\draw[brown] (9.5 + 4*\sep, \h) rectangle (13.5 + 4*\sep , \h + 1);
\end{braid}.
\end{equation}
Finally, \cref{C:Degen} shows that if we set $e = p = 2$, there is another homomorphism $\overline{\theta} \in \Hom_{R_n^\Lambda}(S^\blam, S^\bmu)$. Computing this homomorphism in the cyclotomic quiver Hecke algebra setting requires using the Brundan--Kleshchev isomorphism twice to switch between the cyclotomic Hecke algebra and the cyclotomic quiver Hecke algebra settings, which is a lengthy computation. The homomorphism space can instead be computed directly, and in this case we find that there is a homomorphism of degree 7 which sends $v_{\t^{\blam}}$ to $v_{\t'}$, where $\t' \in \Std(\bmu)$ is defined as
\begin{equation*}
\t' = \left(\,\ShadedTableau{{3,7,14}}{4}, \quad \ShadedTableau[(3,0),(2,-1)]{{1,2,5},{4,6},{11,13}}{4}, \quad\ShadedTableau{{8,9,10,12}, {15,16,17}}{4}\,\right).
\end{equation*}
Here we have elected to shade the nodes $N$ such that $\t_\bmu^{*}(N) \neq \t'(N)$. In this case $\overline{\theta}$ sends the Specht generator $v_{\t^{\blam}}$ to a monomial, but this does not happen in general, even if $\theta$ does send the Specht generator to a monomial.
\end{eg}

The proof of \cref{T:MainOneRow} we give here differs slightly from the one given in~\cite{lm14}. On the one hand, this will make it possible to prove the explicit description of the homomorphisms given above, which is a new result. On the other hand, the proof here resembles more closely the proof of the more general statement in the next section.

Recall our definitions for $\blam, \bmu, \bnu$ at the beginning of the section and let $\t \in \Std(\bnu)$. Let $\bsig$ be either $\blam$ or $\bmu$. Let $\t_\bsig \in \Std(\bnu)$ be a standard tableau such that $(\t_\bsig)_{\downarrow n} = \t^\bsig$. The following notation is borrowed from~\cite{lm14}.
\begin{align*}
\Std_\t^\gamma(\bnu) &=\{\s\in\Std(\bnu)\mid\s^{-1}(n+g)=\t^{-1}(n+g)\text{ for }1\le g\le\gamma\}, \\
\StdCheck &=\{\s\in\Std(\bnu)\mid\Shape(\s_{\downarrow n}) \gdom \bsig\}, \\
\StdHat & =\{\s\in\Std(\bnu)\mid \Shape(\s_{\downarrow n}) \not\ledom \bsig\},\\
\Std_n(\bnu)&=\{\t\in\Std(\bnu)\mid\t_{\downarrow n}=\t^{\Shape(\t_{\downarrow n})}\},\\
\IniNu[\pmb{\bsig}]&=\{\t\in\Std_n(\bnu)\mid\res_\t(n+g)=\res_{\t_{\bsig}^\bnu}(n+g),
               \text{ for }1\le g\le\gamma\},
\end{align*}
and
\begin{align*}
  \check{S}^\bnu_{\t_{\bsig}} &= \<v_\s \mid \s\in\Std_{\t_{\bsig}}^\gamma(\bnu)\cup\StdCheck[\bsig]\>_\Z, &
\check{S}^\bnu_\bsig &= \<v_\s \mid \s\in\StdCheck[\bsig]\>_\Z, \\
\hat{S}^{\bnu}_{\t_{\bsig}} &= \<v_\s \mid \s\in\Std_{\t_{\bsig}}^\gamma(\bnu)\cup\StdHat[\bsig]\>_\Z, &
\hat{S}^{\bnu}_{\bsig} &= \<v_\s \mid \s\in\StdHat[\bsig]\>_\Z.
\end{align*}

Embed $R_n$ in $R_{n+\gamma}$ in the natural way, so that $v \in R_n$ is mapped to the image of $v \otimes 1$ under the map $R_{n} \otimes R_{\gamma} \hookrightarrow R_{n + \gamma}$. Informally, $R_n$ acts on the first $n$ strings of $R_{n+\gamma}$.  Moreover, the embedding descends to the cyclotomic quotients $R_{n}^\Lambda \subset R_{n+\gamma}^\Lambda$. To prove \cref{T:MainOneRow} we will need the following lemma from \cite{lm14}, based on work by Ellers and Murray~\cite{EM}.

\begin{lem}\emph{\textbf{(\cite[Lemma 2.24]{lm14})}} \label{L:EllersMurray} 
For $\bnu\in \Multipart_{n+\gamma}^\ell$ and $\blam,\bmu\in \Multipart_{n}^\ell$  with $\blam,\bmu\subset\bnu$, let $\t_{\blam},\t_{\bmu}\in\Std(\bnu)$
  be standard tableaux such that $(\t_{\blam})_{\downarrow n}=\tlam$,
  $(\t_{\bmu})_{\downarrow n}=\tmu$ and $\t_{\bmu}\gdom\t_{\blam}$. Suppose that there
  exists a homogeneous element $L\in R_{n+\gamma}^\Lambda$ such that $L$ commutes with $R_n^\Lambda$ and
 $$\textbf{(A)} \quad \check{S}^\bnu_{\t_{\blam}}L\subseteq \hat{S}^{\bnu}_{\t_{\bmu}}, \quad \textbf{(B)} \quad
   \check{S}^\bnu_\blam L\subseteq \hat{S}^{\bnu}_\bmu\quad\text{and}\quad \textbf{(C)} \quad
   v_{\t_{\blam}} L\notin \hat{S}^{\bnu}_\bmu.$$
 Then $\Hom_{R_n^\Lambda}(S^\blam\<d\>,S^\bmu)\ne0$, where
 $d=(\deg\t_{\blam}-\deg\tlam) -(\deg\t_{\bmu} - \deg\tmu) +\deg L$.
 \end{lem}
 
\begin{rem}
	In this text we use $\t_\blam^\bnu$ and $\t_\bmu^\bnu$ as the tableaux $\t_\blam, \t_\bmu$ in the lemma.
\end{rem}

From now on we label the residues in $[\mu^*]$ and $[\lambda^*]$ as $i_1, i_2, \dots, i_\gamma$, so that $i_{j} = i_{j-1} + 1 \text{ (mod $e$)}$.
\begin{lem} \label{L:OneRowImmobile}
Let $v \in S^\bnu$ be a standard monomial such that $\t_v \in \Std_\blam^e(\bnu)$. Suppose that string $s$ of residue $i_1$ reaches a node $N$ which is not in a shape $[\xi^j]$ for any $j \in \{0, \dots, c\}$. Then $s$ is an immobile string.
\end{lem}

\begin{proof} Draw a diagram of all $i_{1}$-strings in $v$. The condition $\t_v \in \Std_\blam^e(\bnu)$ implies that when we draw only the strings of residue $i_{1}$ of the diagram for $v$ we obtain
\begin{equation}
    \begin{braid}
	\def\h{5};
	\def\sep{2}
	\def\a{7+\sep};
	\def\dashpart{0.9}
	\def\hrect{1.15}
	
	\draw (0, 0) -- (0, \h);
	\draw (\sep/2, 0) node{$\cdots$};
	\draw (\sep/2, \h) node{$\cdots$};
	\draw (\sep, 0) -- (\sep, \h);
	\draw (\sep+1, 0) --(\sep+2, \h);
	\draw(2*\sep+1, 0) -- (2*\sep+2, \h)node[color=black, anchor=south]{$N$};
	\draw(\sep + 1 + \sep/2, 0) node{$\cdots$};
	\draw(\sep + 2 + \sep/2, \h) node{$\cdots$};
	\draw(3*\sep + 1, 0) -- (3*\sep+2, \h) ;
	\draw(2.5*\sep + 1, 0) node{$\cdots$};
	\draw(2.5*\sep + 2, \h) node{$\cdots$};
	\draw(3*\sep+2, 0) node[anchor=north] {$n+1$} -- (\sep + 1, \h);
    \draw(\sep+1+\sep/2,-0.5)node{$\underbrace{\hspace*{10mm}}_{k}$};
\end{braid}.
\end{equation}
Consider only the upper part of the diagram, so that the lowest crossing is the one between $s$ and $n+1$. Then if the crossing is broken, string $n+1$ will be clearly immobile (as it does not intersect other strings in this part of the diagram) and reach node $N$. But $N$ is not inside a shape of $[\xi^j]$ for any $j$, so that no tableau with $n+1$ in node $N$ can be standard. The diagram must then equal zero. It is clear that if the upper part of the diagram equals zero, then the whole diagram must equal zero. It follows that breaking the $(i_{1}, i_{1})$-crossing of string $s$ always gives zero, or in other words, $s$ is immobile.
\end{proof}

The following lemma is the main computational tool in the proof of \cref{T:MainOneRow}. 

\begin{lem} \label{L:MainOneRow} Let $k < e$. Let $i_1, \dots, i_k \in I$ be residues so that $i_{j} = i_{j-1} + 1$ \emph{(mod $e$)}. Then
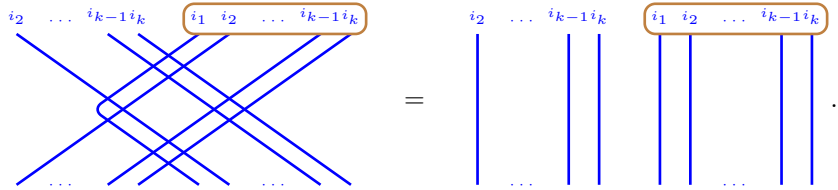
\begin{equation*} 
    \begin{braid}
	\def\h{5};
	\def\sep{2}
	\def\a{5+\sep};
	\def\dashpart{0.9}
	
	\path (1,0) coordinate (A) -- (1+\a, \h) coordinate (B);
	\path (A) -- (B) coordinate[pos=0.5] (C);
	\path (1+\a,0) coordinate (D) -- (1, \h) coordinate (E);
	\path (D) -- (E) coordinate[pos=0.5] (F);
	\draw[rounded corners] (D) -- (C) -- (B) node[anchor=south]{$i_{1}$};
	\draw(2, 0) -- (2 + \a, \h) node[anchor=south]{$i_2$} ;
	\draw (2.5 + \sep/2, 0) node {$\cdots$};
	\draw (2.5 + \sep/2, \h) node[ anchor = south] {$\cdots$};
	
	\draw (3+\sep, 0) -- (3+\sep + \a, \h) node[anchor=south]{$i_{k - 1}$} ;
	\draw (4+\sep, 0) -- (4+\sep + \a, \h) node[anchor=south]{$i_{k}$} ;

	\draw (2 + \a, 0) -- (2, \h) node[anchor=south]{$i_2$} ;
    \draw (3 + \a + \sep, 0) -- (3+\sep, \h) node[anchor=south]{$i_{k-1}$} ;
    \draw (4 + \a + \sep, 0) -- (4+\sep, \h) node[anchor=south]{$i_k$} ;
	
	\draw (2.5 +\a+ \sep/2, 0) node {$\cdots$};
	\draw (2.5 +\a+ \sep/2, \h) node[ anchor = south] {$\cdots$};
	
	\draw[brown] (\a + 0.5, \h) rectangle (2*\a + \sep - 0.5 , \h + 1);

\end{braid} \quad =\quad
    \begin{braid}
	\def\h{5};
	\def\sep{2}
	\def\a{5+\sep};
	\def\dashpart{0.8}
	
	\draw(1+\a, 0) -- (1 + \a, \h) node[anchor=south]{$i_1$} ;
	\draw(2+\a, 0) -- (2 + \a, \h) node[anchor=south]{$i_2$} ;
	\draw (2.5 + \sep/2, 0) node {$\cdots$};
	\draw (2.5 + \sep/2, \h) node[ anchor = south] {$\cdots$};
	
	\draw (3+\a+\sep, 0) -- (3+\sep + \a, \h) node[anchor=south]{$i_{k - 1}$} ;
	\draw (4+\a+\sep, 0) -- (4+\sep + \a, \h) node[anchor=south]{$i_{k}$} ;

	\draw (2, 0) -- (2, \h) node[anchor=south]{$i_2$} ;
    \draw (3+\sep, 0) -- (3+\sep, \h) node[anchor=south]{$i_{k-1}$} ;
    \draw (4+\sep, 0) -- (4+\sep, \h) node[anchor=south]{$i_{k}$} ;
	
	\draw (2.5 +\a+ \sep/2, 0) node{$\cdots$};
	\draw (2.5 +\a+ \sep/2, \h) node[ anchor = south] {$\cdots$};
	
	\draw[brown] (\a + 0.5, \h) rectangle (2*\a + \sep - 0.5 , \h + 1);

\end{braid}.
\end{equation*}
\end{lem}

\begin{proof} We argue by induction on $s$ that 
\begin{equation}
    \begin{gathered}
    \begin{braid}
	\def\h{5};
	\def\sep{2}
	\def\a{5+\sep};
	\def\dashpart{0.9}
	
	\path (1,0) coordinate (A) -- (1+\a, \h) coordinate (B);
	\path (A) -- (B) coordinate[pos=0.5] (C);
	\path (1+\a,0) coordinate (D) -- (1, \h) coordinate (E);
	\path (D) -- (E) coordinate[pos=0.5] (F);
	\draw[rounded corners] (D) -- (C) -- (B) node[anchor=south]{$i_{1}$};
	\draw(2, 0) -- (2 + \a, \h) node[anchor=south]{$i_2$} ;
	\draw (2.5 + \sep/2, 0) node {$\cdots$};
	\draw (2.5 + \sep/2, \h) node[ anchor = south] {$\cdots$};
	
	\draw (3+\sep, 0) -- (3+\sep + \a, \h) node[anchor=south]{$i_{k - 1}$} ;
	\draw (4+\sep, 0) -- (4+\sep + \a, \h) node[anchor=south]{$i_{k}$} ;

	\draw (2 + \a, 0) -- (2, \h) node[anchor=south]{$i_2$} ;
    \draw (3 + \a + \sep, 0) -- (3+\sep, \h) node[anchor=south]{$i_{k-1}$} ;
    \draw (4 + \a + \sep, 0) -- (4+\sep, \h) node[anchor=south]{$i_k$} ;
	
	\draw (2.5 +\a+ \sep/2, 0) node {$\cdots$};
	\draw (2.5 +\a+ \sep/2, \h) node[ anchor = south] {$\cdots$};
	
	\draw[brown] (\a + 0.5, \h) rectangle (2*\a + \sep - 0.5 , \h + 1);

    \end{braid}
    = \\
    \begin{braid}
	\def\h{5};
	\def\sep{2}
	\def\a{9+\sep};
	\def\dashpart{0.9}
	\def\hrect{1.15}

	\draw (2,0) -- (2,\h) node[anchor=south]{$i_{2}$};
	
	\draw (2.5 + \sep/2, 0) node {$\cdots$};
	\draw (2.5 + \sep/2, \h) node[anchor=south] {$\cdots$};
	\draw (5.5 + \sep + \sep/2, 0) node {$\cdots$};
	\draw (5.5 + \sep + \sep/2, \h) node[anchor=south] {$\cdots$};
	
	\draw (3+\sep,0) -- (3+\sep,\h) node[anchor=south]{$i_{s-1}$}; 
	
	\draw (4+\sep,0) -- (4+\a+\sep,\h) node[anchor=south]{$i_{s}$}; 
	\draw (4+\a+\sep,0) -- (4+\sep,\h) node[anchor=south]{$i_{s}$}; 
	
	\draw (5+\sep,0) -- (5+\a+\sep,\h) node[anchor=south]{$i_{s+1}$}; 
	\draw (5+\a+\sep,0) -- (5+\sep,\h) node[anchor=south]{$i_{s+1}$}; 

	\draw (2.5 + \a + \sep/2 , 0) node {$\cdots$};
	\draw (2.5 + \a + \sep/2, \h) node[anchor=south] {$\cdots$};
	\draw (5.5 + \a + \sep + \sep/2 , 0) node {$\cdots$};
	\draw (5.5 + \a + \sep + \sep/2, \h) node[anchor=south] {$\cdots$};
	
	\draw (6+2*\sep,0) -- (6+\a+2*\sep,\h) node[anchor=south]{$i_{k}$}; 
	\draw (6+\a+2*\sep,0) -- (6+2*\sep,\h) node[anchor=south]{$i_{k}$}; 
	
	\path (1, 0) -- (1 + \a, \h) coordinate[pos = 0.5] (A);
	\draw (1 + \a, 0) -- (A) -- (1 + \a, \h) node[anchor=south]{$i_{1}$};
	
	\path (2, 0) -- (2 + \a, \h) coordinate[pos = 0.5] (A);
	\draw (2 + \a, 0) -- (A) -- (2 + \a, \h) node[anchor=south]{$i_{2}$};
	
	\path (3+\sep, 0) -- (3 + \sep + \a, \h) coordinate[pos = 0.5] (A);
	\draw (3 + \sep + \a, 0) -- (A) -- (3 + \sep + \a, \h) node[anchor=south]{$i_{s-1}$};

	\draw (\a+0.1, \h/2) node[red, circle, draw, scale = 1.3]{};
	\draw[brown] (\a + 0.5, \h) rectangle (2*\a + \sep - 0.5 , \h + \hrect);

\end{braid}.
    \end{gathered}
\end{equation}

If $s = 1$, there is nothing to prove. We follow with the induction step. Applying a braid relation at the red circle gives a sum of two diagrams. In one of the diagrams, the $i_{s-1}$-string pulls past the $(i_s, i_s)$-crossing. Then the $i_{s-1}$-string can be pulled past every further crossing of the co-stubborn $i_{s+1}$-string, where \cref{E:Specht} shows that this diagram equals zero. It follows that the diagram above equals
\begin{equation} 
	\begin{braid}
	\def\h{5};
	\def\sep{2}
	\def\a{9+\sep};
	\def\dashpart{0.9}
	\def\hrect{1.15}
	
	\draw (2,0) -- (2,\h) node[anchor=south]{$i_{2}$};
	
	\draw (2.5 + \sep/2, 0) node {$\cdots$};
	\draw (2.5 + \sep/2, \h) node[anchor=south] {$\cdots$};
	\draw (5.5 + \sep + \sep/2, 0) node {$\cdots$};
	\draw (5.5 + \sep + \sep/2, \h) node[anchor=south] {$\cdots$};
	
	\draw (3+\sep,0) -- (3+\sep,\h) node[anchor=south]{$i_{s-1}$}; 
	
	\path (4+\sep,0) -- (4+\a+\sep,\h) coordinate[pos=0.5] (A); 
	\draw (4+\a+\sep,0) -- ($(A) - (0.3, 0)$) -- (4+\a + \sep,\h) node[anchor=south]{$i_{s}$}; 
	\draw (4+\sep,0) -- ($(A) - (0.4, 0)$) -- (4 + \sep,\h) node[anchor=south]{$i_{s}$}; 
	\draw (3 + \sep + \a, 0) -- ($(A) - (0.6, 0)$) -- (3 + \sep + \a, \h) node[anchor=south]{$i_{s-1}$};

	\draw (5+\sep,0) -- (5+\a+\sep,\h) node[anchor=south]{$i_{s+1}$}; 
	\draw (5+\a+\sep,0) -- (5+\sep,\h) node[anchor=south]{$i_{s+1}$}; 

	\draw (2.5 + \a + \sep/2 , 0) node {$\cdots$};
	\draw (2.5 + \a + \sep/2, \h) node[anchor=south] {$\cdots$};
	\draw (5.5 + \a + \sep + \sep/2 , 0) node {$\cdots$};
	\draw (5.5 + \a + \sep + \sep/2, \h) node[anchor=south] {$\cdots$};
	
	\draw (6+2*\sep,0) -- (6+\a+2*\sep,\h) node[anchor=south]{$i_{k}$}; 
	\draw (6+\a+2*\sep,0) -- (6+2*\sep,\h) node[anchor=south]{$i_{k}$}; 
	
	\path (1, 0) -- (1 + \a, \h) coordinate[pos = 0.5] (A);
	\draw (1 + \a, 0) -- (A) -- (1 + \a, \h) node[anchor=south]{$i_{1}$};
	
	\path (2, 0) -- (2 + \a, \h) coordinate[pos = 0.5] (A);
	\draw (2 + \a, 0) -- (A) -- (2 + \a, \h) node[anchor=south]{$i_{2}$};

	\draw (\a+0.1, \h/2) node[red, circle, draw, scale = 1.3]{};
	\draw[brown] (\a + 0.5, \h) rectangle (2*\a + \sep - 0.5 , \h + \hrect);

\end{braid}
\end{equation}
Since $k < e$, the residues $i_s$ and $i_t$ are not adjacent in the quiver $\Gamma$ for any $t \in {1, \dots, s - 2}$. Therefore, we can use \cref{E:PsiSquared} to slide the $i_s$-string to the left through the strings with residues $i_{s-2}, \dots, i_{1}$. Induction is then complete. Taking $s = k$ proves the lemma. 
\end{proof}

Following the same convention as \cite{lm14}, we define $e_\bmu^\blam = \sum_{j \in I^n} e(j_1, \dots, j_n, i_1, \dots, i_\gamma)$, where $i_h = \res_{t_\blam^\bnu}(n+h)$. This corresponds to the image of $e(i_1, \dots, i_\gamma)$ under $R_\gamma \hookrightarrow R_{n+\gamma} \twoheadrightarrow R_{n+\gamma}^\Lambda$. Letting $R_n^\Lambda \subset R_{n+\gamma}^\Lambda$ as above, $e_\bmu^\blam$ commutes with $R_n^\Lambda$.

\begin{proof}[Proof of \cref{T:MainOneRow}] We will prove the theorem by applying \cref{L:EllersMurray}. We argue that $L = y_{n+1}^c e_\bmu ^\blam$ satisfies all three conditions of the lemma. 

First we will check conditions \textbf{(A)} and \textbf{(B)} of \cref{L:EllersMurray}. Let $v_\t \in S^\bnu$ be the basis element of $S^\bnu$ corresponding to tableau $\t \in \Std(\bnu)$. Since $L$ commutes with the natural inclusions of $R_n$ and $R_\gamma$ into $R_{n+\gamma}$, we may assume that $\t \in \Std_{n}(\bnu)$. Since $L$ includes the idempotent $e^\blam_\bmu$, we can further assume that $\t \in \Std_\blam^e(\bnu)$, as $v_\t L = 0$ otherwise.

Next, draw only the $i_1$-strings in $v_\t$ that reach nodes inside a subshape of $[\mu^*]$. Note that there may exist subshapes of $[\mu^*]$ different from $[\xi^0], \dots, [\xi^c]$. Due to \cref{L:OneRowImmobile}, the $i_1$-strings that we do not draw are immobile.
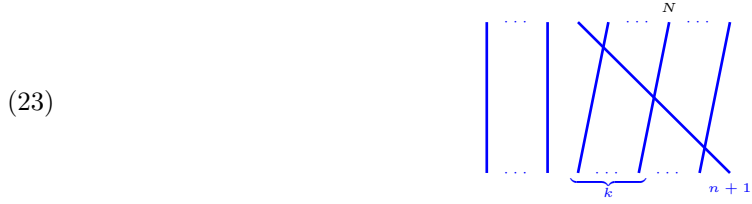
\begin{equation} \label{E:CuttableKnots}
    \begin{braid}
	\def\h{5};
	\def\sep{2}
	\def\a{7+\sep};
	\def\dashpart{0.9}
	\def\hrect{1.15}
	
	\draw (0, 0) -- (0, \h);
	\draw (\sep/2, 0) node{$\cdots$};
	\draw (\sep/2, \h) node{$\cdots$};
	\draw (\sep, 0) -- (\sep, \h);
	\draw (\sep+1, 0) --(\sep+2, \h);
	\draw(2*\sep+1, 0) -- (2*\sep+2, \h)node[color=black, anchor=south]{$N$};
	\draw(\sep + 1 + \sep/2, 0) node{$\cdots$};
	\draw(\sep + 2 + \sep/2, \h) node{$\cdots$};
	\draw(3*\sep + 1, 0) -- (3*\sep+2, \h) ;
	\draw(2.5*\sep + 1, 0) node{$\cdots$};
	\draw(2.5*\sep + 2, \h) node{$\cdots$};
	\draw(3*\sep+2, 0) node[anchor=north] {$n+1$} -- (\sep + 1, \h);
    \draw(\sep+1+\sep/2,-0.5)node{$\underbrace{\hspace*{10mm}}_{k}$};
\end{braid}
\end{equation}
Here $N \coloneq (\t_\bmu^\bnu)^{-1}(n+1)$, and we define $k$, as in the figure, to be the number of strings crossing string $n+1$ which reach a node $N' \leq N$ in $v_\t$ and which are not immobile according to \cref{L:OneRowImmobile}. Clearly, for $v_\t\in \check{S}^\bnu_{\blam}$ we have $k < c$ and for $v_\t\in \check{S}^\bnu_{\t_\blam^\bnu}$ we have $k \leq c$. In particular, for $v_{\t_\blam^\bnu}$ we have $k = c$. 

Note that string $n+1$ is strongly stubborn, so it must reach accessible nodes that are further and further to the right as its crossings are cut. If $k < c$ then for all non-zero summands in the right-hand side of $v_\t L = v_\t y_{n+1}^c e_\bmu^\blam =  \sum c_{\t'} v_{\t'}$, node $N' = (\t')^{-1}(n+1)$ must be to the right of $N$, and therefore $\Shape(\t')_{\downarrow n} \not\ledom \bmu$. This proves that condition \textbf{(B)} of \cref{L:EllersMurray} holds. If $k = c$, then for all non-zero summands in the right-hand side of $v_\t L = \sum c_{\t'} v_{\t'}$, node $N' = (\t')^{-1}(n+1)$ is to the right of or equal to $N$. If $N'$ is to the right of $N$ then $v_{\t'} \in \hat{S}^\bnu_\bmu$. If $N = N'$ in $\t'$, then strings $n+1, \dots, n+\gamma$ must all reach a node in the same row as $N$ and to the right of $N$ due to standardness. Indeed, $N$ is the leftmost node of $[\lambda^*]$, so all the other nodes in $[\lambda^*]$ must be reached by the strings $n+2, \dots, n+\gamma$. In this case, $\t' \in \Std^\gamma_{\t_\bmu^\bnu}(\bnu)$, so that $v_{\t'} \in \hat{S}^\bnu_{\t_\bmu^\bnu}$. This proves that condition \textbf{(A)} of \cref{L:EllersMurray} holds.

Finally we discuss condition \textbf{(C)}. We will show that $v_{\t_\blam^\bnu} L \notin \hat{S}^\bnu_\bmu$. For $v_{\t_\blam^\bnu}$ we have $k = c$, so all ways to cut the crossings of string $n+1$ in diagram \cref{E:CuttableKnots} are in $\hat{S}^\bnu_\bmu$ except the case where we cut the uppermost $c$ intersections. We show that this summand is not in $\hat{S}^\bnu_\bmu$, and this implies that condition \textbf{(C)} of \cref{L:EllersMurray} is satisfied because it is the only summand not in $\hat{S}^\bnu_\bmu$.

We first consider the outcome of cutting a single $(i_{1}, i_{1})$-crossing where string $n+1$ meets with a string reaching the first node in $[\xi^j]$, where $j \in [c]$. Suppose that $[\xi^j]$ has $l$ nodes. Then the situation after cutting the crossing is as follows.

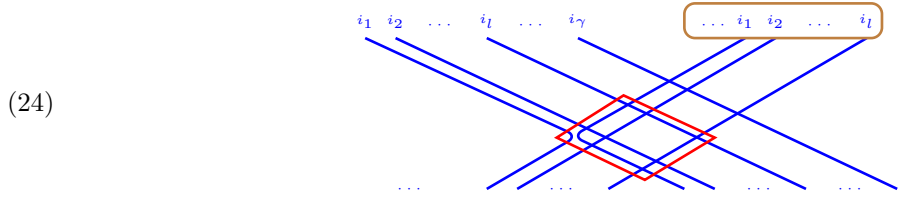
\begin{equation} \label{E:redboxdiagramA1}
\begin{braid}
	\def\h{5};
	\def\sep{2}
	\def\a{10.5+\sep};
	\def\dashpart{0.9}
	\def\hrect{1.15}
	
	\path (1,0) coordinate (A) -- (1+\a, \h) coordinate (B);
	
	\draw (1.5 + \sep/2, 0) node {$\cdots$};
	\draw (2.5+\sep/2, \h) node[anchor=south] {$\cdots$};

	\path (3 + \sep, 0) coordinate (A) -- (7.5 + 3*\sep,\h) coordinate (B) coordinate[pos=0.35] (C);
	\path (\a -1, 0) coordinate (D) -- (1,\h) coordinate (E) coordinate[pos=0.35] (F);
	\draw (A) -- (C) -- (E) node[anchor=south]{$i_1$};
	\draw (D) -- (F) -- (B) node[anchor=south]{$i_1$};
    \draw (6.5+3*\sep,\h) node[anchor=south]{$\cdots$};
	
	\draw (4 + \sep, 0) -- (8.5 + 3*\sep,\h) node[anchor=south]{$i_2$};
	\draw (5 + 2*\sep, 0) -- (9.5 + 4*\sep,\h) node[anchor=south]{$i_l$};
	
	\draw (\a, 0) -- (2, \h) node[anchor=south]{$i_2$} ;
   	\draw (1 + \a + \sep, 0) -- (3+\sep, \h) node[anchor=south]{$i_{l}$} ;
   	\draw(\a+2+2*\sep,0) -- (4 + 2*\sep , \h) node[anchor=south]{$i_\gamma$};
    
	\draw (4.5 + \sep+\sep/2, 0) node{$\cdots$};
	\draw (\a+0.5+\sep/2, 0) node {$\cdots$};
	\draw (\a+1.5+1.5*\sep, 0) node {$\cdots$};
	\draw (2.5 +\a+ \sep/2, \h) node[ anchor = south] {$\cdots$};

	\draw (4.5 + \sep, \h) node[ anchor = south] {$\cdots$};

	\draw[red,sharp corners] (10.2, 0.3) -- (7.3, 1.7) -- (9.5, 3.1) -- (12.5,1.7) -- cycle;
	\draw[brown] (\a - 1, \h) rectangle (\a + \sep + 3.5 , \h + \hrect);
\end{braid}
\end{equation}
We apply \cref{L:MainOneRow} to the area surrounded by the red polygon. We can do this despite the strings to the right of the red box, because their residues are all different from the residues of the strings in the red box. (Alternatively, we can apply \cref{L:MainSubshape}.) We obtain
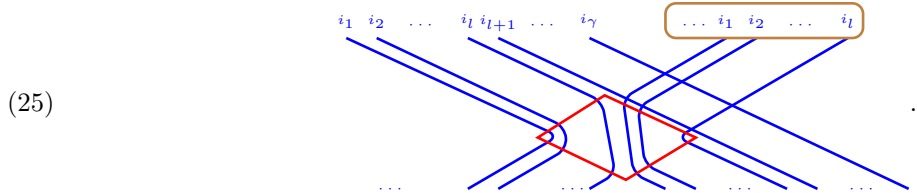
\begin{equation} \label{E:redboxdiagramA2} 
\begin{braid}
	\def\h{5};
	\def\sep{2}
	\def\a{10.5+\sep};
	\def\dashpart{0.9}
	\def\hrect{1.15}
	
	\path (1,0) coordinate (A) -- (1+\a, \h) coordinate (B);
	
	\draw (1.5 + \sep/2, 0) node {$\cdots$};
	\draw (2.5+\sep/2, \h) node[anchor=south] {$\cdots$};
	
	\path (3 + \sep, 0) coordinate (A) -- (7.5 + 3*\sep,\h) coordinate (B) coordinate[pos=0.35] (C) coordinate[pos=0.6] (G);
	\path (\a -1, 0) coordinate (D) -- (1,\h) coordinate (E) coordinate[pos=0.35] (F) coordinate[pos=0.1] (H);
	\draw (A) -- (C) -- (E) node[anchor=south]{$i_1$};
	\draw (D) -- (H) -- (G) -- (B) node[anchor=south]{$i_1$};
	
	\path (4 + \sep, 0) -- (8.5 + 3*\sep,\h) coordinate[pos = 0.265] (A) coordinate[pos = 0.54] (B)node[anchor=south]{$i_2$};
	\path (\a, 0) -- (2, \h) coordinate[pos = 0.41] (C) coordinate[pos = 0.16] (D)node[anchor=south]{$i_2$} ;
	\draw (4 + \sep, 0) -- (A) -- (C) -- (2, \h);
	\draw (\a,0) -- (D) -- (B) -- (8.5 + 3*\sep,\h);

   	\draw(\a+3+2*\sep,0) -- (5 + 2*\sep , \h) node[anchor=south]{$i_\gamma$};
   	
    	\path (5 + 2*\sep, 0) -- (9.5 + 4*\sep,\h)  coordinate[pos = 0.34] (B) coordinate[pos = 0.1] (D) node[anchor=south]{$i_l$};
   	\path (1 + \a + \sep, 0) -- (3+\sep, \h) coordinate[pos = 0.58] (C) node[anchor=south]{$i_{l}$} ;
   	\draw  (5 + 2*\sep, 0) --(D) -- (C) -- (3+\sep, \h);
   	\draw  (1 + \a + \sep, 0) --(B) -- (9.5 + 4*\sep,\h);
	\draw (2 + \a + \sep, 0) -- (4 + \sep, \h) node[yshift=5.7]{$i_{l+1}$};   	

	\draw (5.5 + \sep+\sep/2, 0) node{$\cdots$};
	\draw (\a+0.5+\sep/2, 0) node {$\cdots$};
	\draw (\a+2.5+1.5*\sep, 0) node {$\cdots$};
	\draw (2.5 +\a+ \sep/2, \h) node[ anchor = south] {$\cdots$};

	\draw (5.5 + \sep, \h) node[ anchor = south] {$\cdots$};

	\draw[red,sharp corners] (10.2, 0.3) -- (7.3, 1.7) -- (9.5, 3.1) -- (12.5,1.7) -- cycle;	
	\draw[brown] (\a - 1, \h) rectangle (\a + \sep + 3.5 , \h + \hrect);
     \draw (6.5+3*\sep,\h) node[anchor=south]{$\cdots$};
\end{braid}
.\end{equation}
We name the permutation in the above diagram $\sigma'_{\mu^*, \xi^j}$, and we use this notation to draw the above diagram simply as
$$\begin{braid}
	\def\h{12};
	\def\sep{2};
	\def\a{5+\sep};
	\def\dashpart{0.9};
	\draw (0,0) -- (0, 2);
	\draw (2, 0) node{$\dots$};
	\draw (4,0) -- (4,2);
	\draw[black, rounded corners = false] (-0.3,2) rectangle (4.3, 3.3);
	\draw[black] (2, 2.6) node[font = \fontsize{8}{8}\selectfont]{$\sigma'_{\mu^*, \xi^j}$};
	\draw (4, 3.3) -- (4, 6);
	\draw (2.5, 3.3) -- (2.5, 6);
	
	\draw( 0, 3.3) -- (0, 6);
	\draw(1.5, 3.3) -- (1.5, 6);

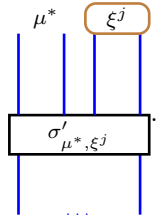
	\draw[brown] (2.2, 6) rectangle (4.3, 7);
	\draw (0.9, 6.5) node[black, font = \fontsize{8}{8}\selectfont]{$\mu^*$};
	\draw (3.25, 6.5) node[black, font = \fontsize{8}{8}\selectfont]{$\xi^j$};
\end{braid}.$$

This process can be repeated for each of the $c$ crossings between string $n+1$ and a string of residue $i_{1}$ reaching a node accommodating $n+1$. For a bipartition $(\mu^*, \xi^j)$ made out of two rows, let $\t_{\xi^j, \mu^*}$ be the unique tableau such that $(\t_{\xi^j, \mu^*})_{\downarrow l} = (0, \t^{\xi^j})$. Then, note that $\sigma'_{\mu^*, \xi^j}$ is the unique permutation such that $\t^{(\mu^*, \xi^j)} \sigma'_{\mu^*, \xi^j} = \t_{\xi^j, \mu^*}  \sigma_{\mu^*, \xi^j}$, where $\t^{(\mu^*, \xi^j)}$ is the initial tableau as usual. In this way, if we do not draw immobile strings, we reach a diagram like the following.

$$v' = \begin{braid}
	\def\h{12};
	\def\sep{2};
	\def\a{5+\sep};
	\def\dashpart{0.9};
	\draw (0,0) -- (0, 3);
	\draw (2, 0) node{$\dots$};
	\draw (4,0) -- (4,3);
	\draw[black, rounded corners = false] (-0.3,3) rectangle (4.3, 4.3);
	\draw[black] (2, 3.6) node[font = \fontsize{8}{10}\selectfont]{${\sigma'_{\mu^*, \lambda^*}}$};
	\draw (4, 4.3) -- (4, 12);
	\draw (2.5, 4.3) -- (2.5, 12);
	
	\draw( 0, 4.3) -- (0, 5.3);
	\draw(1.5, 4.3) -- (1.5, 5.3);
	\draw( 0, 6.6) -- (0, 12);
	\draw(1.5, 6.6) -- (1.5, 12);
	
	\draw (-2.5, 0) -- (-2.5, 5.3);
	\draw (-1, 0) -- (-1, 5.3);
	\draw(-2.5, 6.6) -- (-2.5, 7.6);
	\draw(-1, 6.6) -- (-1, 7.6);
	\draw(-2.5, 8.9) -- (-2.5, 12);
	\draw(-1, 8.9) -- (-1, 12);
	\draw[black, rounded corners = false] (-4.5,7.6) rectangle (-0.8, 8.9);
	\draw[black] (-2.6, 8.2) node[font = \fontsize{8}{10}\selectfont]{${\sigma'_{\mu^*, \xi^{c-2}}}$};

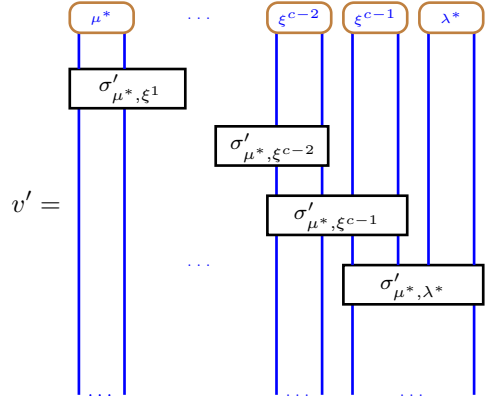
	\draw (-1.75, 0) node{$\dots$};
	\draw[black, rounded corners = false] (-2.8,5.3) rectangle (1.8, 6.6);
	\draw[black] (-0.5, 5.9) node[font = \fontsize{8}{10}\selectfont]{${\sigma'_{\mu^*, \xi^{c-1}}}$};
	
	\draw (-7.5, 0) -- (-7.5, 9.5);
	\draw (-8.25, 0) node{$\dots$};
	\draw (-9, 0) -- (-9, 9.5);
	\draw[black, rounded corners = false] (-9.3,9.5) rectangle (-5.5, 10.8);
	\draw[black] (-7.3, 10.15) node[font = \fontsize{8}{10}\selectfont]{${\sigma'_{\mu^*, \xi^{1}}}$};
	
	\draw (-7.5, 10.8) -- (-7.5, 12);
	\draw (-8.25, 0) node{$\dots$};
	\draw (-9, 10.8) -- (-9, 12);
	
	\draw (-5, 4.3) node{$\dots$};
	\draw (-5, 12.5) node{$\dots$};
	
	\draw[brown] (-9.3, 12) rectangle (-7.2, 13);
	\draw (-8.25, 12.5) node{$\mu^*$};
	\draw[brown] (-2.8, 12) rectangle (-0.7, 13);
	\draw (-1.75, 12.5) node{$\xi^{c-2}$};
	\draw[brown] (-0.3, 12) rectangle (1.8, 13);
	\draw (0.75, 12.5) node{$\xi^{c-1}$};
	\draw[brown] (2.2, 12) rectangle (4.3, 13);
	\draw (3.25, 12.5) node{$\lambda^*$};
\end{braid}$$

While the permutations $\sigma'_{\mu^*, \xi^j}$ have been defined only for bipartitions, we abuse notation and write $\sigma'_{\mu^*}$ to refer to the permutation of the strings entering the corresponding box in $v'$. We have included a concrete example of the above diagram after this proof. The diagram $v'$ is easily seen to be a diagram for a standard tableau. 

The subdiagram corresponding to $\sigma'_{\mu^*, \lambda^*}$ is always the identity, so the last $\gamma$ strings reach $\lambda^*$, and the restriction to the first $n$ strings can be seen to be $\t_\mu^*$, as introduced after the statement of the theorem.
After coming out of a subdiagram for $\sigma'_{\mu^*, \xi^{j}}$ in $v'$, a string $s$ either reaches $\xi^j$ without crossing any other strings, or it enters the subdiagram for $\sigma'_{\mu^*, \xi^{j-1}}$ next. If it reaches $\xi^j$, we say that $s$ is \emph{rightward} at  $\sigma_{\mu^*, \xi^{j}}$. Otherwise we say it is \emph{leftward} at $\sigma_{\mu^*, \xi^{j}}$. Strings only may only intersect in $\sigma_{\mu^*, \xi^{j}}$ if one of them is leftward and the other is rightward. Since a rightward string at $\sigma_{\mu^*, \xi^{j}}$ does not intersect any other strings after the subdiagram for $\sigma_{\mu^*, \xi^{j}}$ before reaching $\xi^j$, it follows that there are no double crossings between any two strings. It follows from this and the fact that $v'$ is the diagram for a standard tableau that $v'$ is a standard monomial of $S^\bnu$.

Moreover, a string $s$ in $v'$ can be rightward at most at one subdiagram $\sigma_{\mu^*, \xi^{j}}$. If $s$ is rightward at some $\sigma_{\mu^*, \xi^{j}}$, then it does not intersect any strings after $\sigma_{\mu^*, \xi^{j}}$. Otherwise, $s$ is leftward at each $\sigma_{\mu^*, \xi^{j}}$. It follows that the pattern
$$
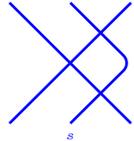
\begin{braid}
        	\draw (0,0) -- (4, 4);
        	\draw (4,0)  -- (0,4);
        	\draw (2, 0) node[anchor=north]{$s$} -- (4, 2) -- (2, 4);
\end{braid},
$$
cannot be found inside $v'$, as it requires that $s$ be rightward first (lower down) and then leftward afterwards (higher up). With this we have shown that $v'$ is a basis element of $S^\bnu$. Moreover, $v' \in \hat{S}^\bnu_{\t_\bmu^\bnu} \setminus \hat{S}^\bnu_{\bmu}$, so that condition \textbf{(C)} holds. Moreover, a restriction to the first $n$ strings of $v'$ yields $\theta(v_{\t^\blam}) = v_{\t_\bmu^*}$, completing the proof.


We refer to~\cite{lm14} for the proof of the formula for the degree of the map and a proof of its positivity. Alternatively, the proof of \cref{T:MainStraight} includes a proof of the degree formula which specializes to the degree formula in \cref{T:MainOneRow} in the single-row case, although the arguments required to establish positivity are much more involved in the more general case. 
\end{proof}

\begin{eg} We replicate the diagram in \cref{E:egDiagramRow} with dashed green lines used for immobile strings and red rectangular outlines highlighting each $\sigma'_{\mu^*, \xi^j}$. Note that the extra strings reaching $[\lambda^*]$ were not drawn in \cref{E:egDiagramRow}, but we draw them here in order to highlight $\sigma'_{\mu^*, \lambda^*}$.

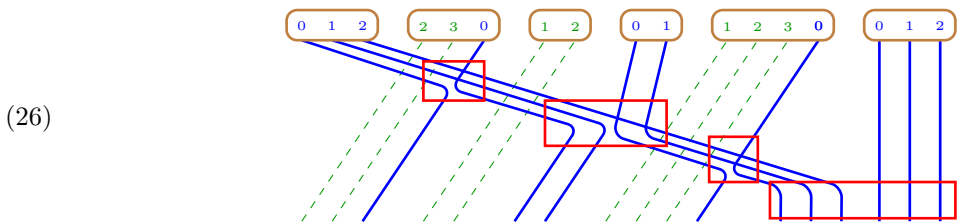
\begin{equation}
 \begin{braid}
	\def\h{6};
	\def\sep{1}
	\def\a{15+\sep};
	\def\dashpart{0.9}
	
	\draw[green!60!black!, dashed, thin] (0, 0) -- (3 + \sep, \h) node[anchor = south]{$2$};
	\draw[green!60!black!, dashed, thin] (1, 0) -- (4 + \sep, \h) node[anchor = south]{$3$};
	\path (2, 0) -- (5 + \sep, \h) coordinate[pos = 0.73] (A);
	\draw (2, 0) -- (A) -- (0, \h) node[anchor = south]{$0$} ;
	\draw[green!60!black!, dashed, thin] (3+\sep, 0) -- (6 + 2*\sep, \h) node[anchor = south]{$1$};
	\draw[green!60!black!, dashed, thin] (4+\sep, 0) -- (7 + 2*\sep, \h) node[anchor = south]{$2$};
	\path (5+2*\sep, 0) -- (8 + 3*\sep, \h) coordinate[pos = 0.52] (B);
	\draw (5+2*\sep, 0) -- (B) -- (A) -- (5+\sep,\h) node[anchor = south]{$0$};
	\path (6+2*\sep, 0) -- (9 + 3*\sep, \h) coordinate[pos=0.525] (C);
	\draw (6+2*\sep, 0) -- (C) -- (1, \h) node[anchor = south]{$1$};
	
	\draw[green!60!black!, dashed, thin] (7 + 3*\sep, 0) -- (10 + 4*\sep, \h) node[anchor = south]{$1$};
	\draw[green!60!black!, dashed, thin] (8 + 3*\sep, 0) -- (11 + 4*\sep, \h) node[anchor = south]{$2$};
	\draw[green!60!black!, dashed, thin] (9 + 3*\sep, 0) -- (12 + 4*\sep, \h) node[anchor = south]{$3$};
	
	\path (11 + 4*\sep, 0) -- (14 + 5*\sep, \h) coordinate[pos=0.19] (F);
	\path (11 + 4*\sep, 0) -- (F) -- (0,\h) coordinate[pos = 0.35] (G);
	
	\path (12 + 4*\sep, 0) -- (15 + 5*\sep, \h) coordinate[pos=0.19] (E);
	\path (12 + 4*\sep, 0) -- (E) -- (1,\h) coordinate[pos = 0.35] (H);
	\draw (12.75 + 4*\sep, 0) -- (E) -- (H) -- (9+3*\sep,\h) node[anchor = south]{$1$};
	\path (13 + 4*\sep, 0) -- (16 + 5*\sep, \h) coordinate[pos=0.19] (D);
	\draw (13.75 + 4*\sep, 0) -- (D) -- (2, \h) node[anchor = south]{$2$};
	\path (10 + 3*\sep, 0) -- (13 + 4*\sep, \h) coordinate[pos=0.27] (I) node[anchor = south]{$0$};
	\draw (10 + 3*\sep, 0) -- (I) -- (G) -- (8+3*\sep,\h) node[anchor = south]{$0$};
	\draw (11.75 + 4*\sep, 0) -- (F) -- (I) --  (13 + 4*\sep, \h) node[anchor = south]{$0$};
	
	\draw (14 + 5*\sep, 0) -- (14 + 5*\sep, \h) node[anchor = south]{$0$};
	\draw (15 + 5*\sep, 0) -- (15 + 5*\sep, \h) node[anchor = south]{$1$};
	\draw (16 + 5*\sep, 0) -- (16 + 5*\sep, \h) node[anchor = south]{$2$};
	
	\draw[red,sharp corners] (4,4) -- (6,4) -- (6,5.3) -- (4,5.3) -- cycle;
	\draw[red,sharp corners] (8,2.5) -- (12,2.5) -- (12,4) -- (8,4) -- cycle;
	\draw[red,sharp corners] (13.4,1.3) -- (15,1.3) -- (15,2.8) -- (13.4,2.8) -- cycle;
	\draw[red,sharp corners] (15.4,0.1) -- (21.5,0.1) -- (21.5,1.3) -- (15.4,1.3) -- cycle;

	\draw[brown] (-0.5, \h) rectangle (2.5 , \h + 1);
	\draw[brown] (2.5+\sep, \h) rectangle (5.5 + \sep, \h + 1);
	\draw[brown] (5.5 + 2*\sep, \h) rectangle (7.5 + 2*\sep , \h + 1);
	\draw[brown] (7.5 + 3*\sep, \h) rectangle (9.5 + 3*\sep , \h + 1);
	\draw[brown] (9.5 + 4*\sep, \h) rectangle (13.5 + 4*\sep , \h + 1);
	\draw[brown] (13.5 + 5*\sep, \h) rectangle (16.5 + 5*\sep , \h + 1);
\end{braid}
\end{equation}
\end{eg}

\begin{rem}
    This provides a different way to compute a homomorphism compared to the proof in \cite{lm14}. It is not clear that the homomorphisms described in both proofs should be identical, but we have found them to be equal (up to sign) in all the cases we have checked. If $\ell = 1$, it is known that $\dim \Hom_{R_n^\Lambda}(S_\lambda, S_\mu) = 1$ whenever $\lambda$ and $\mu$ form a Carter--Payne pair in the sense of \cref{T:Carter-Payne} (see Lyle~\cite{Lyle07} and Dixon~\cite{dixon} for this result), so that two homomorphisms may only differ by a multiplication by a scalar. 
\end{rem}

\section{Moving Straight Shapes} \label{S:Straight}

Let $\bnu$ be any multipartition, and suppose we have two shapes $[\xi], [\rho] \subset [\bnu]$. We say that $[\xi], [\rho]$ are \emph{congruent} if the sets of nodes are in correspondence, so that $[\xi]$ can be obtained from $[\rho]$ by a single translation, and the residues of the corresponding nodes match. Now let $\blam$ and $\bmu$ be multipartitions of~$n$ and suppose that
  $\bmu\gdom\blam$ and that $\bnu=\blam\cup\bmu$ is a multipartition of
  $n+\gamma$ such that $\bnu\backslash\blam$ and $\bnu\backslash\bmu$ are congruent
  removable $e$-small straight shapes. Write $[\lambda^*] = [\bnu] \backslash [\bmu]$ and $[\mu^*] = [\bnu] \backslash [\blam]$, and label the maximal removable $e$-small straight shapes which are subshapes of $[\mu^*]$ between $[\mu^*]$ and $[\lambda^*]$ as $[\mu^*] = [\xi^0] \prec [\xi^1] \prec \dots \prec [\xi^{c-1}] \prec [\xi^c] = [\lambda^*]$. We assume throughout the section that $R^\Lambda_{n+\gamma}$ is a quiver Hecke algebra of type $A_{e-1}^{(1)}$ with $e \in \mathbb{Z}_{\geq 2} \cup \{ \infty \}$ defined over a field of characteristic $p$, or over $\mathbb{Z}$. Let $A = \{N \in \Add(\bnu)\mid [\mu^*] \prec N \preceq [\lambda^*] \}$ and let $B = \{N \in \Rem(\bnu)\mid \mu \prec N \preceq \lambda \}$. Finally, for a straight shape $[\xi]$, define $\Rank([\xi])$ to be the number of nodes in the main diagonal of $[\xi]$ (that is, the nodes in $[\xi]$ inside the same diagonal as the top-left node of $[\xi]$).

\begin{thm} \label{T:MainStraight}
Under the conditions above, there exists a non-zero homomorphism $\theta \in \Hom_{R_n^\Lambda}(S^\blam\<a-b+2d\>,S^\bmu)$, where
\begin{align*}
    a&= \sum_{N \in A} \#\{N' \in [\mu^*] \mid \res N' = \res N \},\\
    b&= \sum_{N \in B} \#\{N' \in [\mu^*] \mid \res N' = \res N \},\\
    d&= \sum_{j \in \{1, \dots, c\}} \Rank[\xi^j].
  \end{align*}
\end{thm}

The tableau $\t^*_\bmu \in \Std(\bmu)$ such that $\theta(v_{\t^\blam}) = v_{\t^*_\bmu}$ is described similarly to the previous section. Let $j \in \{1, \dots, c\}$, and let $[\xi^j_{\mu^*}] \subset [\mu^*]$ be the set of nodes inside $[\mu^*]$ that can be obtained from $[\mu^*]$ by successively removing nodes until reaching a shape with the same underlying partition as $[\xi^j]$. Let $\sigma_{\xi^j, \mu^*}$ be the permutation that exchanges the numbers appearing in $[\xi^j]$ in tableau $\t^\bnu_\blam$ with the numbers $\t_\blam^\bnu([\xi^j_{\mu^*}])$, in order. Then we have $\t_\bmu^* = \left(\t_\blam^\bnu \sigma_{\mu^*, \xi^1} \dots \sigma_{\mu^*, \xi^{c-1}} \sigma_{\mu^*, \lambda^*}\right)_{\downarrow n}$.

\begin{eg} \label{Eg:MainStraight}
Let $p = 0$, $e = 9$ and $\Lambda = (0, 4, 6)$. Suppose $\blam = ((5,5,4,3,3,1), (2,2,1), (3,3,2)) \vdash 34$ and $\bmu = ((8,7,4,3,3,1), (2,2,1), (3)) \vdash 34$, so that $\bnu = ((8,7,4,3,3,1), (2,2,1), (3,3,2)) \vdash 39$. We draw $\bnu$ with the residues corresponding to each node.
\begin{equation*}
    \left(\,\ShadedTableau[(6,0),(7,0),(8,0), (6, -1), (7, -1)]{{0,1,2,3,4,5,6,7}, {8,0,1,2,3,4,5},{7,8,0,1},{6,7,8},{5,6,7},{4}}{10}, \quad \ShadedTableau{{4,5},{3,4},{2}}{10}, \quad\ShadedTableau[(1,-1),(2,-1),(3,-1),(1,-2),(2,-2)]{{6,7,8},{5,6,7},{4,5}}{10}\,\right)
\end{equation*}
Here $[\mu^*]$ and $[\lambda^*]$ have been shaded in. It is easy to see that $a = 4, b = 6$ and $d = 4$, so that $\deg \theta = a - b + 2d = 6$. Let $\theta \in \Hom_{R_n^\Lambda}(S\<6\>, S^\bmu)$ be the homomorphism given by the theorem. We have
\begin{equation*}
    \t_\blam^\bnu = \left(\,\ShadedTableau[(6,0),(7,0),(8,0), (6, -1), (7, -1),(1,-4),(2,-4),(3,-4),(1,-5)]{{1,2,3,4,5,35,36,37}, {6,7,8,9,10,38,39},{11,12,13,14},{15,16,17},{18,19,20},{21}}{10}, \quad \ShadedTableau[(2,0),(2,-1)]{{22,23},{24,25},{26}}{10}, \quad\ShadedTableau[(1,-1),(2,-1),(3,-1),(1,-2),(2,-2)]{{27,28,29},{30,31,32},{33,34}}{10}\,\right)
\end{equation*}
where we have shaded in the shapes $[\mu^*] = [\xi^0] \prec \dots \prec [\xi^3] = [\lambda^*]$. Numbers outside these shapes do not change positions. Next we apply the permutations $\sigma_{\mu^*, \xi^j}$ in order. We have
\begin{equation*}
    \t_\blam^\bnu \sigma_{\mu^*, \xi^1} \dots \sigma_{\mu^*, \lambda^*} = \left(\,\ShadedTableau[(6,0),(7,0),(8,0), (6, -1), (7, -1),(1,-4),(2,-4),(3,-4),(1,-5)]{{1,2,3,4,5,18,19,20}, {6,7,8,9,10,21,34},{11,12,13,14},{15,16,17},{23,31,32},{25}}{10}, \quad \ShadedTableau[(2,0),(2,-1)]{{22,30},{24,33},{26}}{10}, \quad\ShadedTableau[(1,-1),(2,-1),(3,-1),(1,-2),(2,-2)]{{27,28,29},{35,36,37},{38,39}}{10}\,\right)
\end{equation*}
The tableau $\t_\bmu^*$ is obtained simply by restricting to the first $n$ positive integers, and we have $\theta(v_{\t_\blam^\bnu}) = v_{\t_\bmu^*}$.

Finally, \cref{C:Degen} shows that if we set $e = p = 3$, there is also a non-zero homomorphism $\overline{\theta} \in \Hom_{R^\Lambda_{n,3}(\mathbb{F}_3)}(S^\blam, S^\bmu)$.
\end{eg}

We will prove the theorem at the end of the section, after introducing several technical lemmas that will aid us. As a quick shorthand, we refer to the first $n$ strings in a monomial $v\in S^\bnu$ as \emph{basic} strings and to the last $\gamma$ strings as \emph{extra} strings. First we prove an analogous result to \cref{L:MainOneRow}.

\begin{lem} \label{L:MainStraight}
Suppose $[\xi]$ is an $e$-small straight shape with rows $\xi = (\xi_1, \xi_2, \dots, \xi_k)$ and residues $i_1, \dots, i_{\xi_1}$ in the first row, $i_0, \dots, i_{\xi_2 - 1}$ in the second row, and so on, up until $i_{2 - k}, \dots, i_{\xi_k - (k - 1)}$ in the last row. Consider the following diagram where the shape in the top right is $[\xi]$.
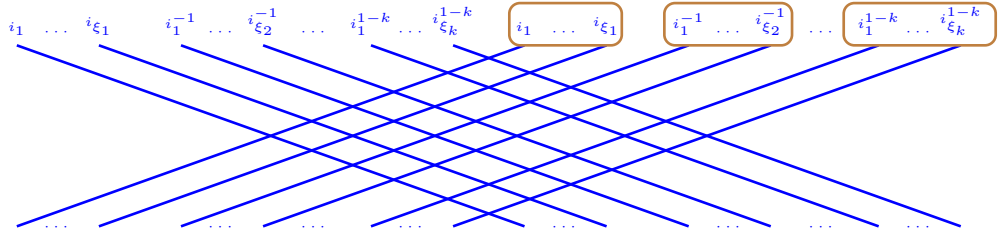
$$     \begin{braid}
	\def\h{6};
	\def\sep{1.7}
	\def\a{15+\sep};
	\def\dashpart{0.9}

	\draw (1,0) -- (1+\a, \h)node[anchor=south]{$i_{1}$};
	\draw (1.5 + \sep/2, 0) node {$\cdots$};
	\draw (1.5 + \sep/2, \h) node[ anchor = south] {$\cdots$};
	\draw (2+\sep, 0) -- (2+\sep + \a, \h) node[anchor=south]{$i_{\xi_1}$} ;
	
	\draw (3 + 2*\sep, 0) -- (3 + 2*\sep +\a, \h) node[anchor=south]{$i_{1}^{-1}$};
	\draw (3.5 + 2.5*\sep, 0)  node{$\cdots$};
	\draw (3.5 + 2.5*\sep +\a, \h) node[anchor=south]{$\cdots$};
	\draw (4 + 3*\sep, 0) -- (4 + 3*\sep +\a, \h) node[anchor=south]{$i_{\xi_2}^ {-1}$};
	
	\draw (4 + 4*\sep, 0)  node{$\cdots$};
	\draw (4 + 4*\sep +\a, \h) node[anchor=south]{$\cdots$};
	\draw (4 + 4*\sep+\a, 0)  node{$\cdots$};
	\draw (4 + 4*\sep , \h) node[anchor=south]{$\cdots$};
	
	\draw (5 + 4.5*\sep, 0) -- (5 + 4.5*\sep + \a, \h) node[anchor=south]{$i_{1}^{1-k}$};
	\draw (5.5 + 5*\sep, 0)  node{$\cdots$};
	\draw (5.5 + 5*\sep +\a, \h) node[anchor=south]{$\cdots$};
	\draw (6 + 5.5*\sep, 0) -- (6 + 5.5*\sep +\a, \h) node[anchor=south]{$i_{\xi_{k}}^{1-k}$};
	
	\draw (1 + \a, 0) -- (1, \h) node[anchor=south]{$i_1$} ;
       \draw (2 + \a + \sep, 0) -- (2+\sep, \h) node[anchor=south]{$i_{\xi_1}$} ;
	\draw (1.5 +\a+ \sep/2, 0) node {$\cdots$};
	\draw (1.5 +\a+ \sep/2, \h) node[ anchor = south] {$\cdots$};
	
	\draw (3 + 2*\sep + \a, 0) -- (3 + 2*\sep, \h) node[anchor=south]{$i_{1}^{-1}$};
	\draw (3.5 + 2.5*\sep + \a, 0)  node{$\cdots$};
	\draw (3.5 + 2.5*\sep, \h) node[anchor=south]{$\cdots$};
	\draw (4 + 3*\sep + \a, 0) -- (4 + 3*\sep, \h) node[anchor=south]{$i_{\xi_2}^{-1}$};
	
	\draw (5 + 4.5*\sep + \a, 0) -- (5 + 4.5*\sep, \h) node[anchor=south]{$i_{1}^{1-k}$};
	\draw (5.5 + 5*\sep + \a, 0)  node{$\cdots$};
	\draw (5.5 + 5*\sep, \h) node[anchor=south]{$\cdots$};
	\draw (6 + 5.5*\sep + \a, 0) -- (6 + 5.5*\sep, \h) node[anchor=south]{$i_{\xi_{k}}^{1-k}$};
	
	\draw[brown] (\a + 0.5, \h) rectangle (\a + 0.5 + 2 + \sep , \h + 1.4);
	\draw[brown] (\a + 2.2 + 2*\sep, \h) rectangle (\a + 2.5 + 2.2 + 3*\sep , \h + 1.4);
	\draw[brown] (\a + 4.7 + 4*\sep, \h) rectangle (\a + 5 + 3 + 5*\sep , \h + 1.4);
\end{braid}$$
We have used the notation $i_\square^{\text{\scalebox{0.8}{$\bigcirc$}}}$, meaning $i_{\square} + \text{\scalebox{0.8}{$\bigcirc$}}$, for graphical clarity. Suppose we apply a dot (from the bottom) to each string which reaches a node in the main diagonal of $\xi$ in the top right. Then we reach (up to sign) the following diagram.
$$     \begin{braid}
	\def\h{6};
	\def\sep{1.7}
	\def\a{15+\sep};
	\def\dashpart{0.9}
	
	\draw (1,0) -- (1, \h)node[anchor=south]{$i_{1}$};
	\draw (1.5 + \sep/2, 0) node {$\cdots$};
	\draw (1.5 + \sep/2, \h) node[ anchor = south] {$\cdots$};
	\draw (2+\sep, 0) -- (2+\sep, \h) node[anchor=south]{$i_{\xi_1}$} ;
	
	\draw (3 + 2*\sep, 0) -- (3 + 2*\sep, \h) node[anchor=south]{$i_{1}^{-1}$};
	\draw (3.5 + 2.5*\sep, 0)  node{$\cdots$};
	\draw (3.5 + 2.5*\sep +\a, \h) node[anchor=south]{$\cdots$};
	\draw (4 + 3*\sep, 0) -- (4 + 3*\sep, \h) node[anchor=south]{$i_{\xi_2}^ {-1}$};
	
	\draw (4 + 4*\sep, 0)  node{$\cdots$};
	\draw (4 + 4*\sep +\a, \h) node[anchor=south]{$\cdots$};
	\draw (4 + 4*\sep+\a, 0)  node{$\cdots$};
	\draw (4 + 4*\sep , \h) node[anchor=south]{$\cdots$};
	
	\draw (5 + 4.5*\sep, 0) -- (5 + 4.5*\sep, \h) node[anchor=south]{$i_{1}^{1-k}$};
	\draw (5.5 + 5*\sep, 0)  node{$\cdots$};
	\draw (5.5 + 5*\sep +\a, \h) node[anchor=south]{$\cdots$};
	\draw (6 + 5.5*\sep, 0) -- (6 + 5.5*\sep, \h) node[anchor=south]{$i_{\xi_{k}^{1-k}}$};
	
	\draw (1 + \a, 0) -- (1+ \a, \h) node[anchor=south]{$i_1$} ;
       \draw (2 + \a + \sep, 0) -- (2+\sep+ \a, \h) node[anchor=south]{$i_{\xi_1}$} ;
	\draw (1.5 +\a+ \sep/2, 0) node {$\cdots$};
	\draw (1.5 +\a+ \sep/2, \h) node[ anchor = south] {$\cdots$};
	
	\draw (3 + 2*\sep + \a, 0) -- (3 + 2*\sep+ \a, \h) node[anchor=south]{$i_{1}^{-1}$};
	\draw (3.5 + 2.5*\sep + \a, 0)  node{$\cdots$};
	\draw (3.5 + 2.5*\sep, \h) node[anchor=south]{$\cdots$};
	\draw (4 + 3*\sep + \a, 0) -- (4 + 3*\sep+ \a, \h) node[anchor=south]{$i_{\xi_2}^{-1}$};
	
	\draw (5 + 4.5*\sep + \a, 0) -- (5 + 4.5*\sep+ \a, \h) node[anchor=south]{$i_{1}^{1-k}$};
	\draw (5.5 + 5*\sep + \a, 0)  node{$\cdots$};
	\draw (5.5 + 5*\sep, \h) node[anchor=south]{$\cdots$};
	\draw (6 + 5.5*\sep + \a, 0) -- (6 + 5.5*\sep+ \a, \h) node[anchor=south]{$i_{\xi_{k}}^{1-k}$};
	
	\draw[brown] (\a + 0.5, \h) rectangle (\a + 0.5 + 2 + \sep , \h + 1.4);
	\draw[brown] (\a + 2.2 + 2*\sep, \h) rectangle (\a + 2.5 + 2.2 + 3*\sep , \h + 1.4);
	\draw[brown] (\a + 4.7 + 4*\sep, \h) rectangle (\a + 5 + 3 + 5*\sep , \h + 1.4);
\end{braid}.$$

\end{lem}

\begin{proof}We will perform an induction on the number of rows $k$. The case $k = 1$ follows easily from \cref{L:MainOneRow}, as the dot on the $i_1$-string breaks the unique $(i_1,i_1)$-crossing to reach the main diagram in that lemma. Suppose we have proved the theorem up to $k-1$ rows. In the next diagram, we focus on the subdiagram above the orange dotted line, which we call $v$.
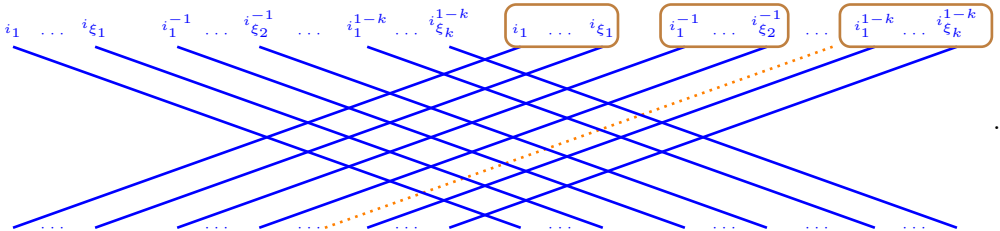
$$    \begin{braid}
	\def\h{6};
	\def\sep{1.7}
	\def\a{15+\sep};
	\def\dashpart{0.9}
	
	\draw (1,0) -- (1+\a, \h)node[anchor=south]{$i_{1}$};
	\draw (1.5 + \sep/2, 0) node {$\cdots$};
	\draw (1.5 + \sep/2, \h) node[ anchor = south] {$\cdots$};
	\draw (2+\sep, 0) -- (2+\sep + \a, \h) node[anchor=south]{$i_{\xi_1}$} ;
	
	\draw (3 + 2*\sep, 0) -- (3 + 2*\sep +\a, \h) node[anchor=south]{$i_{1}^{-1}$};
	\draw (3.5 + 2.5*\sep, 0)  node{$\cdots$};
	\draw (3.5 + 2.5*\sep +\a, \h) node[anchor=south]{$\cdots$};
	\draw (4 + 3*\sep, 0) -- (4 + 3*\sep +\a, \h) node[anchor=south]{$i_{\xi_2}^ {-1}$};
	
	\draw (4 + 4*\sep, 0)  node{$\cdots$};
	\draw (4 + 4*\sep +\a, \h) node[anchor=south]{$\cdots$};
	\draw (4 + 4*\sep+\a, 0)  node{$\cdots$};
	\draw (4 + 4*\sep , \h) node[anchor=south]{$\cdots$};
	
	\draw[orange, dotted] (4.8 + 3.8*\sep, 0) -- (4.8 + 3.8*\sep +\a, \h);
	
	\draw (5 + 4.5*\sep, 0) -- (5 + 4.5*\sep + \a, \h) node[anchor=south]{$i_{1}^{1-k}$};
	\draw (5.5 + 5*\sep, 0)  node{$\cdots$};
	\draw (5.5 + 5*\sep +\a, \h) node[anchor=south]{$\cdots$};
	\draw (6 + 5.5*\sep, 0) -- (6 + 5.5*\sep +\a, \h) node[anchor=south]{$i_{\xi_{k}}^{1-k}$};
	
	\draw (1 + \a, 0) -- (1, \h) node[anchor=south]{$i_1$} ;
       \draw (2 + \a + \sep, 0) -- (2+\sep, \h) node[anchor=south]{$i_{\xi_1}$} ;
	\draw (1.5 +\a+ \sep/2, 0) node {$\cdots$};
	\draw (1.5 +\a+ \sep/2, \h) node[ anchor = south] {$\cdots$};
	
	\draw (3 + 2*\sep + \a, 0) -- (3 + 2*\sep, \h) node[anchor=south]{$i_{1}^{-1}$};
	\draw (3.5 + 2.5*\sep + \a, 0)  node{$\cdots$};
	\draw (3.5 + 2.5*\sep, \h) node[anchor=south]{$\cdots$};
	\draw (4 + 3*\sep + \a, 0) -- (4 + 3*\sep, \h) node[anchor=south]{$i_{\xi_2}^{-1}$};
	
	\draw (5 + 4.5*\sep + \a, 0) -- (5 + 4.5*\sep, \h) node[anchor=south]{$i_{1}^{1-k}$};
	\draw (5.5 + 5*\sep + \a, 0)  node{$\cdots$};
	\draw (5.5 + 5*\sep, \h) node[anchor=south]{$\cdots$};
	\draw (6 + 5.5*\sep + \a, 0) -- (6 + 5.5*\sep, \h) node[anchor=south]{$i_{\xi_{k}}^{1-k}$};
	
	\draw[brown] (\a + 0.5, \h) rectangle (\a + 0.5 + 2 + \sep , \h + 1.4);
	\draw[brown] (\a + 2.2 + 2*\sep, \h) rectangle (\a + 2.5 + 2.2 + 3*\sep , \h + 1.4);
	\draw[brown] (\a + 4.7 + 4*\sep, \h) rectangle (\a + 5 + 3 + 5*\sep , \h + 1.4);
\end{braid}.$$
We argue that the last $\xi_k$ strings, with residues $i_{1}^{1-k}, \dots, i_{\xi_k}^{1-k}$, are all immobile in this part of the diagram. Indeed, these strings are all strongly stubborn, and for any $v' \in \mathcal{V}_v$ where one of these strings, $s$, has been cut, $s$ must reach a node in the first $k-1$ rows of $[\xi]$. However, the residues of the nodes at the end of each of these rows are different from the residues of the last $\xi_k$ strings, so that some $s' < s$ must reach each of these nodes in $v'$. But by standardness, the numbers must increase along each row, so we conclude that there are no $v' \in \mathcal{V}_v$ such that $s$ has been cut. Since $v$ is fully commutative, \cref{L:ImmobileIgnore} shows that we may ignore the last $\xi_k$ strings in $v$, and then we may apply the inductive hypothesis. Looking at the full diagram after this, we reach
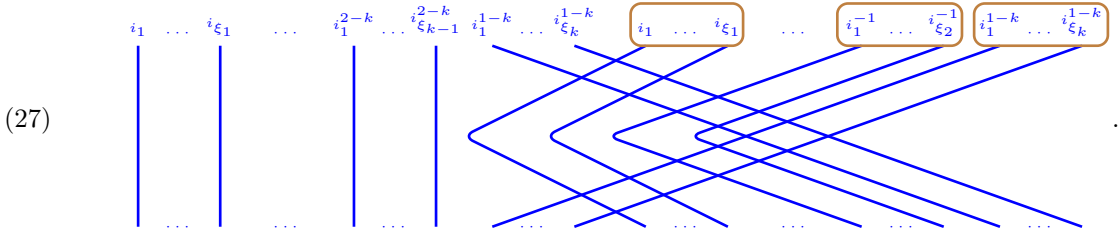
\begin{equation} \label{E:MainProof1}
\begin{braid}
	\def\h{6};
	\def\sep{1.7}
	\def\a{15+\sep};
	\def\dashpart{0.9}
	
	\draw (1,0) -- (1, \h)node[anchor=south]{$i_{1}$};
	\draw (1.5 + \sep/2, 0) node {$\cdots$};
	\draw (1.5 + \sep/2, \h) node[ anchor = south] {$\cdots$};
	\draw (2+\sep, 0) -- (2+\sep, \h) node[anchor=south]{$i_{\xi_1}$} ;
	
	\draw (2.5 + 2*\sep, 0)  node{$\cdots$};
	\draw (2.5 + 2*\sep +\a, \h) node[anchor=south]{$\cdots$};
	\draw (2.5 + 2*\sep+\a, 0)  node{$\cdots$};
	\draw (2.5 + 2*\sep , \h) node[anchor=south]{$\cdots$};

	\draw (3 + 3*\sep, 0) -- (3 + 3*\sep, \h) node[anchor=south]{$i_{1}^{2-k}$};
	\draw (3.5 + 3.5*\sep, 0)  node{$\cdots$};
	\draw (3.5 + 3.5*\sep, \h) node[anchor=south]{$\cdots$};
	\draw (4 + 4*\sep, 0) -- (4 + 4*\sep, \h) node[anchor=south]{$i_{\xi_{k-1}}^ {2-k}$};

	\draw (5 + 4.5*\sep, 0) -- (5 + 4.5*\sep + \a, \h) node[anchor=south]{$i_{1}^{1-k}$};
	\draw (5.5 + 5*\sep, 0)  node{$\cdots$};
	\draw (5.5 + 5*\sep +\a, \h) node[anchor=south]{$\cdots$};
	\draw (6 + 5.5*\sep, 0) -- (6 + 5.5*\sep +\a, \h) node[anchor=south]{$i_{\xi_{k}}^{1-k}$};
	
	\draw (1 + \a, 0) -- (\a - 5, 0.5*\h) -- (1+\a, \h) node[anchor=south]{$i_1$};
       	\draw (2 + \a + \sep, 0) -- (\a - 4 + \sep, 0.5*\h) -- (2 + \a + \sep, \h) node[anchor=south]{$i_{\xi_1}$} ;
	\draw (1.5 +\a+ \sep/2, 0) node {$\cdots$};
	\draw (1.5 +\a+ \sep/2, \h) node[ anchor = south] {$\cdots$};
	
	\path (1 + \a + 2 + 3*\sep, 0) -- (1 + 2 + 3*\sep, \h)  coordinate[pos=0.5] (A);
	\draw (1 + \a + 2 + 3*\sep, 0) -- (A) -- (1+\a + 2 + 3*\sep, \h) node[anchor=south]{$i_1^{-1}$};
       \path (2 + \a + \sep + 2 + 3*\sep, 0) -- (2+\sep + 2 + 3*\sep, \h) coordinate[pos=0.5] (A);
       \draw (2 + \a + \sep + 2 + 3*\sep, 0) -- (A) -- (2 + \a + \sep + 2 + 3*\sep, \h) node[anchor=south]{$i_{\xi_2}^{-1}$} ;
	\draw (1.5 +\a+ \sep/2 + 2 + 3*\sep, 0) node {$\cdots$};
	\draw (1.5 +\a+ \sep/2 + 2 + 3*\sep, \h) node[ anchor = south] {$\cdots$};
	
	\draw (5 + 4.5*\sep + \a, 0) -- (5 + 4.5*\sep, \h) node[anchor=south]{$i_{1}^{1-k}$};
	\draw (5.5 + 5*\sep + \a, 0)  node{$\cdots$};
	\draw (5.5 + 5*\sep, \h) node[anchor=south]{$\cdots$};
	\draw (6 + 5.5*\sep + \a, 0) -- (6 + 5.5*\sep, \h) node[anchor=south]{$i_{\xi_{k}}^{1-k}$};
	
	\draw[brown] (\a + 0.5, \h) rectangle (\a + 0.5 + 2 + \sep , \h + 1.4);
	\draw[brown] (\a + 2.2 + 3*\sep, \h) rectangle (\a + 2.6 + 2 + 4*\sep , \h + 1.4);
	\draw[brown] (\a + 5 + 4*\sep, \h) rectangle (\a + 5 + 2.7 + 5*\sep , \h + 1.4);
\end{braid}.
\end{equation}
We will perform a further induction to show that this diagram equals the result of the lemma if $i_1$ is not a residue of the $k$-th row, and otherwise it equals
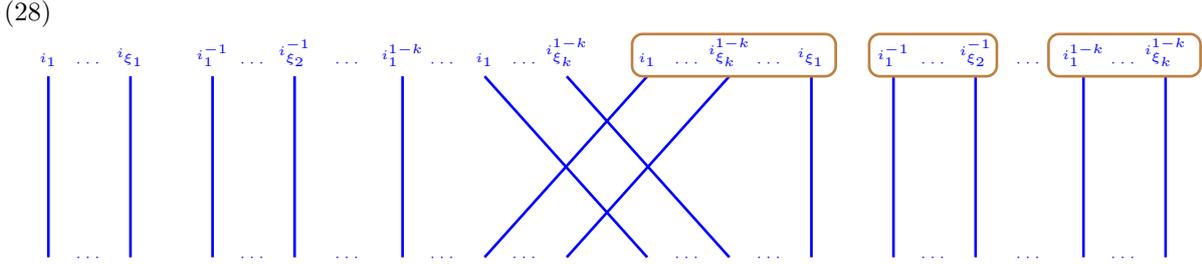
\begin{equation} \label{E:MainProof2}
 \begin{braid} 
	\def\h{6};
	\def\sep{1.7}
	\def\a{18+\sep};
	\def\dashpart{0.9}
	
	\draw (1,0) -- (1, \h)node[anchor=south]{$i_{1}$};
	\draw (1.5 + \sep/2, 0) node {$\cdots$};
	\draw (1.5 + \sep/2, \h) node[ anchor = south] {$\cdots$};
	\draw (2+\sep, 0) -- (2+\sep, \h) node[anchor=south]{$i_{\xi_1}$} ;
	
	\draw (3 + 2*\sep, 0) -- (3 + 2*\sep, \h) node[anchor=south]{$i_{1}^{-1}$};
	\draw (3.5 + 2.5*\sep, 0)  node{$\cdots$};
	\draw (3.5 + 2.5*\sep, \h) node[anchor=south]{$\cdots$};
	\draw (4 + 3*\sep, 0) -- (4 + 3*\sep, \h) node[anchor=south]{$i_{\xi_2}^ {-1}$};
	\draw (4.5 + 3.75*\sep, 0)  node{$\cdots$};
	\draw (4.5 + 3.75*\sep, \h) node[anchor=south]{$\cdots$};
	
	\draw (5 + 4.5*\sep, 0) -- (5 + 4.5*\sep, \h) node[anchor=south]{$i_{1}^{1-k}$};
	\draw (5.5 + 5*\sep, 0)  node{$\cdots$};
	\draw (5.5 + 5*\sep, \h) node[anchor=south]{$\cdots$};
	\draw (6 + 5.5*\sep, 0) -- (1+\a, \h) node[anchor=south]{$i_{1}$};
	\draw (6.5 + 6*\sep, 0)  node{$\cdots$};
	\draw (6.5 + 6*\sep, \h) node[anchor=south]{$\cdots$};
	\draw (7 + 6.5*\sep, 0) -- (2+\sep+ \a, \h) node[anchor=south]{$i_{\xi_k}^{1-k}$};
	
	\draw (1 + \a, 0) -- (6 + 5.5*\sep, \h) node[anchor=south]{$i_1$} ;
	\draw (1.5 +\a+ \sep/2, 0) node {$\cdots$};
	\draw (1.5 +\a+ \sep/2, \h) node[ anchor = south] {$\cdots$};
	\draw (2 + \a + \sep, 0) -- (7 + 6.5*\sep, \h) node[anchor=south]{$i_{\xi_k}^{1-k}$} ;
	\draw (2.5 +\a+ 1.5*\sep, 0) node {$\cdots$};
	\draw (2.5 +\a+ 1.5*\sep, \h) node[ anchor = south] {$\cdots$};
	\draw (3 + \a + 2*\sep, 0) -- (3+2*\sep+ \a, \h) node[anchor=south]{$i_{\xi_1}$} ;
	
	\draw (4 + 3*\sep + \a, 0) -- (4 + 3*\sep+ \a, \h) node[anchor=south]{$i_{1}^{-1}$};
	\draw (4.5 + 3.5*\sep + \a, 0)  node{$\cdots$};
	\draw (4.5 + 3.5*\sep + \a, \h) node[anchor=south]{$\cdots$};
	\draw (5 + 4*\sep + \a, 0) -- (5 + 4*\sep+ \a, \h) node[anchor=south]{$i_{\xi_2}^{-1}$};
	\draw (5.5 + 4.75*\sep + \a, 0)  node{$\cdots$};
	\draw (5.5 + 4.75*\sep + \a, \h) node[anchor=south]{$\cdots$};
	
	\draw (6 + 5.5*\sep + \a, 0) -- (6 + 5.5*\sep+ \a, \h) node[anchor=south]{$i_{1}^{1-k}$};
	\draw (6.5 + 6*\sep + \a, 0)  node{$\cdots$};
	\draw (6.5 + 6*\sep + \a, \h) node[anchor=south]{$\cdots$};
	\draw (7 + 6.5*\sep + \a, 0) -- (7 + 6.5*\sep+ \a, \h) node[anchor=south]{$i_{\xi_{k}}^{1-k}$};
	
	\draw[brown] (\a + 0.5, \h) rectangle (\a + 0.8 + 3 + 2*\sep , \h + 1.4);
	\draw[brown] (\a + 3.2 + 3*\sep, \h) rectangle (\a + 3.5 + 2.2 + 4*\sep , \h + 1.4);
	\draw[brown] (\a + 5.7 + 5*\sep, \h) rectangle (\a + 6 + 3 + 6*\sep , \h + 1.4);
\end{braid}.\end{equation}
For now, let us assume the claim. If $i_1$ is not a residue of the $k$-th row of $\xi$, we have proved the lemma. Otherwise, the $i_1$-string which reaches the $i_1$-node in the $k$-th row of $\xi$ in \cref{E:MainProof1} reaches the upper-leftmost $i_1$-node inside $[\xi]$ in \cref{E:MainProof2}, and applying a dot to this string yields the desired result after an application of \cref{L:MainOneRow}.

It is left to prove the claim about \cref{E:MainProof1}. We will perform an induction on the number of rows $k$. The case where $k = 1$ is clear. Going back to the diagram in \cref{E:MainProof1}, we focus our attention on the last two rows of $\xi$, where we have a subdiagram as follows.
$$\begin{braid}
    \def\h{7.8};
    \def\sep{2.21};
    \def\a{2.6+\sep};
    \def\dashpart{1.17};

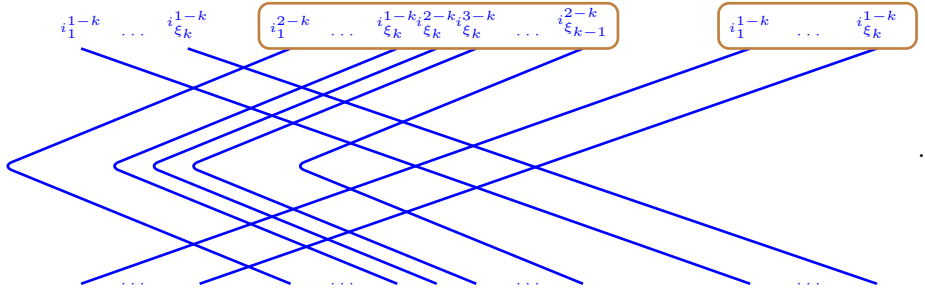
    \draw (0, 0) -- (7.8 + 3.9*\sep +1.3*\a, \h) node[anchor=south]{$i_{1}^{1-k}$};
    \draw (1.69 + \sep, 0) -- (9.1 + 5.2*\sep +1.3*\a, \h) node[anchor=south]{$i_{\xi_k}^{1-k}$};
    
    \draw (1.3 + 1.3*\a, 0) -- (-2.6, \h/2) -- (1.3 + 1.3*\a, \h) node[anchor=south]{$i_{1}^{2-k}$};
    \draw (2.6 + 1.3*\a + \sep, 0) -- (-1.3 + \sep, \h/2)-- (2.6 + 1.3*\a + \sep, \h) node[anchor=south]{$i_{\xi_k}^{1-k}$};
    \draw (3.9 + 1.3*\a+\sep, 0) -- (0+\sep, \h/2) -- (3.9 + 1.3*\a+\sep, \h) node[anchor=south]{$i_{\xi_k}^{2-k}$};
    \draw (5.2 + 1.3*\a + \sep, 0) -- (1.3+\sep, \h/2) -- (5.2 + 1.3*\a + \sep, \h) node[anchor=south]{$i_{\xi_k}^{3-k}$};
    \draw (6.5 + 1.3*\a + 2*\sep, 0) -- (2.6+2*\sep, \h/2) -- (6.5 + 1.3*\a + 2*\sep, \h) node[anchor=south]{$i_{\xi_{k-1}}^{2-k}$};
    
    \draw (0.65+\sep/2, \h) node[anchor=south]{$\cdots$};
    \draw (1.95 + 1.3*\a +0.5*\sep, \h) node[anchor=south]{$\cdots$};
    \draw (5.85 + 1.3*\a +1.5*\sep, \h) node[anchor=south]{$\cdots$};
    \draw (8.45 + 1.3*\a +4.5*\sep, \h) node[anchor=south]{$\cdots$};
    
    \draw (0.65+\sep/2, 0) node{$\cdots$};
    \draw (1.95 + 1.3*\a +0.5*\sep, 0) node{$\cdots$};
    \draw (5.85 + 1.3*\a +1.5*\sep, 0) node{$\cdots$};
    \draw (8.45 + 1.3*\a +4.5*\sep, 0) node{$\cdots$};
    
    \draw (7.8 + 3.9*\sep + 1.3*\a, 0) -- (0, \h) node[anchor=south]{$i_{1}^{1-k}$};
    \draw (9.1 + 5.2*\sep + 1.3*\a, 0) -- (1.3+\sep, \h) node[anchor=south]{$i_{\xi_k}^{1-k}$};
    
    \draw[brown] (1.3*\a + 0.26, \h) rectangle (7 + 1.3*\a + 2*\sep + 0.65 , \h + 1.5);
    \draw[brown] (6.76 + 3.9*\sep + 1.3*\a, \h) rectangle (10.4 + 5.2*\sep + 1.3*\a, \h+1.5);
\end{braid}.$$

Next we perform computations on this diagram. We slide the strings which reach row $k-1$ with residues $i_{\xi_k}^{2-k}, \dots, i_{\xi_{k-1}}^{2-k}$ to the right side of the diagram. Strings to the right of the $i_{\xi_k}^{2-k}$-string slide to the right without leaving any error terms. When sliding to the right, the $i_{\xi_k}^{2-k}$-string slides past a $(i_{\xi_k}^{1-k}, i_{\xi_k}^{1-k})$-crossing to its right, leaving two terms. First we draw the diagram after the $i_{\xi_k}^{2-k}$-string slides past this crossing.
$$\begin{braid}
    \def\h{7.8};
    \def\sep{2.21};
    \def\a{2.6+\sep};
    \def\dashpart{1.17};

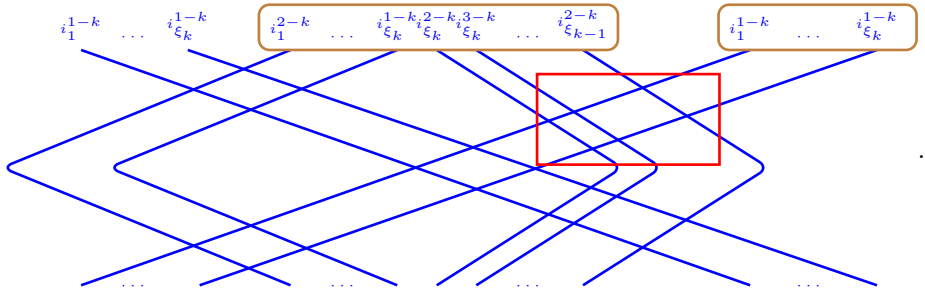
    \draw (0, 0) -- (7.8 + 3.9*\sep +1.3*\a, \h) node[anchor=south]{$i_{1}^{1-k}$};
    \draw (1.69 + \sep, 0) -- (9.1 + 5.2*\sep +1.3*\a, \h) node[anchor=south]{$i_{\xi_k}^{1-k}$};
    
    \draw (1.3 + 1.3*\a, 0) -- (-2.6, \h/2) -- (1.3 + 1.3*\a, \h) node[anchor=south]{$i_{1}^{2-k}$};
    \draw (2.6 + 1.3*\a + \sep, 0) -- (-1.3 + \sep, \h/2)-- (2.6 + 1.3*\a + \sep, \h) node[anchor=south]{$i_{\xi_k}^{1-k}$};

    \draw (3.9 + 1.3*\a+\sep, 0) -- (10+1.3*\a+\sep, \h/2) -- (3.9 + 1.3*\a+\sep, \h) node[anchor=south]{$i_{\xi_k}^{2-k}$};
    \draw (5.2 + 1.3*\a + \sep, 0) -- (10+1.3*\a +1.3+\sep, \h/2) -- (5.2 + 1.3*\a + \sep, \h) node[anchor=south]{$i_{\xi_k}^{3-k}$};
    \draw (6.5 + 1.3*\a + 2*\sep, 0) -- (10+1.3*\a + 2.6+2*\sep, \h/2) -- (6.5 + 1.3*\a + 2*\sep, \h) node[anchor=south]{$i_{\xi_{k-1}}^{2-k}$};
    
    \draw (0.65+\sep/2, \h) node[anchor=south]{$\cdots$};
    \draw (1.95 + 1.3*\a +0.5*\sep, \h) node[anchor=south]{$\cdots$};
    \draw (5.85 + 1.3*\a +1.5*\sep, \h) node[anchor=south]{$\cdots$};
    \draw (8.45 + 1.3*\a +4.5*\sep, \h) node[anchor=south]{$\cdots$};
    
    \draw (0.65+\sep/2, 0) node{$\cdots$};
    \draw (1.95 + 1.3*\a +0.5*\sep, 0) node{$\cdots$};
    \draw (5.85 + 1.3*\a +1.5*\sep, 0) node{$\cdots$};
    \draw (8.45 + 1.3*\a +4.5*\sep, 0) node{$\cdots$};
    
    \draw (7.8 + 3.9*\sep + 1.3*\a, 0) -- (0, \h) node[anchor=south]{$i_{1}^{1-k}$};
    \draw (9.1 + 5.2*\sep + 1.3*\a, 0) -- (1.3+\sep, \h) node[anchor=south]{$i_{\xi_k}^{1-k}$};
    
    \draw[red, sharp corners] (15, 4) -- (21,4) -- (21,7) -- (15, 7) -- cycle;
    \draw[brown] (1.3*\a + 0.26, \h) rectangle (7 + 1.3*\a + 2*\sep + 0.65 , \h + 1.5);
    \draw[brown] (6.76 + 3.9*\sep + 1.3*\a, \h) rectangle (10.4 + 5.2*\sep + 1.3*\a, \h+1.5);
\end{braid}.$$
The $\xi_{k-1} + 1$ strings inside the highlighted red rectangle are in the same disposition as \cref{E:Garnir}, so that this diagram equals zero. We are left only with the diagram after the $i_{\xi_k}^{2-k}$-string cuts the $(i_{\xi_k}^{1-k}, i_{\xi_k}^{1-k})$-crossing, which we draw next. 
$$\begin{braid}
    \def\h{7.8};
    \def\sep{2.21};
    \def\a{3.9+\sep};
    \def\dashpart{1.17};

    \draw (0, 0) -- (6.5 + 3.9*\sep +1.3*\a, \h) node[anchor=south]{$i_{1}^{1-k}$};
    \draw (1.3 + \sep, 0) -- (7.8 + 5.2*\sep +1.3*\a, \h) node[anchor=south]{$i_{\xi_k}^{-k}$};
    \draw[orange] (2.6 + \sep, 0) -- (10.3+\a, \h/2) -- (2.6  +\sep, \h) node[anchor=south]{$i_{\xi_k}^{1-k}$};
    
    \draw (1.3 + 1.3*\a, 0) -- (0, \h/2) -- (1.3 + 1.3*\a, \h) node[anchor=south]{$i_{1}^{2-k}$};
    \draw (2.6 + 1.3*\a + \sep, 0) -- (1.3 + \sep, \h/2) -- (2.6 + 1.3*\a + \sep, \h) node[anchor=south]{$i_{\xi_k}^{1-k}$};
    \draw (3.9 + 1.3*\a+\sep, 0) -- (16.9 + 1*\sep, \h/2) -- (3.9 + 1.3*\a+\sep, \h) node[anchor=south]{$i_{\xi_k}^{2-k}$};
    \draw (5.2 + 1.3*\a+2*\sep, 0) -- (18.2 + 2*\sep, \h/2) -- (5.2 + 1.3*\a+2*\sep, \h) node[anchor=south]{$i_{\xi_{k-1}}^{2-k}$};
    
    \draw (.65+.65*\sep, \h) node[anchor=south]{$\cdots$};
    \draw (1.95 + 1.3*\a +0.5*\sep, \h) node[anchor=south]{$\cdots$};
    \draw (3.9+1.3*\a+1.8*\sep, \h) node[anchor=south]{$\cdots$};
    \draw (7.2 + 4.5*\sep + 1.3*\a, \h) node[anchor=south]{$\cdots$};
    
    \draw (.65 + \sep/2, 0) node{$\cdots$};
    \draw (1.95 + 1.3*\a +0.5*\sep, 0) node{$\cdots$};
    \draw (3.9+1.3*\a+1.8*\sep, 0) node{$\cdots$};
    \draw (7.2 + 4.5*\sep + 1.3*\a, 0) node{$\cdots$};
    
    \draw (6.5 + 3.9*\sep + 1.3*\a, 0) -- (0, \h) node[anchor=south]{$i_{1}^{1-k}$};
    \draw (7.8 + 5.2*\sep + 1.3*\a, 0) --  (1.3 + \sep, \h) node[anchor=south]{$i_{\xi_k}^{-k}$};
    \draw (9.1 + 5.2*\sep + 1.3*\a, 0) -- (9.1 + 5.2*\sep + 1.3*\a, \h) node[anchor=south]{$i_{\xi_k}^{1-k}$};
    
    \draw[brown] (1.3*\a + 0.26, \h) rectangle (5.7 + 1.3*\a + 2*\sep + 0.65 , \h + 1.5);
    \draw[brown] (5.46 + 3.9*\sep + 1.3*\a, \h) rectangle (9.9 + 5.2*\sep + 1.3*\a, \h+1.5);
\end{braid}.$$
Next we slide the orange $i_{\xi_k}^{1-k}$-string reaching the left side of the diagram to the left. The summand where it slides past the $(i_{\xi_k}^{-k}, i_{\xi_k}^{-k})$-crossing is zero, as the $i_{\xi_k}^{1-k}$-string slides to the left past other crossings and reaches another $i_{\xi_k}^{1-k}$-string where we may apply \cref{E:PsiSquared}. Therefore, we consider only the summand where the $i_{\xi_k}^{1-k}$-string cuts the $(i_{\xi_k}^{-k}, i_{\xi_2}^{-k})$-crossing to its left. Similarly, the $(i_{\xi_k}^{-k})$-string must cut the $(i_{\xi_k}^{-k-1}, i_{\xi_k}^{-k-1})$-crossing to its left, and so on. After rearranging the strings, we reach the diagram

$$ \begin{braid}
    \def\h{7.2};
    \def\sep{1.7};
    \def\a{3.6+\sep};
    \def\dashpart{0.4};
    
    \draw (\dashpart + 6 + 3*\sep +\a, 0) -- (\dashpart + 6 + 3*\sep +\a, \h) node[anchor=south]{$i_{1}^{1-k}$};
    \draw (\dashpart + 7.2 + 4*\sep+\a, 0) -- (\dashpart + 7.2 + 4*\sep +\a, \h) node[anchor=south]{$i_{\xi_k}^{-k}$};
    \draw (\dashpart + 8.4 + 4*\sep + \a, 0) -- (\dashpart + 8.4 + 4*\sep + \a, \h) node[anchor=south]{$i_{\xi_k}^{1-k}$};
    
    \draw (1.2 + \a, 0) -- (0, \h/2) -- (1.2 + \a, \h) node[anchor=south]{$i_{1}^{2-k}$};
    \draw (2.4 + \a + \sep, 0) -- (1.2 + \sep, \h/2) -- (2.4 + \a + \sep, \h) node[anchor=south]{$i_{\xi_k}^{1-k}$};
    \draw (\dashpart +3.6 + \a+\sep, 0)  -- (\dashpart +3.6 + \a+\sep, \h) node[anchor=south]{$i_{\xi_k}^{2-k}$};
    \draw (\dashpart + 4.8 + \a+2*\sep, 0)  -- (\dashpart + 4.8 + \a+2*\sep, \h) node[anchor=south]{$i_{\xi_{k-1}}^{2-k}$};
    
    \draw (.6+.5*\sep, \h) node[anchor=south]{$\cdots$};
    \draw (1.8 + \a +0.5*\sep, \h) node[anchor=south]{$\cdots$};
    \draw (\dashpart + 4.2 + \a +1.5*\sep, \h) node[anchor=south]{$\cdots$};
    \draw (\dashpart + 6.6 + \a +3.5*\sep, \h) node[anchor=south]{$\cdots$};
    
    \draw (.6 + \sep/2, 0) node{$\cdots$};
    \draw (1.8 + \a +0.5*\sep, 0) node{$\cdots$};
    \draw (\dashpart + 4.2 + \a +1.5*\sep, 0) node{$\cdots$};
    \draw (\dashpart + 6.6 + \a +3.5*\sep, 0) node{$\cdots$};
    
    \draw (0, 0) -- (2.4 + \a- \sep, \h/2) -- (0, \h) node[anchor=south]{$i_{1}^{1-k}$};
    \draw (1.2 + \sep, 0) --  (3.6 + \a , \h/2) -- (1.2 + \sep, \h) node[anchor=south]{$i_{\xi_{k}}^{-k}$};
    \draw (2.4 + \sep, 0) -- (4.8+\a, \h/2) -- (2.4  +\sep, \h) node[anchor=south]{$i_{\xi_k}^{1-k}$};
    
    \draw[brown] (\a + 0.24, \h) rectangle (\dashpart + 5.4+ \a + 2*\sep + 0.6 , \h + 1.58);
    \draw[brown] (\dashpart + 5.04 + 3*\sep + \a, \h) rectangle (\dashpart + 9.6 + 4*\sep + \a, \h+1.58);
\end{braid}.$$
An easy induction shows that the above equals (up to sign)
\begin{equation} \label{E:MainProof3}
\begin{braid}
	\def\h{6};
	\def\sep{1.7}
	\def\a{10+\sep};
	\def\dashpart{0.9}
	
	\draw (2.5 + 2*\sep, 0) -- (2.5 + 2*\sep, \h) node[anchor=south]{$i_{1}^{1-k}$};
	\draw (4 + 2*\sep, 0) -- (1+\a, \h) node[anchor=south]{$i_{1}^{2-k}$};
	
	\draw (5 + 3*\sep, 0) -- (2 + \a + \sep, \h) node[anchor=south]{$i_{\xi_k}^{1-k}$};
	
	\draw (1 + \a, 0) -- (4 + 2*\sep, \h)  node[anchor=south]{$i_{2}^{1-k}$};
       \draw (2 + \a + \sep, 0) -- (5 + 3*\sep, \h) node[anchor=south]{$i_{\xi_k}^{1-k}$};
	\draw (3.5 + \a+\sep, 0) --(3.5 + \a+\sep, \h) node[anchor=south]{$i_{\xi_k}^{2-k}$};
       \draw (4.5 + \a + 2*\sep, 0) -- (4.5 + \a + 2*\sep, \h) node[anchor=south]{$i_{\xi_{k-1}}^{2-k}$};
       
       \draw (4.5+2.5*\sep, \h) node[anchor=south]{$\cdots$};
       \draw (1.5 + \a +0.5*\sep, \h) node[anchor=south]{$\cdots$};
       \draw (4 + \a +1.5*\sep, \h) node[anchor=south]{$\cdots$};
       \draw (6.3 + \a +3.5*\sep, \h) node[anchor=south]{$\cdots$};
       
       \draw (4.5+2.5*\sep, 0) node{$\cdots$};
       \draw (1.5 + \a +0.5*\sep, 0) node{$\cdots$};
       \draw (4 + \a +1.5*\sep, 0) node{$\cdots$};
       \draw (6.3 + \a +3.5*\sep, 0) node{$\cdots$};
	
	\draw (6 + 3*\sep + \a, 0) -- (6 + 3*\sep + \a, \h) node[anchor=south]{$i_{1}^{1-k}$};
	\draw (7 + 4*\sep + \a, 0) -- (7 + 4*\sep + \a, \h) node[anchor=south]{$i_{\xi_k}^{1-k}$};
	
	\draw[brown] (\a + 0.1, \h) rectangle (4.9 + \a + 2*\sep + 0.5 , \h + 1.58);
	\draw[brown] (4.9 + 3*\sep + \a, \h) rectangle (7.8 + 4*\sep + \a, \h+1.58);
	
\end{braid}.
\end{equation}
Going back to the diagram in \cref{E:MainProof1}, we may substitute our result back in. Then, the leftmost $i_{1}^{1-k}$-string in the diagram in \cref{E:MainProof3} can be pulled to the left through all the strings of different residues without leaving error terms. This string, together with the last $\xi_k$ strings, is now immobile and may be ignored. When we remove these strings from the diagram, we reach a new diagram which is identical to \cref{E:MainProof1} but has one less row. Here we invoke the induction hypothesis to prove the claim.
\end{proof}

Label the basic strings which have residue $i_1$ as $s_1, s_2, \dots, s_m$. Then $k[x_1, x_2, \dots, x_m]$ acts on $R_{n+\gamma}^\Lambda$, where $x_i \mapsto y_{s_i}$, and so it also acts on the Specht modules. In particular, symmetric polynomials on $x_1, \dots, x_m$ are in the center of $R_m$, in the same way that symmetric polynomials are in the center of nil-Hecke algebras. We define $\varepsilon_m^u \in k[x_1, x_2, \dots, x_m]$ to be the elementary symmetric polynomial of degree $u$. Recall that the ring $R_\gamma$ is naturally embedded in $R_{n+\gamma}$, with $R_\gamma$ acting on the last $\gamma$ strings via $R_{n} \otimes R_\gamma \hookrightarrow R_{n+\gamma}$. Since $[\lambda^*]$ is a straight shape, it makes sense to consider the partition $\lambda^* \vdash \gamma$, and recall that by convention, the nodes in the main diagonal of $[\lambda^*]$ have residue $i_1$. 

\begin{lem} \label{L:DotsHom}
Let $v \in \Std_\blam^e(\bnu)$ be a monomial with at most $u$ non-immobile basic strings of residue $i_1$. Further, let $S^{\bnu}_\gamma$ be the $R_\gamma$-module obtained by restricting the action of $R_{n+\gamma}$ on $S^{\bnu}$ to $R_\gamma$. Then, there is a homomorphism of $R_\gamma$-modules  from $S^{\lambda^*}$ to $S^{\bnu}_\gamma$. This homomorphism sends the Specht module generator of $S^{\lambda^*}$ (that is, $v_{\t^{\lambda^*}}$) to $v\varepsilon_m^u$.
\end{lem}

\begin{proof}
Notice first that $\varepsilon^{u}_m$ commutes with $R_n \otimes R_\gamma \subset R_{n+\gamma}$. It commutes with $R_n \subset R_{n+\gamma}$ because it is a symmetric polynomial, and it commutes with $R_\gamma \subset R_{n+\gamma}$ because it acts on the basic strings.

Next we see that \cref{E:ResidueRelation} trivially holds. We check the other relations together. Let $\Psi \in R_\gamma$ be such that $v_{\t^{\lambda^*}}\Psi = 0$ is a relation given by \cref{E:Specht} or \cref{E:Garnir} of $S^{\lambda^*}$. Note that $\Psi$ does not feature any $(i, i)$-crossings, as $\lambda^*$ is $e$-small. Let $\Psi' \in R_{n+\gamma}$ be the image of $\Psi$ under the embedding $R_\gamma \subset R_{n+\gamma}$. 

Next, we can slide $\Psi'$ past some $(i, i)$-crossings in $v\Psi'$. This gives error terms where different $(i, i)$-crossings (for some $i \in I$) between basic and extra strings are cut. After $\Psi'$ slides past every crossing, we may apply \cref{E:Specht} to obtain zero. Each of the error terms can be further decomposed into standard monomials. For any monomial $v'$ in this decomposition, the string $s$ reaching the top leftmost node $N \in [\lambda^*]$ in $v'$ must be basic. This is because, if an extra string reaches $N$, then standardness requires that every other node in $[\mu^*]$ be reached by an extra string. However, after cutting a crossing of a basic string with an extra string, the set of nodes $[\mu^*]$ is no longer accessible to the set of extra strings, according to \cref{L:Stubborn3Sets}.

Note that string $s$ is non-immobile in $v$, since at least one of its $(i, i)$-crossings must have been cut in order to change the node it reaches. Moreover, $s$ cannot intersect any of the extra strings in $v'$, as those strings are stubborn and they originally reached nodes equal to or to the right of $N$. 

Clearly $\varepsilon_m^u$ and $\Psi'$ commute, as they act on entirely different sets of strings. But $v\Psi'$ has at most $u-1$ non-immobile $i_1$-strings which cross the extra strings at least once, because $v$ has at most $u$ of these strings, including $s$, while string $s$ does not cross extra strings in any summand of $v\Psi'$. Finally, note that $\varepsilon^{u}_m$ can only cut crossings between basic and extra strings, as the symmetric polynomials slide past basic-basic crossings without leaving error terms. We deduce that $(v\Psi')\varepsilon^{u}_m = 0$, which concludes the proof.
\end{proof}

The homomorphism given above may be the zero homomorphism. Recall our definition of $\sigma_{\xi, \mu^*}$ after the statement of \cref{T:MainStraight}. Let $\sigma_{\xi, \rho}$ be a permutation defined analogously for general straight shapes $[\rho], [\xi]$ inside a multipartition.

\begin{lem} \label{L:MainSubshape}
Let $[\rho]$ and $[\xi]$ be $e$-small straight shapes such that $[\xi]$ is a subshape of $[\rho]$. Suppose $[\rho]$ has $\gamma$ nodes and $[\xi]$ has $\gamma' \leq \gamma$ nodes. Place the two Young diagrams in two components to form a bipartition $(\rho, \xi)$. Let $\t_{\xi, \rho}$ be the unique maximal $(\rho, \xi)$-tableau in dominance order such that $(\t_{\xi, \rho})_{\downarrow{\gamma'}} = (0, \t^\xi)$. Then, applying one dot to each $i_1$-string inside the first $\gamma'$ strings of $v_{\t_{\xi, \rho}}$ gives (up to sign) $v_{\t_{\xi, \rho} \sigma_{\xi,\rho}}$.
\end{lem}

\begin{proof}
We label by $[\xi_\rho]$ the unique set of nodes in $[\rho]$ which is congruent to $[\xi]$ and can be obtained from $[\rho]$ by successively removing removable nodes. Consider the diagram for $v_{\t_{\xi, \rho}}$, and let $s$ be a string of residue $i$ reaching a node $N$ which is in $[\rho] \setminus [\xi_\rho]$. We claim that $s$ is immobile. We argue by induction on the number of rows of $[\rho] \setminus [\xi_\rho]$. Suppose that $N$ is the lower leftmost node of $[\rho] \setminus [\xi_\rho]$. If $N$ does not have a node in $[\rho]$ to its left then the residue of $N$ does not appear in $[\xi]$, and therefore $s$ is immobile. On the other hand, if $N$ does have a node $N'$ to its left in $[\rho]$, then $N' \in [\xi_\rho]$, and there are no nodes with the same residue as $N'$ in $[\rho] \setminus [\xi_\rho]$ due to the choice of $N$. Let $s'$ be the string which reaches $N'$ in $v_{\t_{\xi, \rho}}$. Then $s'$ is the largest string of its residue in the diagram, and $s$ is the unique string of its residue such that $s > s'$. This shows that $s$ is immobile. By standardness, each node to the right of $N$ must be higher than $s$. But there is only a single string larger than $s$ of the corresponding residue, so that all strings reaching these nodes are immobile. Next, consider the shape $[\rho']$, obtained from $[\rho]$ by removing the row-strip of nodes that starts with $N$. By induction on the number of rows of $[\rho] \setminus [\xi_\rho]$, all strings reaching nodes in $[\rho'] \setminus [\xi_\rho]$ must be immobile.

We conclude that all strings reaching nodes inside $[\rho] \setminus [\xi_\rho]$ in the first component are immobile. These immobile strings can be ignored in this fully commutative diagram, according to \cref{L:ImmobileIgnore}, so the statement reduces to \cref{L:MainStraight}
\end{proof} 

The next definition will be relevant in the next few lemmas.
\begin{defn}
Let $v \in S^\bnu$ be a standard monomial, and consider a subset $\mathfrak{A} = \{s_1, \dots, s_k\}$ of the strings in the diagram. We say that the strings $\mathfrak{A}$ \emph{can be arranged} in a straight shape $[\xi]$ if the numbers $s_1, \dots, s_k$ can be arranged to form a standard $\xi$-tableau $\t_{[\xi]}^\mathfrak{A}$ so that the residue of the node filled by $s_j$ matches the residue of $s_j$ in $v$.
\end{defn}
Recall that $\bnu \vdash n+\gamma$, and by convention we refer to the first $n$ strings as basic strings, while the last $\gamma$ strings are said to be extra strings.
Notice that, if an extra string $s$ reaches some node $N$ in a diagram, then each node south-east of $N$ has to be reached by an extra string by standardness. In particular, there need to be enough extra strings of the appropriate residues to fill the removable shape south-west of $N$. We can use this fact to easily deduce the following. 
\begin{lem} \label{L:SubshapeCondition}
Let $v \in S^\bnu$ and suppose that the extra strings can be arranged in an $e$-small straight shape $[\xi]$ with residue $i_1$ on the main diagonal. Then, in any $v' \in \mathcal{V}_v$, the extra $i_1$-strings may only reach $i_1$-nodes of $\bnu$ which are inside a subshape of $[\xi]$.
\end{lem}

\begin{lem} \label{L:ImmobileResidue}
Let $v \in S^\bnu$ be a monomial such that the basic strings do not cross each other. Suppose the extra strings can be arranged in an $e$-small straight shape $[\xi]$. Let $s$ be a basic string in $v$. If $s$ reaches a node which is to the left or above a string of a residue not appearing in $[\xi]$, then $s$ is an immobile string.
\end{lem}

\begin{proof}
Let $s \in \{1, \dots, n\}$ be a basic string of residue $i$. Suppose first that $i$ is not a residue inside the shape $[\xi]$. String $s$ can only intersect extra strings, so it does not meet any other $i$-strings, which shows that $s$ must be immobile.

Suppose instead that string $s$ reaches a node $N$ such that the basic string $s+1$ reaching the node $N'$ to the immediate right of $N$ is immobile. Extra strings are all greater than $s+1$, so they cannot reach $N$ in the monomial for a standard tableau. On the other hand, the basic strings are co-stubborn, so only a string greater than or equal to $s$ can reach $N$ in any $v' \in \mathcal{V}$. But there are no strings between $s$ and $s+1$, so $s$ must also be immobile. 

Similarly, suppose that string $s$ reaches a node $N$ such that the basic string $s'$ reaching the node $N'$ immediately below $N$ is immobile. Then for the same reasoning as above, only a basic string which reached a node to the right of $N$ in $\row(N)$, or a node to the left of $N'$ in $\row(N')$, can reach $N$ after a cut. However, each string reaching a node to the left of $N'$ in $\row(N')$ is immobile by the previous paragraph. Suppose instead that there is a node to the right of $N$ in $\row(N)$ with the same residue as $N$. Then there must also be a node of a residue not appearing in $[\xi]$ between $N$ and this other node, and therefore $s$ must be immobile by the previous paragraph. The lemma follows by an easy induction.
\end{proof}

\begin{lem} \label{L:ProofMaxFlow}
Let $v \in S^\bnu$ be a standard monomial with associated tableau $\t \in \Std_{n}(\bnu)$, and suppose that the extra strings can be arranged in an $e$-small removable straight shape $[\xi]$. We label the residue of the main diagonal of $[\xi]$ as $i_1$. Let $[\rho] \subset [\bnu]$ be another removable $e$-small straight shape. Suppose that there are $r$  non-immobile basic strings of residue $i_1$ reaching $[\rho]$ and $k$ extra strings of residue $i_1$ reaching nodes smaller than all nodes in $[\rho]$. Then if we apply more than $\min(r, k)$ dots to different basic $i_1$-strings reaching shape $[\rho]$, we obtain zero.
\end{lem}

\begin{proof}
We begin by noting that extra strings of residue $i_1$ cannot reach a node outside $[\rho]$ which is smaller than some nodes in $[\rho]$ but greater than others: such a node would be strictly west of $[\rho]$, and would necessarily also be west or north of a node with residue not appearing in $[\xi]$, so that an extra string could never reach it.

In the following diagram, we only draw non-immobile basic strings of residue $i_1$ reaching $[\rho]$ and extra strings of residue $i_1$ reaching $[\rho]$ or smaller nodes.
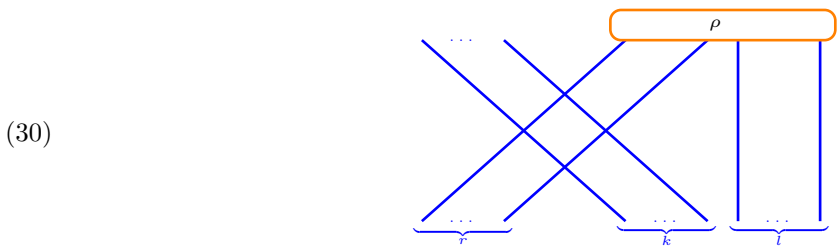
\begin{equation} \label{E:JKLDiagram}
\begin{braid}
	\def\h{6};
	\def\sep{1.7}
	\def\a{5+\sep};
	\def\dashpart{0.9}
	
	\draw (1,0) -- (1+\a, \h);
	\draw (1.5+\sep/2, 0) node{$\dots$};
	\draw (2 + \sep, 0) -- (2 + \a + \sep, \h);
	\draw (1+\a, 0) -- (1, \h);
	\draw (1.5 + \a + \sep/2, 0) node{$\dots$};
	\draw (1.5 + \sep/2, \h) node{$\dots$};
	\draw (2 + \sep + \a,0) -- (2 + \sep, \h);
	\draw (3 + \sep + \a, 0) -- (3 + \sep + \a, \h);
	\draw (3.5 + \a + 3*\sep/2, 0) node{$\dots$};
	\draw (4 + 2*\sep + \a, 0) -- (4 + 2*\sep + \a, \h);
	
	\draw[orange] (\a + 0.5, \h) rectangle (\a + 2*\sep +4.5 , \h + 1);
	\draw[black] (\a + \sep + 2.25, \h + 0.5) node[scale=1.3]{$\rho$};
	\draw(1.5 + \sep/2,-0.5)node{$\underbrace{\hspace*{13mm}}_{r}$};
	\draw(\a + 1.5 + \sep/2,-0.5)node{$\underbrace{\hspace*{13mm}}_{k}$};
	\draw(\a + 5.2 + \sep/2,-0.5)node{$\underbrace{\hspace*{13mm}}_{l}$};
\end{braid}
\end{equation}
Despite the way that we have drawn the diagram, it is possible for extra strings to intersect each other in $\Std_n(\bnu)$. We refer to $i_1$-strings reaching nodes to the left of $[\rho]$ as \emph{leftward}, and to other strings as \emph{rightward}. Since $[\rho]$ is $e$-small, rightward strings must reach nodes in increasing order in a standard monomial. Let $Y^j \in R_{n+\gamma}^\Lambda$ be an element applying $j$ dots to different strings out of the basic rightward strings in the diagram above. We suppose that elements $Y^j$ are constructed in a way that $Y^{j+1}$ always contains dots on the same strings as $Y^{j}$, plus an extra dot on another string to the right of all the previous ones.

Write the decomposition $ v_\t Y^j = \sum_i c_i^j v_{\t_i^j}$, where the $v_{\t_i^j}$ are basis elements and $c_i^j \neq 0$. Suppose $Y^j$ places dots on strings $s_{a_1}, \dots, s_{a_j}$. We proceed by induction on $j$, and prove the stronger result that for each $v_{\t_i^j}$, at least $j$ of the strings in ${s_1, \dots, s_{a_{j}-1}, s_{a_{j}}}$ are leftward. If $j = 0$ there is nothing to prove, so suppose that $j>0$. By induction, we may assume that $j-1$ out of the first ${s_1, \dots, s_{a_{j-1}-1}, s_{a_{j-1}}}$ strings are leftward for any $v_{\t_i^{j-1}}$. Since $v_{\t_i^{j-1}}$ is a standard monomial, each leftward string in it may only intersect smaller strings and other leftward strings, while rightward strings may only intersect leftward strings. This implies the following: let $N_{j-1}$ be the rightmost node reached by one of ${s_{1}, \dots, s_{a_{j-1}}}$. Then any nodes to the right of $N_{j-1}$ are not accessible to strings $s_{1}, \dots, s_{a_{j-1}}$. Therefore, the maximum flow from strings $s_1, \dots, s_{a_{j-1}}$ to shape $\rho$ is at most $a_{j-1} - (j-1)$.

For $v_{\t_i^j}$, suppose $s_{a_{j}}$ is rightward. Then standardness requires that all strings intersecting $s_{a_{j}}$ be leftward. After cutting a crossing of $s_{a_{j}}$, the string becomes leftward, and it may only cross other leftward strings and smaller strings. Now the node reached by $s_{a_j}$ in $v$, denoted by $N_j$, is not accessible in $v_{\t_i^j}$ either, so the maximum flow from strings  $s_{1}, \dots, s_{a_{j}}$ to $[\rho]$ is at most $a_{j} - j$. Suppose instead that $s_{a_{j}}$ is leftward. By induction, we also have that at least $j-1$ strings out of $s_1, \dots, s_{a_{j-1}}$ are leftward. Then there are at least $j$ leftward strings out of $s_1, \dots, s_{a_{j}}$, and the argument in the previous paragraph confirms that the maximum flow from strings $s_1, \dots, s_{a_{j}}$ to shape $[\rho]$ is at most $a_{j} - j$. 

Finally, for $j$ basic rightward strings to become leftward, there need to be at least $j$ basic rightward strings in $v$ to begin with, and at least $j$ extra leftward strings. Therefore, if we apply more than $\min(r, k)$ dots, there are no standard monomials that we could reach, so the result must be zero.
\end{proof}

Let $d_\rho$ be the number of $i_1$-nodes in $[\rho]$ which are inside a removable subshape of $[\xi]$. Then, according to \cref{L:SubshapeCondition} and using $l_\rho$ for the $l$ appearing in \cref{E:JKLDiagram}, we have $r \leq d_\rho - l_\rho$. We will need the following lemma about tableaux to prove that the homomorphisms in \cref{T:MainStraight} send the Specht generator $v_{\t^\blam}$ to a single standard monomial.

\begin{lem} \label{L:StandardBitableaux}
Let $[\xi]$ and $[\rho]$ be $e$-small straight shapes with nodes of residue $i_1$ inside their main diagonals such that $[\xi]$ is \emph{not} a subshape of $[\rho]$. Place the two diagrams in two components to form a bipartition $(\rho, \xi)$. Let $\t_{\xi, \rho}$ be defined as in \cref{L:MainSubshape}. Let $\mathcal{T}_{\t_{\xi,\rho}}$ be the set of standard tableaux of shape $(\rho, \xi)$ with the same residue sequence as $\t_{\xi, \rho}$. Let $H_1 \subset [\rho]$ be the largest hook in $[\rho]$. Then, for any $\t \in \Std(\rho,\xi)$ and any $N \in H_1$ such that $\t^{-1}(\t_{\xi,\rho}(N))$ is a node in the first component, $\t \notin \mathcal{T}_{\t_{\xi,\rho}}$.
\end{lem}

\begin{proof}
Let $\gamma$ be the number of nodes in $[\rho]$ and $\gamma'$ be the number of nodes in $[\xi]$. Write $H_1, H_2, \dots, H_h$ for the sets of nodes in the hooks of $[\xi]$ inside $(\rho, \xi)$ ordered so that $(j,j,2) \in H_j$ for each $j$. It is easy to show that, if $\t \in \mathcal{T}_{\t_{\xi, \rho}}$, then the nodes in $\t^{-1}(\t_{\xi,\rho}(H_j))$ must all lie in the same component for each $j$. In particular, it follows that all residues in $[\xi]$ are also residues in $[\rho]$, since the hook $H_1 \subseteq [\xi]$ must fit inside $[\rho]$. We assume, towards a contradiction, that there exists a standard $(\rho,\xi)$-tableau $\t$ such that $\t^{-1}(1)$ is a node in the first component of $(\rho, \xi)$ and $\t \in \mathcal{T}_{\t_{\xi,\rho}}$.

Let us pair up the nodes in the two components of $(\rho, \xi)$ so that $N = (a,b,1)$ is paired up with $N' = (a,b,2)$. Since $[\xi]$ is not a subshape of $[\rho]$, it follows that there is at least one node of $[\xi]$ which is not paired up. There may also be a subset of unpaired nodes in $[\rho]$, which we denote by $P$. We already argued in the proof of \cref{L:MainSubshape} that the strings reaching these nodes in a diagram are immobile, and the argument can be adapted to tableaux to show that for any tableau $\t \in \mathcal{T}_{\t_{\xi,\rho}}$ the numbers inside $P$ are the same as they are in $\t_{\xi,\rho}$. For this reason, we may ignore $P$, and assume in this way that each node in $[\rho]$ is paired up.

Let $R \coloneq (\rho, \xi) \backslash \t^{-1}(\{1,\dots,\gamma'\}) = (\rho, \xi) \backslash \t^{-1}(\t_{\xi,\rho}(\cup_{1 \leq j \leq h}H_j))$. Let $D_1, D_2, \dots, D_d$ be the sets of nodes in the diagonals of $[\rho]$ inside $(\rho, \xi)$, ordered from the top-right diagonal to the bottom-left diagonal, and let $D'_1, D'_2, \dots, D'_d$ denote the sets of nodes in the diagonals of $[\xi]$ ordered in the same way. Let $DR_1, DR_2, \dots, DR_d$ be the sets of nodes in $R$ with the same residues as $D_1, D_2, \dots, D_d$ respectively. Clearly, $D_j$ has the same cardinality as $DR_j$ for any $j$. Let $d^*$ be the minimal such that $|D'_{d^*}| > |D_{d^*}|$. The assumption that $[\xi]$ is not a subshape of $[\rho]$ guarantees the existence of $d^*$.

In tableau $\t$, the numbers $\gamma' + 1, \dots, \gamma' + \gamma$ fill the nodes in $R$ with the appropriate residues, but we will show that this cannot be done while maintaining standardness. To accomplish this, we first show that there is at most one way of filling the nodes in $DR_1, \dots, DR_{d^*- 1}$ with the numbers $\t_{\xi, \rho}(D_1), \dots, \t_{\xi, \rho}(D_{d^* - 1})$ respectively without breaking standardness. Moreover, the maximal number inside each diagonal $D_1, \dots, D_{d^* - 1}$ in tableau $\t_{\xi,\rho}$ must be placed inside the maximal rim hook of $[\xi]$ in tableau $\t$.

We work by induction. Since $D_1$ and $DR_1$ both contain a single node, it is obvious that $DR_1$ can be filled in a unique way by the number $\t_{\xi,\rho}(D_1)$. Let $N$ be the unique node inside $D_1$. By assumption, the numbers in $\t_{\xi,\rho}(H_1)$ must be inside the first component in $\t$, so that $N \notin R$. If the node $N$ is paired up with $N'$, then $DR_1 = \{N'\}$, and in particular $DR_1$ is part of the maximal rim hook of $[\xi]$. 

Next, let $j < d^*$ and suppose that $DR_{j-1}$ can be filled in a unique way with the numbers in $\t_{\xi, \rho}(D_{j-1})$. We must show that, in this case, there is at most one way of filling the nodes in $DR_{j}$ with the numbers in $\t_{\xi, \rho}(D_j)$. First, we assume that the diagonal $D_j$ is either to the north-east of the main diagonal of $[\rho]$, or it is itself the main diagonal, so that if $N_1 < N_2 < \dots < N_t$ are the nodes of $D_j$, then $\alpha \coloneq \t_{\xi, \rho}(N_1)$ is smaller than all the numbers in $\t_{\xi, \rho}(D_{j-1})$. Next, note that either $|DR_j| = |DR_{j-1}|$ or $|DR_j| = |DR_{j-1}| + 1$. First we assume that $|DR_j| = |DR_{j-1}| + 1$.
Let $DR_j^\uparrow \subseteq DR_j$ be the set of nodes in $DR_j$ such that the node above them is in $DR_{j-1}$. Then it can be checked that $|DR_j^\uparrow| = |DR_j| - 1$.  That is, each node in $DR_j$, except one node which we call $M$, has a node in $DR_{j-1}$ above it. Standardness then requires that $\alpha$ is placed in node $M$ in $\t$. Similarly, $\t_{\xi, \rho}(N_2)$ may only be placed in the node below the node housing the minimal number in $\t_{\xi, \rho}(D_{j-1})$, and so on for each number in $\t_{\xi,\rho}(D_j)$. In particular, the largest number in $\t_{\xi,\rho}(D_j)$ must be below the largest number in $\t_{\xi,\rho}(D_{j-1})$, which is in the maximal rim hook of $[\xi]$ by induction. Entirely similar arguments cover the case where $|DR_j| = |DR_{j-1}|$ (we would introduce the set $DR_{j}^{\rightarrow}$) and the cases where $D_j$ is to the south-west of the main diagonal of $[\rho]$.

Finally, we reach the diagonal $D_{d^*}$. First, we assume that $D^*$ is either to the north-east of the main diagonal of $[\rho]$, or it is itself the main diagonal. Consider the lowest node in $D'_{d^*}$, which we call $A$. The node above $A$ in $\t$ contains $m \coloneq \max \t_{\xi,\rho}(D_{d^*-1})$ and the condition $|D'_{d^*}| > |D_{d^*}|$ implies that $|D_{d^*}| =  |D_{d^* - 1}|$, from which it follows that all numbers in $\t_{\xi, \rho}(D_{d^*})$ are smaller than $m$. Therefore, there is no number in $\t_{\xi,\rho}(D_{d^*})$ which can be placed inside $A$ in tableau $\t$ to make it a standard tableau, which contradicts our assumption. A similar argument covers the case where $D_{d^*}$ is to the south-west of the main diagonal of $[\rho]$. Therefore, there does not exist a tableau $\t$ satisfying the properties in the statement.
\end{proof}

Now we are ready to prove our main theorem. Recall that the residue of the nodes of the main diagonal of $[\mu^*]$ is taken to be $i_1$. As in the previous section, we define $e_\mu^\lambda = \sum_{j \in I^n} e(j_1, \dots, j_n, i_1, \dots, i_\gamma)$ and, if there are $m$ basic strings with residue $i_1$, then $\varepsilon_m^d$ is the elementary symmetric polynomial of degree $d$ on $m$ variables, which acts on $S^\bnu$ by attaching dots to the basic $i_1$ strings. 

\begin{proof}[Proof of \cref{T:MainStraight}]
As in the previous section, we will prove the main theorem by applying \cref{L:EllersMurray}. Let $i_1$ be the residue of the main diagonal of $[\lambda^*]$. Label the basic $i_1$-strings as $s_{a_1}, \dots, s_{a_m}$, so that there is an action of $k[x_1, \dots, x_m]$ on $S^\blam$ by applying dots to $i_1$-strings. We argue that $L = \varepsilon_m^d e^\blam_\bmu$ satisfies all three conditions of \cref{L:EllersMurray}.  

First we will check conditions \textbf{(A)} and \textbf{(B)} of \cref{L:EllersMurray}. Let $v_\t \in S^\bnu$ be the basis element of $S^\bnu$ corresponding to tableau $\t \in \Std(\bnu)$. Since $L$ commutes with the natural inclusions of $R_n$ and $R_\gamma$ into $R_{n+\gamma}$, we may assume that $\t \in \Std_{n}(\bnu)$ and that the extra strings do not cross each other. Since $L$ includes the idempotent $e^\blam_\bmu$, we can further assume that $\t \in \Std_\blam^e(\bnu)$, as $v_\t L = 0$ otherwise.

By \cref{L:ImmobileResidue}, strings reaching nodes which are above or to the left of a node with a residue not appearing in $[\lambda^*]$ are immobile. We may ignore these strings according to \cref{L:ImmobileIgnore}. Moreover, any string reaching a node $N \prec [\mu^*]$ is immobile, so they may be safely ignored too. Finally, if a basic string reaching a node $N \succ [\lambda^*]$ in $v_\t$ is cut, the new diagram must be in $\hat{S}_\bnu$ because the extra strings form a strongly stubborn set due to \cref{L:StronglyStubbornSet}, taking $\mathcal{R}$ in that lemma to be the set of extra strings. We immobilize these strings by using \cref{L:ImmobileTrick}. This is equivalent to working in the quotient space $S^\bnu / \hat{S}_{\bmu}^\bnu$, which suits our purposes. Furthermore, according to \cref{L:ImmobileResidue}, any string reaching a node located to the north-west of a node with a residue not appearing in $[\mu^*]$ is immobile and may be ignored.

After removing all nodes to the north-west of any node having a residue outside the residues appearing in $[\mu^*]$, we are left with the maximal (under inclusion) removable $e$-small shapes with residues appearing inside $[\mu^*]$. According to the above, we only need to consider those shapes between $[\mu^*]$ and $[\lambda^*]$. We label these shapes as $[\mu^*] = [\rho^0] \prec \dots \prec [\rho^{c'}] = [\lambda^*]$.

Conditions \textbf{(A)} and \textbf{(B)} of \cref{L:EllersMurray} follow readily from \cref{L:ProofMaxFlow}, as we see now. Suppose $v_\t \in \check{S}_{\t^\blam}^\bnu$. Let $d_{\rho^j}$ be the number of $i_1$-nodes inside $[\rho^j]$ which are inside a subshape of $[\mu^*]$, as in \cref{L:ProofMaxFlow}. The elementary symmetric polynomial $\varepsilon_m^d$ applies $d$ dots to different strings, where $d = \sum_{j=1}^{c'} d_{\rho^j}$. We know from \cref{L:ProofMaxFlow} that we can apply at most $d_{\rho^j}$ dots to the strings reaching $[\rho^j]$ for $j \in \{1, \dots, c'\}$. Furthermore, the application of $d_{\lambda^*}$ dots to the basic $i_1$-strings reaching $[\lambda^*]$ ensures that all extra $i_1$-strings reach $[\lambda^*]$ in all the standard monomials that $v_\t L$ can decompose into, according to \cref{L:ProofMaxFlow}. (Recall that we have immobilized all strings after $[\lambda^*]$.) Standardness then requires that all extra strings reach $[\lambda^*]$. This shows that $v_\t L \in \hat{S}_{\t_\bmu}^\bnu$, which proves condition \textbf{(A)}.

Next suppose that $v_\t \in \check{S}_{\blam}^\bnu$. We write $l_{\rho^j}$ for the number of extra strings which reach $\rho^j$ in $v_\t$, as in diagram~\ref{E:JKLDiagram}. Then $l_{\mu^*} < \gamma$, so that $l_{\rho^j} \geq 1$ for some $j \geq 1$. Therefore, applying $d$ dots to the $i_1$-strings reaching the different $[\rho_j]$ gives zero by \cref{L:ProofMaxFlow}. (Recall that we are working modulo $\hat{S}_{\bmu}^\bnu$.) This proves condition \textbf{(B)}.

Finally, we prove condition \textbf{(C)}. We must show that $v_{\t_\blam^\bnu}L \notin \hat{S}^\bnu_\bmu$. It is not necessary to keep track of the signs, so we perform computations up to sign.

As we have seen, the only non-zero terms of $v_{\t_\blam^\bnu} \varepsilon_m^d$ are obtained from the summands of $\varepsilon_m^d$ which apply $d_{\rho^j}$ dots to the basic strings reaching $[\rho^j]$ for each $j \in \{0,1, \dots, c'\}$. More is true: applying a dot to an $i_1$-string reaching a node which is not inside a subshape of $[\mu^*]$ gives zero. That is, the only non-zero term results from applying a dot to each $i_1$-string which reaches a node inside a subshape of $[\mu^*]$. We prove this claim next by induction on $j$.

When $j = 0$, there is nothing to prove, as no basic strings reach $[\rho^0] = [\mu^*]$, so we let $j \in \{1, \dots, c'\}$. For $\rho^1, \dots, \rho^{j-1}$, the inductive procedure lets us assume that the only non-immobile strings reaching these shapes are the $d_{\rho^1} + \dots + d_{\rho^{j-1}}$ $i_1$-strings reaching nodes inside a subshape of $[\mu^*]$, and $L$ applies a dot to each of these strings in any non-zero summand of $v_{\t_\blam^\bnu} L$. 
The intersections between the basic strings reaching $[\rho^j]$ and the extra strings form a diagram which looks like  $v_{\t_{\rho^j, \mu^*}}$, as defined in \cref{L:MainSubshape}. An application of \cref{L:DotsHom} shows that we may apply Specht relations to the north-west of the intersections of the extra strings with the strings reaching $[\rho^j]$. We illustrate this procedure in the following sketch, where, for each set of strings reaching each of $[\rho^0], \dots, [\rho^{c}]$, we draw a single string representing the whole set for graphical clarity. Moreover, $L'$ represents the part of $L$ acting on the strings reaching shapes $[\rho^1]$ through $[\rho^{j-1}]$, and we have drawn a red segment at the point where we wish to apply Specht relations.
$$\begin{braid}
	\def\h{12};        
	\def\sep{2};      
	\def\a{5+\sep};    
	\def\dashpart{0.9}; 

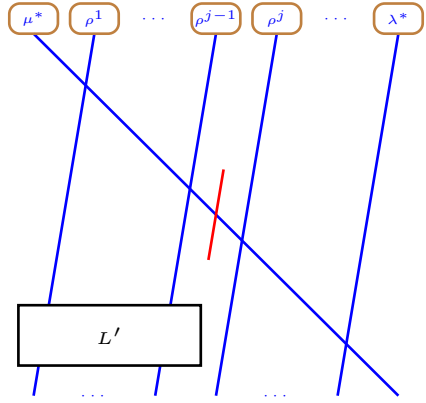
	\draw (-8,0) -- (-6, \h);
	\draw (-4,0) -- (-2, \h);
	\draw (-2,0) -- (0, \h);
	\draw (2, 0) -- (4, \h);
	\draw(4, 0) -- (-8, \h);
	\draw[red] (-2.25, 4.5) -- (-1.75, 7.5);

	\draw[brown] (-8.8, 12) rectangle (-7.2, 13);
	\draw (-8, 12.5) node{$\mu^*$};

	\draw[brown] (-6.8, 12) rectangle (-5.2, 13);
	\draw (-6, 12.5) node{$\rho^1$};

	\draw (-6, 0) node{$\cdots$};
	\draw (-4, 12.5) node{$\cdots$};

	\draw[brown] (-2.8, 12) rectangle (-1.2, 13);
	\draw (-2, 12.5) node{$\rho^{j-1}$};
	
	\draw[brown] (-0.8, 12) rectangle (.8, 13);
	\draw (0, 12.5) node{$\rho^{j}$};

	\draw (0, 0) node{$\cdots$};
	\draw (2, 12.5) node{$\cdots$};
	
	\draw[brown] (3.2, 12) rectangle (4.8, 13);
	\draw (4, 12.5) node{$\lambda^*$};

	\draw[black, sharp corners, fill=white] (-8.5, 1) rectangle (-2.5, 3);
	\draw[black] (-5.5, 2) node[scale=1.3]{$L'$};
\end{braid}$$

 We may now restrict our attention to the extra strings and the strings reaching $[\rho^j]$. More specifically, we look at the part of the diagram where these strings meet: above this subdiagram, we can apply the Specht relations for the bipartition $(\mu^*, \rho^j)$. For this bipartition, an application of \cref{L:StandardBitableaux} shows that the strings reaching nodes in $[\rho^j]$ which are not part of a subshape of $[\mu^*]$ are immobile. The proof by induction of the claim is complete.

We may now make use the of claim to discard the shapes $[\rho^0], \dots, [\rho^{c'}]$ and use instead $[\xi^0], \dots, [\xi^{c}]$, the maximal subshapes of $[\mu^*]$ defined at the start of the section. Note that $c \leq c'$.

Let us go back to the subdiagram of the intersections between the extra strings and the strings reaching $[\xi^j]$. An application of \cref{L:MainSubshape} shows that after applying dots to all $i_1$-strings reaching $[\xi^j]$ we obtain $v_{\t_{\xi^j,\mu^*}\sigma_{\mu^*, \xi^j}}$, up to sign. We define $\sigma'_{\mu^*, \xi^j}$ to be the permutation such that $v_{\t^{(\xi^j,\mu^*)} \sigma'_{\mu^*,\xi^j}} = v_{\t_{\xi^j,\mu^*} \sigma_{\mu^*,\xi^j}}$, where $\t^{(\xi^j,\mu^*)}$ is the initial tableau for the bipartition $(\xi^j,\mu^*)$.

After applying dots to each basic $i_1$-string reaching a node in some $[\xi^j]$ for $j \in \{1, \dots, c\}$ and applying relations for each subshape $[\xi^j]$ as in the above paragraph, we reach a new diagram $v'$. We define the $\bnu$-tableau $(\t_{\blam}^\bnu)'$ to be the tableau associated to $v'$. The tableau $(\t_{\blam}^\bnu)'$ is easily seen to be standard.  We picture $v'$ below.

$$\begin{braid}
	\def\h{12};
	\def\sep{2};
	\def\a{5+\sep};
	\def\dashpart{0.9};
	\draw (0,0) -- (0, 3);
	\draw (2, 0) node{$\dots$};
	\draw (4,0) -- (4,3);
	\draw[black, rounded corners = false] (-0.3,3) rectangle (4.3, 4.3);
	\draw[black] (2, 3.6) node[font = \fontsize{8}{10}\selectfont]{${\sigma'_{\mu^*, \lambda^*}}$};
	\draw (4, 4.3) -- (4, 12);
	\draw (2.5, 4.3) -- (2.5, 12);
	
	\draw( 0, 4.3) -- (0, 5.3);
	\draw(1.5, 4.3) -- (1.5, 5.3);
	\draw( 0, 6.6) -- (0, 12);
	\draw(1.5, 6.6) -- (1.5, 12);
	
	\draw (-2.5, 0) -- (-2.5, 5.3);
	\draw (-1, 0) -- (-1, 5.3);
	\draw(-2.5, 6.6) -- (-2.5, 7.6);
	\draw(-1, 6.6) -- (-1, 7.6);
	\draw(-2.5, 8.9) -- (-2.5, 12);
	\draw(-1, 8.9) -- (-1, 12);
	\draw[black, rounded corners = false] (-4.5,7.6) rectangle (-0.8, 8.9);
	\draw[black] (-2.6, 8.2) node[font = \fontsize{8}{10}\selectfont]{${\sigma'_{\mu^*, \xi^{c-2}}}$};

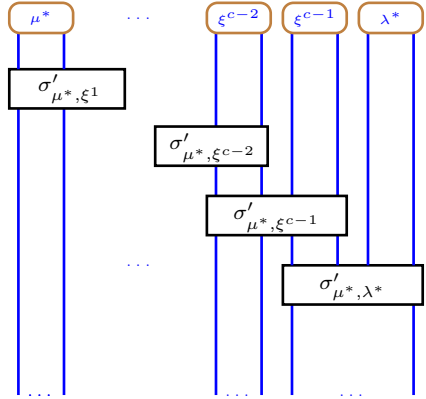
	\draw (-1.75, 0) node{$\dots$};
	\draw[black, rounded corners = false] (-2.8,5.3) rectangle (1.8, 6.6);
	\draw[black] (-0.5, 5.9) node[font = \fontsize{8}{10}\selectfont]{${\sigma'_{\mu^*, \xi^{c-1}}}$};
	
	\draw (-7.5, 0) -- (-7.5, 9.5);
	\draw (-8.25, 0) node{$\dots$};
	\draw (-9, 0) -- (-9, 9.5);
	\draw[black, rounded corners = false] (-9.3,9.5) rectangle (-5.5, 10.8);
	\draw[black] (-7.3, 10.15) node[font = \fontsize{8}{10}\selectfont]{${\sigma'_{\mu^*, \xi^{1}}}$};
	
	\draw (-7.5, 10.8) -- (-7.5, 12);
	\draw (-8.25, 0) node{$\dots$};
	\draw (-9, 10.8) -- (-9, 12);
	
	\draw (-5, 4.3) node{$\dots$};
	\draw (-5, 12.5) node{$\dots$};
	
	\draw[brown] (-9.3, 12) rectangle (-7.2, 13);
	\draw (-8.25, 12.5) node{$\mu^*$};
	\draw[brown] (-2.8, 12) rectangle (-0.7, 13);
	\draw (-1.75, 12.5) node{$\xi^{c-2}$};
	\draw[brown] (-0.3, 12) rectangle (1.8, 13);
	\draw (0.75, 12.5) node{$\xi^{c-1}$};
	\draw[brown] (2.2, 12) rectangle (4.3, 13);
	\draw (3.25, 12.5) node{$\lambda^*$};
\end{braid}$$

As in the previous section, we are using $\sigma'_{\mu^*, \xi^j}$ to refer to a permutation of the corresponding parts in $v'$ by abuse of notation. It is possible that the diagram that we have drawn is not a reduced expression, so that $v'$ may not be standard monomial. However, this problem turns out to be easily solved, as we see next.

After going up from the subdiagram $\sigma_{\mu^*, \xi^j}'$, a string either reaches $[\xi^j]$ without crossing any other strings, or it enters the subdiagram corresponding to $\sigma_{\mu^*, \xi^{j-1}}'$ next. If it reaches $[\xi^j]$, we say that $s$ is \emph{rightward} at $\sigma_{\mu^*, \xi^{j}}'$. Otherwise we say that it is \emph{leftward}. Now, strings \emph{with the same or adjacent residues} may only intersect in $\sigma_{\mu^*, \xi^{j}}'$ if one of them is leftward and the other is rightward. This is a direct consequence of the fact that strings with the same or adjacent residues do not intersect in reduced words for standard tableaux of $e$-small shape: indeed, leftward strings form a standard tableau of shape $[\mu^*]$, while rightward strings form a standard tableau of shape $[\xi^j]$. It follows that strings with the same or adjacent residues may not intersect twice in $v'$.

Moreover, a string $s$ in $v'$ can only be rightward at most once, and it does not intersect any strings after the subdiagram at which it is rightward. It follows that the pattern
$$
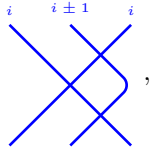
\begin{braid}
        	\draw (0,0) -- (4, 4) node[anchor=south]{$i$};
        	\draw (4,0) -- (0,4) node[anchor=south]{$i$};
        	\draw (2, 0) -- (4, 2) -- (2, 4)node[anchor=south]{$i \pm 1$};
\end{braid},
$$
where $i$ and $j$ are adjacent residues, cannot be found inside $v'$. Consequently, any double crossings 
$$
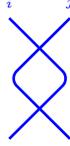
\begin{braid}
        	\draw (0,0) -- (2, 2) -- (0,4) node[anchor=south]{$i$};
        	\draw (2,0) -- (0, 2) -- (2,4) node[anchor=south]{$j$};
\end{braid}
$$
can be undone by simply pulling the $i$-string to the left past the $j$-string, leaving a new diagram $v_{(\t_{\blam}^\bnu)'} \neq 0$. It follows that $v_{\t_{\blam}^\bnu} L \notin \hat{S}^\bnu_\bmu$. This proves \textbf{(C)}, so we may apply \cref{L:EllersMurray}, and there exists a graded homomorphism $\theta$ of degree $(\deg \t_\blam^\bnu - \deg \t^\blam) - (\deg \t_\bmu^\bnu - \deg \t^\bmu) + \deg L$. Moreover, since $v_{\t_\blam^\bnu} L$ gives a single standard monomial $v''$ modulo $\hat{S}^\bnu_\bmu$, it follows that $\theta(v_{\t^\blam})$ is a single standard monomial whose diagram is given by the basic strings of $v''$.

To finish the proof we must check the given formula for the degree and establish its positivity. For $i \in I$, we define $d^\mu(i)$ to be the number of times that $i$ appears as a residue in shape $[\mu^*]$. From \cref{E:Degree}, we have $\deg \t_\blam^\bnu - \deg \t^\blam = \sum_{N^a} d^\mu(i_{N^a}) - \sum_{N^r} d^\mu(i_{N^r})$, where the first sum is over all addable nodes below $[\mu^*]$ and the second sum is over all removable nodes below $[\mu^*]$. $\deg \t_\bmu^\bnu - \deg \t^\bmu$ is computed similarly, so we have $(\deg \t_\blam^\bnu - \deg \t^\blam) - (\deg \t_\bmu^\bnu - \deg \t^\bmu) = a - b$. Clearly $\deg L = 2d$, so the degree of the homomorphism is $a - b + 2d$ as claimed.

Next we prove that $a - b + 2d \geq 0$. The term $b$, corresponding to removable nodes, is the only summand that contributes negatively. Moreover, any removable node is inside exactly one (inclusion-wise) maximal removable shape with residues appearing inside $[\mu^*]$, so that we may restrict our consideration to a single one of these shapes $[\xi]$. We label the contributions of $[\xi]$ to each summand as $a_\xi, -b_\xi$ and $2d_\xi$. We will show that $a_\xi - b_\xi + 2d_\xi \geq 0$ for any $[\xi]$, and finally that for $[\xi] = [\lambda^*]$, strict inequality holds. We write $H \subset [\xi]$ for the longest rim hook inside $[\xi]$. 

Assume first that $i_1$, the residue of the main diagonal of $[\mu^*]$, does not appear inside $H$. Without loss of generality, suppose that the residues in $H$ are residues in diagonals north-east of the $i_1$-diagonal in $[\mu^*]$. Then, for each removable node $N^r \in [\xi]$, there is an addable node $N^a$ in the next row such that $d^\mu(i_{N^r}) \leq d^\mu(i_{N^a})$. It follows that $a_\xi \geq b_\xi$, and in particular $a_\xi - b_\xi + 2d_\xi \geq 0$, as claimed.

Next, assume that  $i_1$ appears inside $H$. Label the removable nodes in $H$ on or to the north-east of the $i_1$-diagonal as $N_{1}^{r,NE}, N^{r,NE}_{2}, \dots, N^{r,NE}_{f_{NE}}$ in order down the rows. Label the addable node in the first row of $[\xi]$ as $N^{a,NE}_{1}$. If $N^{a,NE}_{1}$ is not an addable node in $\bnu$, then the maximality of $[\xi]$ implies that $d^\mu (i_{N^{a,NE}_{1}}) = 0$, so the existence of this addable node in the first row does not affect the integer $a_\xi$. Label the addable node between ${N_{1}^{r,NE}}$ and ${N^{r,NE}_{2}}$ as $N^{a,NE}_{2}$, and so on until $N^{a,NE}_{f_{NE}}$. 

Let $h^{NE} = \sum_{j=1}^{f_{NE}} d^\mu(i_{N^{r,NE}_{j}}) - d^\mu(i_{N^{a,NE}_{j}})$. Next, notice that $d^\mu(i_{N^{a,NE}_{{1}}}) \leq d^\mu(i_{N^{r,NE}_{{1}}}) \leq d^\mu(i_{N^{a,NE}_{{2}}}) \leq \dots \leq d^\mu(i_{N^{r,NE}_{{f_{NE}}}})$. This chain of inequalities implies that $d^\mu(i_{N^{r,NE}_{f_{NE}}}) \geq h^{NE}$, and more generally that 
$d^\mu(i_{N^{r,NE}_{f_{NE}}})  - d^\mu(i_{N^{a,NE}_{f_{NE}}}) + \dots + d^\mu(i_{N^{r,NE}_j}) \geq h^{NE}$ for any $0<j \leq f_{NE}$. After rearranging, we obtain
\[
d^\mu(i_{N^{r,NE}_j}) \geq h^{NE} - \left(d^\mu(i_{N^{r,NE}_{f_{NE}}})  - d^\mu(i_{N^{a,NE}_{f_{NE}}})\right) - \dots - \left(d^\mu(i_{N^{r,NE}_{j+1}})  - d^\mu(i_{N^{a,NE}_{j+1}})\right).
\]
Let us introduce the following notation: if we index the diagonals of $[\xi]$ from north-east to south-west, then given $N \in [\xi]$, we let $\diag(N)$ be the index of the diagonal of $[\xi]$ containing $N$. In particular, we have
$d^\mu(i_{N^{r,NE}_{j}}) - d^\mu(i_{N^{a,NE}_{j}}) \leq \diag(N^{r,NE}_{j}) - \diag(N^{a,NE}_{j})$,
as the numbers of nodes in adjacent diagonals in $[\mu^*]$ may differ by at most one. Further, it is easily checked that 
$\diag(N^{r,NE}_{j}) - \diag(N^{a,NE}_{j}) = \row(N^{r,NE}_{j}) - \row(N^{r,NE}_{j-1})$. Putting this together, for any $0 < j \leq f_{NE}$ we obtain
\[
d^\mu(i_{N^{r,NE}_j}) \geq h^{NE} - \left(\row(N^{r,NE}_{f_{NE}}) - \row(N^{r,NE}_{j}) \right).
\]
An entirely similar construction can be performed on the south-western part of $[\xi]$, swapping rows for columns, to obtain a similar inequality for $h^{SW}$, defined analogously to $h^{NE}$. We further define $h = \min(h^{SW}, h^{NE})$. Since $\row(N^{r,NE}_{f_{NE}}) \geq h^{NE}$ and $\col(N^{a,SW}_{f_{SW}}) \geq h^{SW}$, it follows that $[\xi]$ has at least $h$ nodes on its $i_1$-diagonal.

Consider the removable straight shape $[\xi'] \subset [\xi]$ having $h$ nodes of residue $i_1$ in its main diagonal. Then the removable nodes to the north-east of the $i_1$-diagonal of $[\xi']$ satisfy
\[
d^{\xi'}(i_{N_j^{r,NE}}) \leq h - (\row(N^{r,NE}_{f_{NE}}) - \row(N^{r,NE}_{j})) \leq h^{NE} - (\row(N^{r,NE}_{f_{NE}}) - \row(N^{r,NE}_{j})) \leq d^\mu(i_{N_j^{r,NE}}).
\]
Analogous inequalities apply for residues to the south-west of the $i_1$-diagonal in $[\xi']$. We conclude that $[\xi'] \subset [\xi]$ has at most the same number of nodes on each diagonal as $[\mu^*]$ (it is only necessary to check this property at the removable nodes of $[\xi']$). It follows that $[\xi']$ is a rank $h$ subshape of $[\mu^*]$, and consequently  the maximal subshape of $[\mu^*]$ inside $[\xi]$ has rank
\[d_\xi \geq h.\]
Without loss of generality, we will assume that $h = h^{NE} < h^{SW}$. 

First we suppose that the $i_1$-diagonal of $[\xi]$ has a removable node. Then this node is $N^{r,NE}_{f_{NE}} = N^{r,SW}_{f_{SW}}$. We wish to compute $a_\xi - b_\xi$. On the one hand, the addable and removable nodes to the north-east of the $i_1$-diagonal of $[\xi']$ up to node $N^{r,NE}_{f_{NE}}$ give a contribution of $-h^{NE}$. On the other hand, each node $N^{r,SW}_{j}$ can be paired up with the node $N^{a,SW}_{j + 1}$, so that $d^{\mu}(i_{N^{r,SW}_{j}}) \geq d^{\mu}(i_{N^{a,SW}_{j + 1}})$. Putting these together, we find that $a_\xi - b_\xi + d_\xi \geq 0$, and therefore $a_\xi - b_\xi + 2d_\xi \geq 0$, as we wished to show.

Suppose instead that the $i_1$-diagonal of $[\xi]$ does not have a removable node, so that there is an addable node between  $N^{r,NE}_{f_{NE}}$ and $N^{r,SW}_{f_{SW}}$, which we call $N^{a,M}$. Suppose first that $N^{a,M}$ is on the $i_1$-diagonal or to the south-west of it. Then a pairing similar to the last paragraph easily yields $a_\xi - b_\xi + 2d_\xi > 0$. Suppose instead that $N^{a,M}$ is to the north-east of the $i_1$-diagonal. Define $A = \diag(i_1) - \diag(N^{a,M})$, where $\diag(i_1)$ is the index of the $i_1$-diagonal. Let $g = d^\mu(i_{N^{r,SW}_{f_{SW}}}) - d^\mu(i_{N^{a,M}})$. We claim that 
\[
d_\xi \geq \min(h^{SW}, g).
\]
It is clear that $g \leq A$, as the number of nodes in each diagonal of $[\mu^*]$ is weakly increasing up to $\diag(i_1)$, and then weakly decreasing from that point, and the number of nodes in adjacent diagonals may differ by at most one. Moreover, the column containing the $i_1$-node in $H$ has at least $A$ nodes above the $i_1$-node. Therefore, to check whether the removable shape $[\xi'] \subset [\xi]$ with $g' \leq g$ nodes of residue $i_1$ in the main diagonal is a subshape of $[\mu^*]$, it is enough to consider only the south-western part of $[\xi]$. Since $\col(N^{r,SW}_{f_{SW}}) \geq h^{SW}$, we see that
\[d_\xi \geq \min(h^{SW}, g).\] 
With these bounds on $d_\xi$, we next prove our claims.

Suppose first that $h^{SW} \geq g$. From the inequalities above, we know that $d_\xi \geq h^{NE}$ and $d_\xi \geq g$. The addable and removable nodes up to $N^{r,NE}_{f_NE}$ contribute by $-h_{NE}$ to $a_\xi - b_\xi$. Next, pair up every removable node in the south-western part of the diagram with the next addable node, up to and including pairing $N^{r,SW}_{f_{SW}}$ with $N^{a,M}$. Each pair contributes a nonnegative amount to the degree, except for the last pairing of $N^{r,SW}_{f_{SW}}$ with $N^{a,M}$, which contributes $-g$. But then the total degree contributed by $[\xi]$ is greater than or equal to $2d_\xi - g - h^{NE} \geq 0$. Suppose instead that $g \geq h^{SW}$. From the inequalities above, we know that $d_\xi \geq h^{SW}$ and $d_\xi \geq h^{NE}$. Then, pair up every removable node in the north-eastern part of the diagram with the previous addable node, up to and including pairing $N^{r,NE}_{f_{NE}}$ with $N^{a,NE}_{f_{NE}}$. These nodes contribute $-h^{NE}$ to the degree. Do the same in the south-western part of $[\xi]$, and the nodes will contribute $-h^{SW}$ to the degree. The total degree contributed by $[\xi]$ is $2d_\xi + d^\mu(i_{N^{a, M}}) - h^{NE} - h^{SW} \geq 0$. This confirms that the degree of the given homomorphism must be greater than or equal to zero.

Finally we show that the inequality is strict, that is, $a - b + 2d > 0$. In fact, we have $a - b + 2d \geq \Rank(\mu^*)$. Notice first that one of the maximal $e$-small shapes up to $[\lambda^*]$ with residues appearing inside $[\mu^*]$ is $[\lambda^*]$ itself. But $[\lambda^*]$ contributes to the degree of the homomorphism by $\Rank(\lambda^*)$. To see this, first note that $d_{\lambda^*} = \Rank(\lambda^*)$. Moreover, $b_{\lambda^*} - a_{\lambda^*}$ is the sum of the number of nodes in diagonals ending at a removable node minus the sum of the number of nodes in diagonals ending at an addable node. By a theorem of Nazarov and Tarasov~\cite[Theorem 1.4]{NT98}, $b_{\lambda^*} - a_{\lambda^*} = \Rank(\lambda^*)$. Therefore, $a_{\lambda^*} - b_{\lambda^*} + 2d_{\lambda^*}$ also equals $\Rank(\lambda^*) = \Rank(\mu^*)$, which is a positive integer. This completes the proof.
\end{proof}
The author expects that \cref{T:MainStraight} can be further generalized.
\begin{conj} \label{C:Skew}
\cref{T:MainStraight} holds more generally, with $[\mu^*]$ and $[\lambda^*]$ being $e$-small \emph{skew} shapes, as opposed to straight shapes.
\end{conj}
For straight shapes, the rank of $\lambda^*$ played an important role in the description of the degree of the homomorphisms. For skew shapes, the generalized Durfee rank described in~\cite{NT98} would play the analogous role. Applying the techniques in this paper to skew shapes runs into the problem that it becomes necessary to apply dots to strings of different residues, which then makes proving conditions \textbf{(A)} and \textbf{(B)} of \cref{L:EllersMurray} much harder. However, it is possible to prove a skew-shape version of \cref{L:MainStraight} and use it to reprove Witty's theorem~\cite[Theorem 3.14]{Witty}. It is possible to prove \cref{C:Skew} in the special case of removable skew hooks. Moreover, it is possible to use an inductive argument to generalize \cite[Theorem 3.14]{Witty} to arbitrary level, as will appear in the author's doctoral thesis. However, it is not easy to see how one may dispose of the condition $e = \infty$. These will appear in the doctoral thesis of the author.

\begin{rem}
Robert Muth~\cite{muthskew} has introduced graded skew Specht modules with a definition in terms of a cyclic generator and relations, similar to the definition of Specht modules from \cite{kmr} that we describe in \cref{S:Background}. Our results apply in this more general picture without any modifications, so that $\blam$ and $\bmu$ may be skew partitions. Note that the shape $[\mu^*]$ must be straight, as we have not been able to prove our main result for a skew $[\mu^*]$ yet.
\end{rem}

\begin{rem}
Quiver Hecke algebras have an automorphism $\sigma$ given by reflecting diagrams across a vertical line, and further multiplying by $(-1)^s$, where $s$ is the number of crossings between two strings of the same residue. This automorphism is already described in \cite{kl09}. Quiver Hecke algebras of type A have another automorphism $\sgn$ described in \cite{kmr}. It is easy to check that twisting with $\sgn \circ \sigma$ gives an isomorphism between $S^\bnu$ and $S^{\bnu^{\circlearrowleft}}$ as $R_n$-modules, where $\bnu^{\circlearrowleft}$ is the skew multipartition resulting from rotating each component of $\bnu$ by $180^{\circ}$ and arranging them in opposite order compared to $\bnu$, with the residues of nodes changing as $i(N^{\circlearrowleft}) = -i(N)$. This twist identifies $\Hom_{R_n}(S^\blam, S^\bmu)$ with $\Hom_{R_n}(S^{\blam^{\circlearrowleft}}, S^{\bmu^{\circlearrowleft}})$. In particular, applying this twist to \cref{T:MainStraight} yields some new homomorphisms between skew Specht modules, where $\blam$ and $\bnu$ differ by a north-west-removable $e$-small shape $[\mu^*]$ such that $[\mu^*]^\circlearrowleft$ is a straight shape. 
\end{rem}

\bibliographystyle{amsalpha} 

\addtocontents{toc}{\vspace{2em}} 
\bibliography{Preamble/master}

\end{document}